\documentclass[12pt]{amsart}
\usepackage{amssymb}

\def\Z{{\mathbb Z}}\def\T{{\mathbb T}}\def\R{{\mathbb R}}\def\C{{\mathbb C}}

\def\CC{{\mathcal C}}\def\LL{{\mathcal L}}\def\RR{{\mathcal R}}

\def\TT{{\mathcal T}}\def\AA{{\mathcal A}}\def\EE{{\mathcal E}}
\def\FF{{\mathcal F}}\def\NN{{\mathcal N}}\def\DD{{\mathcal D}}
\def\HH{{\mathcal H}}\def\II{{\mathcal I}}\def\OO{{\mathcal O}}
\def\MM{{\mathcal M}}

\def\PP{{\mathcal P}}

\def\f{\varphi}\def\o{\omega}\def\z{\zeta}
\def\g{\gamma}\def\ga{\gamma}
\def\r{\rho}\def\s{\sigma}\def\k{\kappa}
\def\al{\alpha}\def\be{\beta}\def\d{\delta}\def\ep{\varepsilon}
\def\be{\beta}\def\t{\tau}\def\la{\lambda}\def\m{\mu}

\def\L{\Lambda}\def\D{\Delta}\def\G{\Gamma}\def\O{\Omega}

\def\cte{\mathrm{cte.}}
\def\Cte{\mathrm{Cte.}}

\def\expo{\mathrm{exp}}
\def\Leb{\mathrm{Leb}}
\def\Lip{\mathrm{Lip}}
\def\dist{\mathrm{dist}}

\def\oli{\overline}

\def\b#1{\lbrace#1\rbrace}
\def\a#1{\left|#1\right|}
\def\aa#1{\left\Vert#1\right\Vert}
\def\sc#1{<\!\!#1\!\!>}
\def\l#1{<\!#1\!>}

\def\lan{\langle}
\def\ran{\rangle}
\def\i{\infty}
\def\p{\partial}

\def\sbs{\subset}\def\sps{\supset}
\def\sm{\setminus}
\def\lsim{\lesssim}\def\gsim{\gtrsim}

\theoremstyle{plain}
\newtheorem{Thm}{Theorem}[section]
\newtheorem{Prop}[Thm]{Proposition}
\newtheorem{Lem}[Thm]{Lemma}

\newtheorem{Cor}[Thm]{Corollary}

\newtheorem{Main}{Theorem}

\theoremstyle{remark}
\newtheorem*{Rem}{Remark}
\newtheorem*{Not}{Notation}

\begin{document}

\title[KAM for NLS]
{KAM for the Non-Linear Schr\"odinger Equation}
\author{L. H. Eliasson}
\address{University of Paris 7,
Department of Mathematics,
Case 7052, 2 place Jussieu, Paris, France}
\email{hakane@math.jussieu.se}
\author{S. B. Kuksin}
\address{Heriot-Watt University,
 Department of Mathematics, Edinburgh}
\email{Kuksin@math.hw.ac.uk}

\date{\today}

%\thanks{This work was partially supported by...}

\begin{abstract}
We consider the $d$-dimensional nonlinear Schr\"o\-dinger
equation under periodic boundary conditions:
$$
  -i\dot u=-\Delta u+V(x)*u+\ep \frac{\p F}{\p \bar u}(x,u,\bar u),
  \quad u=u(t,x),\;x\in\T^d
$$
where $V(x)=\sum \hat V(a)e^{i\sc{a,x}}$ is an analytic function
with $\hat V$ real, and $F$ is a real analytic function in $\Re u$,
$\Im u$ and $x$. (This equation is a popular model for the `real'
NLS equation, where instead of the convolution term $V*u$ we have
the  potential term $Vu$.) For $\ep=0$ the equation is linear and
has time--quasi-periodic solutions $u$,
$$
u(t,x)=\sum_{a\in \AA}\hat u(a)e^{i(|a|^2+\hat V(a))t}e^{i\sc{a,x}}
\quad (|\hat u(a)|>0),
$$
where $\AA$ is any finite subset of $\Z^d$.
We shall treat $\omega_a=|a|^2+\hat V(a)$, $a\in\AA$, as free parameters
in some domain $U\subset\R^{\AA}$.

This is a Hamiltonian system in infinite degrees of freedom, degenerate
but with external parameters, and we shall describe
a KAM-theory which, under general conditions, will have the following consequence:
\smallskip

{\it If $|\ep|$ is sufficiently small, then there is a large subset
$U'$ of $U$ such that for all $\omega\in U'$ the solution $u$
persists as a time--quasi-periodic solution which has all Lyapounov
exponents equal to zero and whose linearized  equation is reducible
to constant coefficients.}

\end{abstract}

\date{\today}

\maketitle

\tableofcontents

\section{Introduction}
\label{s1}

We consider the $d$-dimensional nonlinear Schr\"odinger
equation
$$
  -i\dot u=-\Delta u+V(x)*u+\ep \frac{\p F}{\p
  \bar u}(x,u,\bar u),\quad u=u(t,x)\quad (*)
$$
under the periodic boundary condition $x\in\T^d$. The convolution
potential $V:\T^d\to\C$ must have real Fourier coefficients $\hat
V(a),\ a\in\Z^d$, and we shall suppose it is analytic. $F$ is an
analytic function in $\Re u$, $\Im u$ and $x$.

{\it The non-linear Schr\"odinger as an $\infty$-dimensional Hamiltonian system.}
If we write
$$\left\{\begin{array}{l}
u(x)=\sum_{a\in\Z^d}u_ae^{i<a,x>}\\
\oli{u(x)}=\sum_{a\in\Z^d}v_{a}e^{i<-a,x>}\quad(v_a=\bar u_a)
\end{array}\right.$$
and let
$$\zeta_a\left(\begin{array}{c}
\xi_a\\
\eta_a
\end{array}\right)
\left(\begin{array}{c}
\frac{1}{\sqrt{2}}(u_a+v_a)\\
\frac{-i}{\sqrt{2}}(u_a-v_a)
\end{array}\right),
$$
then, in the symplectic space
$$
\b{(\xi_a,\eta_a):a\in\Z^d}=\C^{\Z^d}\times \C^{\Z^d},\quad
\sum_{a\in\Z^d}d\xi_a\wedge d\eta_a,
$$
the equation becomes a real Hamiltonian system with an integrable
part
$$
\frac12\sum_{a\in\Z^d}(|a|^2+\hat V(a))(\xi_a^2+\eta_a^2)$$
plus a perturbation.

Let $\AA$ be a finite subset of $\Z^d$ and fix
$$0< p_a,\quad a\in\AA$$
The $(\#\AA)$-dimensional torus
$$\begin{array}{ll}
\frac12(\xi_a^2+\eta_a^2)=p_a& a\in\AA\\
\xi_a=\eta_a=0 & a\in\LL=\Z^d\sm\AA,
\end{array}$$
is invariant for the Hamiltonian flow when $\ep=0$.
Near this torus we
introduce action-angle variables $(\varphi_a,r_a)$, $a\in\AA$,
$$\begin{array}{l}
\xi_a=\sqrt{2(p_a+r_a)}\cos(\varphi_a)\\
\eta_a=\sqrt{2(p_a+r_a)}\sin(\varphi_a).
\end{array}$$
The integrable Hamiltonian now becomes (modulo a constant)
$$h=\sum_{a\in\AA} \omega_a r_a +
\frac12\!\sum_{a\in\LL}\!\Omega_a(\xi_a^2+\eta_a^2),$$
where
$$\omega_a=|a|^2+\hat V(a),\quad a\in\AA,$$
are the {\it basic frequencies}, and
$$\Omega_a=|a|^2+\hat V(a),\quad a\in\LL,$$
are the {\it normal frequencies} (of the invariant torus). The
perturbation $\ep f(\xi,\eta,\varphi,r)$ will be a function of all
variables (under the assumption, of course, that the torus lies in
the domain of $F$).

This is a standard form for the perturbation theory of
lower-dimensio\-nal (isotropic) tori with one exception: it is
strongly degenerate. We therefore need external parameters to
control the basic frequencies and the simplest choice is to let the
basic frequencies (i.e. the potential itself) be our free
parameters.

The parameters will belong to a set
$$U \subset \{ \omega \in \R^{\AA} : |\omega|\le C \}\,.$$

The normal frequencies will be assumed to verify
$$
\begin{array}{ll}
|\Omega_a|\ge C'>0 &\quad\forall\,a\in\LL\,,\\
|\Omega_a+\Omega_b| \ge C' &\quad\forall\,a,b\in\LL\,,\\
|\Omega_a-\Omega_b|\ge C'  &\quad\forall\,a,b\in\LL,|a|\ne|b|.
\end{array}
$$
This will be fulfilled, for example,  if $\AA$ is sufficiently large, or
if $V$ s small and $\AA\ni0$.

We define the complex domain
$$
\OO^{0}(\s,\r,\mu)\left\{\begin{array}{l}
\aa{\zeta}_0\sqrt{\sum_{a\in\LL}(|\xi_a|^2+|\eta_a|^2)\lan a\ran^{2m_*}}<\s\\
|\Im \f|<\r\\
|r|<\mu,
\end{array}\right.
$$
$\lan a\ran=\max(|a|,1)$. We assume $m_*>\frac {d}2$ because in this
space $h+\ep f$ is analytic and the Hamiltonian equations have a
well-defined local flow.

By   $\sc{\ ,\ }$ we denote the usual paring
$$\sc{\z ,\z' }=\sum\xi_a\xi_a'+\eta_a\eta_a'.$$

\begin{Main}\label{tA} Under the above assumptions, for
  $\ep$ sufficiently small there exist a subset
  $U'\subset U$, which is large in the sense that
$$
\Leb\,(U\setminus U')\le \cte\ep^{exp}\,,
$$
and for each $\omega\in U'$, a real analytic symplectic
diffeomorphism $\Phi$
$$
\OO^{0}(\frac{\s}2,\frac{\r}2,\frac{\m}2)\to
\OO^{0}(\s,\r,\m)
$$
and a vector $\o'$
such that $(h_{\o'}+\ep f)\circ\Phi$ equals (modulo a constant)
$$\sc{\o,r}+\frac12\!\!\sc{\xi,Q_1\xi}+\sc{\xi,Q_2\eta}+
\frac12\!\!\sc{\eta,Q_1\eta}+\ep f'\,,
$$
where
$$f'\in\OO(\a{r}^2,\a{r}\aa{\z}_0,\aa{\zeta}_0^3)$$
and $Q=Q_1+iQ_2$ is a Hermitian and block-diagonal matrix
with finite-dimensional blocks.

Moreover
$\Phi =(\Phi_\z,\Phi_\f,\Phi_r)$
verifies, for all $(\z,\f,r)\in
\OO^{0}(\frac\s 2,\frac\r 2,\frac\m 2)$,
$$
\aa{\Phi_\z-\z}_0+\a{\Phi_\f-\f}+\a{\Phi_r-r}\le\be\ep,$$ and the
mapping $\o\mapsto\o'(\o)$ verifies
$$\a{\o'-\mathrm{id}}_{\CC^1(U')}\le\be\ep.$$
$\be$ is a constant that depends on the dimensions $d,\#\AA,m_*$, on
the constants $C,C'$ and on $V$ and $F$.
\end{Main}

The consequences of the theorem are  well-known. The dynamics of the
Hamiltonian vector field of $h_{\o'}+\ep f$ on
$\Phi(\b{0}\times\T^d\times\b{0})$ is the same as that of
$$\sc{\o,r}+\frac12\!\!\sc{\xi,Q_1\xi}+\sc{\xi,Q_2\eta}+
\frac12\!\!\sc{\eta,Q_1\eta}\,.
$$
The torus $\b{\z=r=0}$ is invariant, since the Hamiltonian vector field on it is
$$
\left\{\begin{array}{l}
\dot\z=0\\
\dot\f=\o\\
\dot r=0,
\end{array}\right.
$$
and the flow on the torus is linear
$$t\mapsto \f+t\o.$$

Moreover, the linearized equation on this torus becomes
$$
\left\{\begin{array}{l} \frac{d}{dt} \hat \z J
\left(\begin{array}{cc} Q_1(\o) & Q_2(\o)\\ {}^t\!Q_2 (\o) &Q_1(\o)
\end{array}\right)\hat \z+Ja(\f+t\o,\o)\hat r\\
\frac{d}{dt} \hat \f=\sc{a(\f+t\o,\o),\hat \z}+b(\f+t\o,\o)\hat r\\
\frac{d}{dt} \hat  r=0,
\end{array}\right.
$$
where $a=\ep\p_r\p_\z f'$ and $b=\ep\p_r^2 f'$. Since $Q_1+iQ_2$  is
Hermitian and block diagonal the eigenvalues of the $\z$-linear part
are purely imaginary
$$\pm i\O_a',\quad a\in\LL.$$

The linearized equation is reducible to constant coefficients if the
imaginary part $\O_a'$  of the eigenvalues are non-resonant with
respect to $\o$, something which can be assumed if we restrict the
set $U'$ arbitrarily little. Then the $\hat \z$-component (and of
course also the $\hat r$-component) will have only quasi-periodic
(in particular bounded) solutions. The $\hat \f$-component  may have
a linear growth in $t$, the growth factor (the ``twist'') being
linear in $\hat r$.

\medskip

{\it Reducibility.} Reducibility is not only an important outcome of
KAM but also  an essential ingredient in the proof. It simplifies
the iteration since it makes it possible to reduce all approximate
linear equations to constant coefficients. But it does not come for
free. It requires a lower bound on small divisors of the form
$$
(**)\qquad\a{\sc{k,\o}+ \O'_a-\O'_b},\quad k\in\Z^{\AA},\ a,b\in\LL.
$$
The basic frequencies $\o$ will be kept fixed during the iteration
--  that's what the parameters are there for -- but the normal
frequencies will vary. Indeed $\O'_a(\o)$ and $\O'_b(\o)$ are
perturbations of $\O_a$ and $\O_b$ which are not known a priori but
are determined by the approximation process. \footnote{A lower bound
on $(**)$, often known as the second Melnikov condition, is strictly
speaking not necessary at all for reducibility. It is necessary,
however, or reducibility with a reducing transformation close to the
identity.}

This is a lot of conditions
for a few parameters $\o$. It is usually possible to
make a (scale dependent) restriction of $(**)$ to
$$\a{k},\ \a{a-b}\le\D=\D_\ep$$
which improves the situation a bit. Indeed, in  one space-dimension
($d=1$) it improves a lot, and $(**)$ reduces to only finitely many
conditions. Not so however when $d\ge2$, in which case the number of
conditions in $(**)$ remains infinite.

To cope with this problem we shall exploit the
T\"oplitz-Lipschitz-property which allows for a sort of
compactification of the dimensions and reduces the infinitely many
conditions $(**)$ to finitely many. These can then be controlled by
an appropriate choice of $\o$.

\medskip

{\it The T\"oplitz-Lipschitz property.} The T\"oplitz-Lipschitz
property is defined for  infinite-dimensional matrices with
exponential decay. We say that a matrix
$$A:\LL\times\LL\to\C$$
is T\"oplitz at $\infty$ if, for all $a,b,c\in\Z^d$ the limit
$$\lim_{t\to\infty}A_{a+tc}^{b+tc}\ \exists\quad =:\ A_a^b(c).$$
The T\"oplitz-limit $A(c)$ is a new matrix which is $c$-invariant
$$A_{a+c}^{b+c}(c)=A_a^b(c).$$
So it is a simpler object because it is  ``more constant''.

The approach to the T\"oplitz-limit in direction $c$ is controlled
by a Lipschitz-condition. This control
does not take place everywhere, but on a certain subset
$$D_{\L}(c)\in\LL\times\LL$$
--  the Lipschitz domain. $\L$ is a parameter which,  together with
$|c|$, determines the size of the domain.

The T\"oplitz-Lipschitz property permits us to verify certain bounds
of the matrix-coefficients or functions of these, like determinants
of sub-matrices, in the T\"oplitz-limit and then recover these
bounds for the matrix restricted to the Lipschitz domain.

The matrices we shall consider will not be scalar-valued but $gl(2,\C)$-valued
$$A:\LL\times\LL\to gl(2,\C)$$
and we shall define a T\"oplitz-Lipschitz property for such matrices
also. These matrices constitute an algebra: one can multiply them
and  solve linear differential  equations. A function $f$ is said to
have the T\"oplitz-Lipschitz property if its Hessian (with respect
to $\zeta$) is T\"oplitz-Lipschitz.  If this is the case, as it is
for the perturbation $f$ of the non-linear Schr\"odinger, then this
is also true for the linear part of our KAM--transformations and for
the transformed Hamiltonian. This will permit us to formulate an
inductive statement which, as usual in KAM, gives Theorem~\ref{tA}.

\medskip

{\it Some references.} For finite dimensional Hamiltonian systems
the first proof  of persistence of stable (i.e. vanishing of  all
Lyapunov exponents) lower dimensional invariant tori was obtained in
\cite{E85,E88} and there are now many works on this subjects. There
are also many works on reducibility (see for example \cite{K99,
E01}) and the situation in finite dimension is now pretty well
understood in the perturbative setting. Not so, however, in infinite dimension.

If $d=1$ and the space-variable $x$ belongs to a finite segment
supplemented by Dirichlet or Neumann boundary conditions, this
result was obtained  in \cite{K88} (also see \cite{K1, P96a}). The
case of periodic boundary conditions was treated  in \cite{Bo96},
using another multi--scale scheme, suggested by Fr\"ohlich--Spencer
in their work on the Anderson localization \cite{FS83}. This
approach, often referred to as the Craig-Wayne scheme, is different
from KAM. It avoids the, sometimes, cumbersome condition $(**)$ but
to a high cost: the approximate linear equations are not of constant
coefficients. Moreover, it gives persistence of the invariant tori
but no reducibility and no information on the linear stability. A
KAM-theorem for periodic boundary conditions has recently been
proved in \cite{GY05} (with a perturbation $F$ independent of $x$)
and the perturbation theory for quasi-periodic solutions of
one-dimensional Hamiltonian PDE is now sufficiently well developed
(see for example \cite{K1, Cr00, K2}).

The study of the corresponding problems for $d\ge2$ is at its early
stage. Developing further the scheme, suggested by
Fr\"ohlich--Spencer, Bourgain proved persistence for the case $d=2$
\cite{Bo98}. More recently, the new techniques developed by him and
collaborators in their work on the linear problem has allowed him to
prove persistence in any dimension $d$ \cite{Bo04}. (In this work he
also treats the non-linear wave equation.)

\medskip

{\it Description of the paper.} The paper is divided into three
parts. The first part deals with linear algebra of
T\"oplitz-Lipschitz matrices and the analysis  of functions with the
T\"oplitz-Lipschitz property. In Section \ref{s2} we introduce
T\"oplitz-Lipschitz matrices and prove a product formula. This part
is treated in greater generality in \cite{EK1}. In Section \ref{s3}
we analyze functions with the T\"oplitz-Lipschitz property.

The second part deals with the bounds on the small divisors $(**)$
which occurs in the solution of the homological equation. In
Section~\ref{s4} we analyze the block decomposition of the lattice
$\Z^d$ and in  Section \ref{s5} we study the small divisors. In
Section \ref{s6} we solve the homological equations. This part is
independent of the first part except for basic definitions and
properties given in Sections \ref{ss23} and \ref{ss24}.

The third part treats KAM-theory with T\"oplitz-Lipschitz property
and contains a general KAM-theorem, Theorem \ref{t71}. This theorem
is applied  to the non-linear Schr\"odinger to give Theorem
\ref{t72} of which the theorem above is a variant.

\medskip

{\it Notations.} $\sc{\ ,\ }$ is the standard scalar product in
$\R^d$.  $\aa{\ }$ is an operator-norm or $l^2$-norm. $\a{\ }$ will
in general denote a supremum norm, with a notable exception: for  a
lattice vector $a\in\Z^d$ we use $\a{a}$ for the $l^2$-norm.

$\AA$ is a finite subset of $\Z^d$ and $\LL$ is the complement of
a finite subset of $\Z^d$. For the non-linear Schr\"odinger equation
$\LL$ will be the complement of $\AA$, but this not assumed in general.

A matrix on $\LL$ is just a mapping
$A:\LL\times\LL\to\ \C\quad \text{or}\quad gl(2,\C)$. Its
components will be denoted $A_a^b$.

The dimension $d$ will be fixed and $m_*$
will be a fixed constant $>\frac{d}2$.

$\lsim$ means $\le$ modulo a multiplicative constant that only, unless
otherwise specified, depends on $d,m_*$ and $\#\AA$.

The points in the lattice $\Z^d$ will be denoted $a,b,c,\hdots$. Also $d$
will sometimes be used, without confusion we hope.

For a vector $c\in \Z^d$, $c^{\perp}$ will denote the $\perp$ complement
of $c$ in $\Z^d$ or in $\R^d$, depending on the context. If $c\not=0$,
for any $a\in \Z^d$ we let
$$a_c\in(a+\R c)\cap\Z^d$$
be the lattice point $b$ on the line $a+\R c$ with smallest norm,
i.e. that minimizes
$$\a{\sc{b,c}}$$
-- if there are two such $b$'s we choose the one  with
$\sc{b,c}\ge0$. It is the``$\perp$ projection of $a$ to
$c^{\perp}$''.

Greek letter $\al,\be,\hdots$ will mostly be used for bounds.
Exceptions are $\f$ which will denote an element in the  torus  --
an angle  --  and $\o,\O$.

For two subsets $X$ and $Y$ of a metric space,
$$\dist(X,Y)=\inf_{x\in X,y\in Y}d(x,y).$$
(This is not a metric.) $X_{\ep}$ is the $\ep$-neighborhood of $X$, i.e.
$$\b{y: \dist(y,X)<\ep}.$$
Let $B_{\ep}(x)$ be the ball $\b{y:d(x,y)<\ep}.$
Then $X_{\ep}$ is the union, over $x\in X$, of all $B_{\ep}(x)$.

If $X$ and $Y$ are subsets of $\R^d$ or $\Z^d$ we let
$$X-Y=\b{x-y:x\in X,\ y\in Y}$$
 -- not to be confused with the set theoretical difference $X\sm Y$.

\medskip

{\it Acknowledgment.} This work started a few years ago during the
Conference on Dynamical Systems in Oberwolfach as an attempt to try
to understand if a  KAM--scheme could be applied to multidimensional
Hamiltonian PDE's and in particular to the non-linear Schr\"odinger.
This has gone on at different place and we are grateful for support
from ETH, IAS, IHP and from the Fields Institute in Toronto, where
these ideas were presented  for  the first time in May 2004 at the
workshop on Hamiltonian dynamical systems. The first author also
want to acknowledge the hospitality of the Chinese University of
Hong-Kong and the second author the support of EPSRC, grant
S68712/01.

\bigskip
\bigskip

\centerline{PART I. THE T\"OPLITZ-LIPSCHITZ PROPERTY}

\bigskip
In this part we consider
$$\LL\sbs\Z^d$$
and matrices $A:\LL\times\LL\to gl(2,\C)$. We define: the sup-norms
$\a{\ \cdot\ }_\g$; the notion of being T\"oplitz at $\i$; the
Lipschitz-domains $D_\D^{\pm}(c)$; the Lipschitz- norm $\sc{\ \cdot\
}_{\L,\g}$ and the notion of being T\"oplitz-Lipschitz. (For a more
general exposition see \cite{EK1}.) We define the
T\"oplitz-Lipschitz property for functions and the norms $[\ \cdot\
]_{\L,\g,\s}$.

\section{T\"oplitz-Lipschitz matrices}
\label{s2}

\subsection{Spaces and matrices}
\label{ss21}
\

\noindent
We denote by $ l^2_\ga(\LL,\C^2) ,\ \g\ge0$,
the following weighted $l_2$-spaces:
$$
l^2_\ga(\LL,\C^2) =\{\zeta=(\xi,\eta)\in\C^{\LL}
\times\C^{\LL}:\aa{\zeta}_\ga<\infty\},
$$
where
$$
\|\zeta\|_\ga^2=\sum_{a\in{\LL}}(|\xi_a|^2+|\eta_a|^2)e^{2\ga|a|}\lan
a\ran^{2m_*}\,,\;\;\;\lan a\ran=\max(|a|,1).
$$

We provide $l^2_\ga(\LL,\C^2)$ with the symplectic form
$$\sum_{a\in\LL}d\xi_a\wedge d\eta_a.$$
Using the pairing
$$\sc{\z,\z'}=\sum_{a\in\LL}(\xi_a\xi'_a+\eta_a\eta'_a)$$
we can write the symplectic form as
$$\sc{\cdot,J\cdot}$$
where $J:l^2_\ga(\LL,\C^2)\to l^2_\ga(\LL,\C^2)$ is  the standard
involution, given by the component-wise application of the matrix
$$
J=\left(\begin{array}{cc} 0&1\\-1&0\end{array}\right).
$$

We consider the space $gl(2,\C)$ of all complex $2\times2$-matrices
provided with the scalar product
$$
Tr({}^t\bar A B),
$$
and consider the orthogonal projection
$$
\pi:gl(2,\C)\to M\,,\quad M=\C I+\C J.
$$
It is easy to verify that
$$\left\{\begin{array}{l}
M\times M, M^{\perp}\times M^{\perp}\sbs M\\
M\times M^{\perp}, M^{\perp}\times M\sbs M^{\perp}
\end{array}\right.$$
and
$$\left\{\begin{array}{l}
\pi(AB)=\pi A\pi B+(I-\pi)A(I-\pi)B\\
(I-\pi)(AB)=(I-\pi)A\pi B+\pi A(I-\pi)B.
\end{array}\right.$$

If $A=(A_i^j)_{i,j=1}^2$ $B=(B_i^j)_{i,j=1}^2$ we define
$$[A]=(|A_i^j|)_{i,j=1}^2,$$
and
$$A\le B\iff |A_i^j|\le B_i^j,\quad \forall i,j.$$

Since any Euclidean space $E$ is naturally isomorphic to its dual
$E^*$, the canonical relations
$$E\otimes E\simeq E^*\otimes E^*\simeq
Hom(E,E^*)\simeq Hom(E,E)$$
permits the identification of the tensor product $\z\otimes\z'$
with a $2\times2$-matrix
$$(\z\otimes\z')_i^j=\z_i\z_j'.$$

\subsection{Matrices with exponential decay}
\label{ss22}
\

\noindent Consider now an infinite-dimensional $gl(2,\C)$-valued
matrix
$$A:\LL\times\LL\to gl(2,\C),\quad (a,b)\mapsto A_a^b.$$
We define matrix multiplication through
$$(AB)_a^b=\sum_d A_a^dB_d^b,$$
and, for any subset $\DD$ of $\LL\times\LL$, the semi-norms
$$
\a{A}_{\DD}=\sup_{(a,b)\in\DD}\|A_a^b\|
$$
(here $\aa{\ }$ is the operator-norm).

We define $\pi A$ through
$$(\pi A)_a^b=\pi A_a^b,\quad \forall a,b.$$
Clearly we have
\begin{equation}\label{e21}
\begin{array}{l}
\pi(A+B)=\pi A+\pi B\\
\pi(AB)=\pi A\pi B+(I-\pi)A(I-\pi)B\\
(I-\pi)(AB)=(I-\pi)A\pi B+\pi A(I-\pi)B.
\end{array}\end{equation}

We define
$$A\le B\iff A_a^b\le B_a^b,\quad \forall a,b,$$
and
$$(\EE_{\g}^{\pm}A)_a^b=[A_a^b]e^{\g\a{a\mp b}},\quad \forall a,b.$$
All operators $\EE_{\g}^{\pm} $ commute and
we have
$$\left\{\begin{array}{l}
\EE_{\g}^{x} (A+B)\le\EE_{\g}^{x} A+\EE_{\g}^{x}B,\quad x\in\b{+,-}\\
\EE_{\g}^{xy} (AB)\le(\EE_{\g}^{x} A)(\EE_{\g}^{y} B),\quad
x,y\in\b{+,-}.
\end{array}\right.$$
\footnote{We use the sign convention that $xy=+$  whenever $x$ and
$y$ are equal and $xy=-$ whenever they are different.}

We define the norm
$$
|A|_\ga \max(|\EE_{\g}^+\pi A_a^b|_{\LL\times\LL},
|\EE_{\g}^-(1-\pi) A_a^b|_{\LL\times\LL}).$$

We have, by {\it Young's inequality} (see \cite{F}), that
\begin{equation}\label{e22}
\aa{A\z}_{\g'}\lsim
(\frac1{\g-\g'})^{d+m_*}
\a{A}_{\ga}\aa{\z}_{\g'},\quad \forall \g'<\g.
\end{equation}
(Take for example $A=\pi A$ and apply Young's inequality to
the matrix $\tilde A$ defined by
$$\tilde A_a^b= e^{\g'|a|} \lan a\ran^{m_*} A_a^b  \lan b\ran^{-m_*}e^{-\g'|b|}.)$$
It follows that if $\a{A}_{\g}<\infty$, then $A$ defines a bounded
operator on any $l^2_{\g'}(\LL,\C^2),\ \g'<\g$.

{\it Truncations.} Let
$$(\TT_\D^{\pm})A_a^b\left\{\begin{array}{ll}
A_a^b & \textrm{\ if\ } |a\mp b|\le\D\\
0 &  \textrm{\ if not},
\end{array}\right.$$
and
$$ \TT_\D A  = \TT_\D^{+}\pi A +
\TT_\D^{-}(I-\pi)A.$$
It is clear that
\begin{equation}\label{e22,5}
\a{\TT_\D A}_\g\le\a{A}_\g\quad\text{and}\quad
\a{A- \TT_\D A}_{\g'}\le e^{-\D(\g-\g')}\a{A}_\g.
\end{equation}

{\it Tensor products.} For any two elements  $\z,\z'\in
l^2_\ga(\LL,\C^2)$, their tensor product $ \z\otimes\z' $ is a
matrix on $\LL\times\LL$, and it is easy to verify that
\begin{equation}\label{e23}
\a{\z\otimes\z' }_{\g}\lsim\aa{\z}_\ga\aa{\z'}_\g.
\end{equation}

{\it Multiplication.}
We have
\begin{equation}\label{e24}
\a{AB}_{\g'}+\a{BA}_{\g'}
\lsim (\frac1{\g-\g'})^{d}
\a{A}_\g\a{B}_{\g'},\quad\forall\g'<\g.
\end{equation}

{\it Linear differential equation.} Consider the linear system
$$\left\{\begin{array}{l}
X'=A(t)X\\
X(0)=I.\end{array}\right.$$
It follows from (\ref{e24}) that the series
$$I+\sum_{n=1}^{\infty} \int_0^{t_0}\int_0^{t_1}\hdots
\int_0^{t_{n-1}}
A(t_1)A(t_2)\hdots A(t_n)dt_n\hdots dt_2dt_1,$$
as well as its derivative with respect to $t_0$, converges
to a solution which verifies, for $\g'<\g$,
\begin{equation}\label{e25}
\a{X(t)-I}_{\g'}
\lsim(\g-\g')^{d}
(\exp(\cte(\frac1{\g-\g'})^{d}
|t|\al(t))-1),
\end{equation}
where
$$\al(t)=\sup_{0\le |s|\le |t|}\a{A(s)}_\g.$$

\subsection{T\"oplitz-Lipschitz matrices ($d=2$)}
\label{ss23}
\

\noindent A matrix
$$A:\LL\times\LL\to gl(2,\C)$$
is said to be {\it T\"oplitz at} $\infty$ if, for all
$a,b,c$, the two limits
$$\lim_{t\to+\infty}A_{a+tc}^{b\pm tc}\ \exists\ =A_a^b(\pm,c).$$
It is easy to verify that if $\a{A}_\g<\infty$ and
$\a{B}_\g<\infty$,  then
$$(\pi A)(-,c)=(I-\pi)A(+,c)=0$$
and
\begin{equation}\label{e26}
\begin{array}{l}
\pi(AB)(+,c) =\\
\pi A(+,c)\pi B(+,c)+ (I-\pi)A(-,c)(I-\pi)B(-,-c)\\
(I-\pi)(AB)(-,c) =\\
(I-\pi) A(-,c)\pi B(+,-c)+ \pi A(+,c)(I-\pi)B(-,c).
\end{array}\end{equation}

In the rest of this section we assume that
$$
c\ne0.
$$

We define
$$(\MM_cA)_a^b=(\max(\frac{|a|}{|c|},\frac{|b|}{|c|})+1)[A_a^b],
\quad\forall a,b.$$
The operators $\MM_c$ and $\EE_{\g}^{\pm}$ all commute and
$$\MM_c(AB)\le(\MM_cA)(\MM_cB).$$

{\it Lipschitz domains.} For a non-negative constant $\L$, let
$$D_{\L}^+(c)\sbs\LL\times\LL$$
be the set of all $(a,b)$ such that there exist  $a',\ b'\in\Z^d$
and $t\ge0$ such that
$$
\left\{\begin{array}{ccc}
\a{a=a'+tc}& \ge &\L(\a{a'}+\a{c})\a{c}\ \\
\a{b=b'+tc}& \ge &\L(\a{b'}+\a{c})\a{c}
\end{array}\right.$$
and
$$\frac{\a{a}}{\a{c}},\quad \frac{\a{b}}{\a{c}}\ \ge\ 2\L^2.$$

We give here some elementary properties of the Lipschitz domains.
They will be studied further in  Section~\ref{s4}.

\begin{Lem}\label{l21}
Let $t\ge0$.
\begin{itemize}
\item[(i)] For $\L\ge1$,
$$ t\ge\L\a{c}\ge\L$$
if $\a{a=a'+tc}\ge\L(\a{a'}+\a{c})\a{c}$.
\item[(ii)] For $\L>1$,
$$
\left\{\begin{array}{ll}
\a{a'}\le \frac t{\L-1}-\a{c} & \text{if}\
\a{a=a'+tc}\ge\L(\a{a'}+\a{c})\a{c}\\
\a{a'}\ge \frac t{\L+1}-\a{c} & \text{if not}.
\end{array}\right.
$$
\item[(iii)] For $\L>1$,
$$\a{\frac{\a{a}}{\a{c}}-t}\le\frac t{\L-1}
\ \text{and}\  \a{\frac{\sc{a,c}}{\a{c}^2}-t} \le\frac t{\L-1},
$$ if
$\a{a=a'+tc}\ge\L(\a{a'}+\a{c})\a{c}$.
\item[(iv)] For $\O\ge(\L+1)(\a{a-b}+1)$ we have
$$
|b=b'+tc|\ge \Lambda(|b'|+|c|)|c| \quad\text{with}
\quad b'=a'+b-a,
%\a{b}\ge\L(\a{a'+b-a}+\a{c})\a{c}
$$
if $\ \a{a=a'+tc}\ge\O(\a{a'}+\a{c})\a{c}$.
\end{itemize}
\end{Lem}

\begin{proof}
This is a direct computation.
\end{proof}

\begin{Cor}\label{c22}
Let $\L\ge3$.
\begin{itemize}
\item[(i)]
$$
(a,b)\in D_{\L}^+ (c)
 \Longrightarrow\\
\frac{\a{a}}{\a{c}}\approx \frac{\a{b}}{\a{c}}\approx
\frac{\sc{a,c}}{\a{c}^2}\approx
\frac{\sc{b,c}}{\a{c}^2}\gsim \L\a{c}.$$
\item[(ii)]
$$(a,b)\in D_{\L}^+(c)
 \Longrightarrow\\
(a+tc,b+tc)\in D_{\L}^+(c)\quad \forall t\ge0.$$
\item[(iii)]
$$
(a,b)\in D_{\L}^+(c)
\ \Longrightarrow\  (\tilde a,\tilde b)\in D_{\O}^+(c),$$
where
$$\O=\L-\max(|\tilde a-a|,\ |\tilde b-b|)-2.$$
\item[(iv)]
$$
(a,b)\in D_{\L+3}^+(c),\ (a,d)\notin D_{\L}^+(c)
\Longrightarrow
\a{a-d},\ \a{b-d}\gsim\frac1{\L^2}\frac{|a|}{|c|}.$$
\end{itemize}
\end{Cor}

\begin{proof}
(i) follows from Lemma \ref{l21} (i)+(iii) if we just  observe that
$$t \approx t+\frac{t}{\L-1} \approx t-\frac{t}{\L-1}.$$

In order to see (ii) we write $a=a'+sc,\ s\ge0,$ with
$\a{a} \ge \L(\a{a'}+\a{c})\a{c}$.
Then
$$\a{a+tc}^2= |a|^2+t^2|c|^2+2t\sc{a,c}|a|^2+t^2|c|^2+ 2ts|c|^2+2t\sc{a',c}.$$
By Lemma \ref{l21}(ii)
$$2ts|c|^2+2t\sc{a',c}\ge 2ts(1-\frac1{\L-1})|c|^2\ge 0.$$
Hence
$$\a{a+tc}^2\ge |a|^2+t^2|c|^2\ge|a|^2\ge \L(\a{a'}+\a{c})\a{c}.$$
Moreover, for all $t\ge0$
$$ \frac{|a+tc|}{|c|}\ge \frac{|a|}{|c|}\ge 2\L^2.$$

The same argument applies to $b$.

To see (iii), let $\D=\max(|\tilde a-a|,\ |\tilde b-b|)+2$ and write
$a=a'+tc$ with $|a|\ge\L(|a'|+|c|)|c|$. Then $ \tilde a=a'+ \tilde
a-a +tc$, and if
$$|\tilde a|<\O(|a'+\tilde a-a|+|c|)|c|$$
then by Lemma~\ref{l21}(ii)
$$|\tilde a-a|\ge \frac{t\D}{ (\O+1)(\L-1) }.$$
This implies that $t\le(\O+1)(\L-1)$ and, hence,
$$\frac{|a|}{|c|}<2\L^2$$
which is impossible. Therefore
$$|\tilde a|\ge\O(|a'+\tilde a-a|+|c|)|c|.$$
Moreover
$$ \frac{|\tilde a|}{|c|}\ge \frac{|a|}{|c|}-
\frac{\D}{|c|}\ge 2\L^2-\D\ge 2\O^2.$$

The same argument applies to $b$.

To see (iv), assume that  $\frac{|d|}{|c|}<2\L^2$. As
$\frac{|b|}{|c|}\ge 2(\L+3)^2$ it follows that
$$
\frac{|b-d|}{|c|}\ge 12\L.
$$
So $|b-d|\ge \L^{-2}\,\frac{|a|}{|c|}$, unless
$$
\frac{|a|}{|c|}\ge 12\L^3|c|.
$$
In this case due to Lemma~\ref{l21}.(iii) $\frac{|b|}{|c|}\ge
\frac{\L+1}{\L+3}\frac{\a{a}}{\a{b}}\ge12\L^2$. So  we must have
$$
\frac{|d|}{|c|}\le 2\L^2\le\frac1{6}\frac{|b|}{|c|}
$$
which implies that
$$
\frac{|b-d|}{|c|}\ge\frac{5}{6}\,\frac{|b|}{|c|}\ge
\frac1{\L^2}\frac{|a|}{|c|}.
$$

Therefore we can assume that $\frac{|d|}{|c|}\ge2\L^2$. Since
$(a,b)\in D_{\L+3}(c)$, then $b=b'+tc$, where
$$
|b|\ge (\L+3)(|b'|+|c|)|c|.
$$
Let us write $d$ as $d=b+(d-b)=d'+tc$, $d'=b'+(d-b)$. Since
$(a,d)\notin D^+_\L(c)$ while $(a,b)\in D^+_{\L+3}(c)\subset
D^+_\L(c)$ and $\frac{|d|}{|c|}\ge 2\L^2$, then
$|d|<\L(|d'|+|c|)+|c|$. Applying Lemma~\ref{l21}.(ii) we ge that
$$
|b'|\le\frac{t}{\L+2}-|c|,\quad |d'|\ge \frac{t}{\L+1}-|c|.
$$
Hence,
$|b-d|=|d'-b'|\ge\frac{t}{\L+1}-\frac{t}{\L+2}\gsim\frac{t}{\L^2}
\gsim\frac1{\L^2}\,\frac{|a|}{|c|}$, where we used
Lemma~\ref{l21}.(iii).

Now the required estimate for $|b-d|$ is
established. Similar arguments apply to $|a-d|$.
%Now the assumption that $(a,d)\notin D_\L^+(c)$ leads to the
%conclusion by Lemma~\ref{l21} (i)+(ii).
\end{proof}

{\it Lipschitz constants and norms.} Define the  Lipschitz-constants
$$\Lip_{\L,\g}^xA\sup_{c}|\EE_{\g}^x\MM_c(A-A(x,c))|_{D_{\L}^x(c)},\quad
x\in\b{+,-},$$ (see the notations of Section \ref{ss22}) and the
Lipschitz-norm
$$\l{A}_{\L,\g}=\max(\Lip_{\L,\g}^+\pi A,\Lip_{\L,\g}^-
(1-\pi)A)+\a{A}_\g.$$
Here we have defined
$$(a,b)\in D_{\L}^-(c)\iff (a,-b)\in D_{\L}^+(c).$$
The matrix $A$ is {\it T\"oplitz-Lipschitz} if it is T\"oplitz at
$\i$ and $\l{A}_{\L,\g}<\i$ for some $\L,\g$.

{\it Truncations.} It is easy to see that
\begin{equation}\label{e265}
\begin{array}{ll}
\l{\TT_\D A}_{\L,\g} &\le \quad \l{A}_{\L,\g}\\
\l{A-\TT_\D A}_{\L,\g'}&\le \quad e^{-\D(\g-\g')}\l{A}_{\L,\g}.
\end{array}
\end{equation}

{\it Tensor products.} It is easy to verify that
\begin{equation}\label{e27}
\l{\z\otimes\z' }_{\L,\g}\lsim\aa{\z}_\ga\aa{\z'}_\g.
\end{equation}

Multiplications and differential equations are more delicate and we
shall need the following proposition.

\begin{Prop}\label{p22}
For all $x,y\in\b{+,-}$, all $\g'<\g$ and any $c\ne0$
\begin{itemize}
\item[(i)]
$$
\begin{array}{ll}
\a{\EE_{\g'}^{xy}\MM_c(AB)}_{D_{\L+3}^{xy}(c)} \lsim&
(\frac1{\g-\g'})^{d} \a{\EE_{\g_1}^{x}\MM_c(A)}_{D_{\L}^{x}(c)}
\a{\EE_{\g_2}^{y}B}_{\LL\times\LL}+\\
&\L^2(\frac1{\g-\g'})^{d+1}
\a{\EE_{\g_1}^{x}A}_{\LL\times\LL}
\a{\EE_{\g_2}^{y}B}_{ \LL\times\LL},
\end{array}$$
where one of $\g_1,\g_2$ is $=\g$ and the other one is $=\g'$. The
same bound holds for $BA$.
\item[(ii)]
$$
\begin{array}{ll}
\a{\EE_{\g'}^{xyz}\MM_c(ABC)}_{D_{\L+6}^{xyz}(c)}\lsim&
(\frac1{\g-\g'})^{2d} \a{\EE_{\g_1}^{x}A}_{\LL\times\LL}
\a{\EE_{\g_2}^{y}\MM_c(B)}_{D_{\L}^{y}(c)}
\a{\EE_{\g_3}^{z}C}_{\LL\times\LL}+\\
&\L^2(\frac1{\g-\g'})^{2d+1}
\a{\EE_{\g_1}^{x}A}_{\LL\times\LL}
\a{\EE_{\g_2}^{y}B}_{ \LL\times\LL}
\a{\EE_{\g_3}^{z}C}_{\LL\times\LL},
\end{array}$$
where two of $\g_1,\g_2,\g_3$ are $=\g$ and the third one is $=\g'$.
The same bound holds if we permute the factors $A, B$ and
$C$.
\end{itemize}
\end{Prop}

\begin{proof}
To prove (i), let first $x=y=+$. We shall only prove the  estimate
for $AB$  -- the estimate for $BA$ being the same. Notice that for
$(a,b)\in D_{\L+3}^+(c)$ we have, by Corollary \ref{c22}(i), that
$$M_c(a,b)=\max(\frac{|a|}{|c|},\frac{|b|}{|c|})+1
\approx \frac{|a|}{|c|}+1.$$

Now, for $(a,b)\in D^+_{\L+3}(c)$ we have
$$
\begin{array}{l}
(\EE_{\g'}^{+}\MM_c(AB))_a^b\le
\sum_d M_c(a,b)[A_a^d][B_d^b] e^{\g'\a{a-b}}=\\
\sum_{(a,d)\in D_{\L}^+(c)}\hdots+
\sum_{(a,d)\notin D_{\L}^+(c)}\hdots
= (I) + (II).
\end{array}
$$

In the domain of (I) we have, by Corollary~\ref{c22}(i), that
$$M_c(a,b)\approx \frac{|a|}{|c|}+1\approx M_c(a,d),$$
so
$$(I)\lsim
\a{\EE_{\g_1}^{+}\MM_cA}_{D_{\L}^{+}(c)}
\a{\EE_{\g_2}^{+}B}_{\LL\times\LL} \sum_d
e^{-(\g_1-\g')\a{a-d}-(\g_2-\g')\a{d-b}}.$$   Since one of
$\g_1-\g'$ and $\g_2-\g'$ is $\g-\g'$ the sum is
$$\lsim (\frac1{\g-\g'})^d.$$

In the domain of (II) we have,
by Corollary \ref{c22}(iv), that
$$|a-d|,\ |b-d|\gsim\frac1{\L^2}\frac{|a|}{|c|},$$
so $(II)$ is
$$\begin{array}{l}
\lsim
\a{\EE_{\g_1}^{+}A}_{\LL\times\LL}
\a{\EE_{\g_2}^{+}B}_{\LL\times\LL}\times\\
\sum_{|a-d|,|d-b|\gsim\frac1{\L^2}\frac{|a|}{|c|}}
(\frac{|a|}{|c|}+1)
e^{-(\g_1-\g')\a{a-d}-(\g_2-\g')\a{d-b}}.
\end{array}$$
Since one of $\g_1-\g'$ and $\g_2-\g'$ is $\g-\g'$ the sum is
$$\lsim \L^2(\frac1{\g-\g'})^{d+1}.$$

The three other cases of (i) are treated in the same way.

To prove (ii), let first $x=y=z=+$. Notice that for $(a,b)\in
D_{\L+6}^+(c)$ we have, by  Corollary \ref{c22}(i), that
$$M_c(a,b)=\max(\frac{|a|}{|c|},\frac{|b|}{|c|})+1
\approx \frac{|a|}{|c|}+1.$$

Now
$$
\begin{array}{l}
(\EE_{\g'}^{+}\MM_c(ABC))_a^b\le
\sum_{d,e} M_c(a,b)[A_a^d][B_d^e][C_e^b] e^{\g'\a{a-b}}\le\\
\sum_{|d|\ge|e|}\hdots+
\sum_{|e|\ge|d|}\hdots.
\end{array}
$$
We shall only consider the first of these sums  --  the second one
being analogous. We decompose this sum as
$$\sum_{\begin{subarray}{c} (a,d)\in D_{\L+3}^+(c)\\ (d,e)\in
D_{\L}^+(c)\end{subarray}}\hdots
+ \sum_{\begin{subarray}{c} (a,d)\in D_{\L+3}^+(c)\\ (d,e)\notin
D_{\L}^+(c)\end{subarray}}\hdots + \sum_{(a,d)\notin
D_{\L+3}^+(c)}\hdots= (I)+(II)+(III).$$

In the domain of (I) we have, by Corollary~\ref{c22}(i), that
$$M_c(d,e)\approx M_c(a,b),$$
so $(I)$ is
$$\begin{array}{l}
\lsim \a{\EE_{\g_1}^{+}A}_{\LL\times\LL} \a{\EE_{\g_2}^{+}\MM_c
B}_{D_{\L}^{+}(c)}\a{\EE_{\g_3}^{+}C}_{\LL\times\LL}\times\\
\sum_{d,e}
e^{-(\g_1-\g')\a{a-d}-(\g_2-\g')\a{d-e}-(\g_3-\g')\a{e-b}}.
\end{array}$$
Since two of $\g_1-\g'$, $\g_2-\g'$ and $\g_3-\g'$ are $\g-\g'$ the
sum is
$$\lsim (\frac1{\g-\g'})^{2d}.$$

By Corollary \ref{c22}(iv) we have, in the domain of (II),
$$|a-d|,\ \a{d-e}\gsim\frac1{\L^2}\frac{|a|}{|c|}.$$
and, in the domain of (III),
$$|a-d|,\ \a{d-b}\gsim\frac1{\L^2}\frac{|a|}{|c|}.$$
Hence in both these domains we have
$$s(d,e)=\max(|a-d|,|d-e|,|e-b|)\gsim\frac1{\L^2}\frac{|a|}{|c|},$$
so $(II)+(III)$ is
$$\begin{array}{l}
\lsim
\a{\EE_{\g_1}^{+}A}_{\LL\times\LL}
\a{\EE_{\g_2}^{+}B}_{\LL\times\LL}
\a{\EE_{\g_3}^{+}C}_{\LL\times\LL}\times\\
\sum_{s(d,e)\gsim\frac1{\L^2}\frac{|a|}{|c|}} (\frac{|a|}{|c|}+1)
e^{-(\g_1-\g')\a{a-d}-(\g_2-\g')\a{d-e}-(\g_3-\g')\a{e-b}}.
\end{array}$$
Since two of $\g_1-\g'$, $\g_2-\g'$ and $\g_3-\g'$ are $\g-\g'$ the
sum is
$$\lsim \L^2(\frac1{\g-\g'})^{2d+1}.$$

The seven other cases of (ii) are treated in the same way, as well
as the case when the factors $A,B$ and $C$ are permuted.
\end{proof}

We give a more compact and slightly weaker formulation of this
result.

\begin{Cor}\label{c23}
For all $x,y\in\b{+,-}$, all $\g'<\g$ and any $c\ne0$
\begin{itemize}
\item[(i)]
$$
\begin{array}{ll}
\a{\EE_{\g'}^{xy}\MM_c(AB)}_{D_{\L+3}^{xy}(c)} \lsim &
\L^2(\frac1{\g-\g'})^{d+1}
\big[\a{\EE_{\g_1}^{x}A}_{\LL\times\LL}+\\
& \a{\EE_{\g_1}^{x}\MM_c(A)}_{D_{\L}^{x}(c)}\big]\a{\EE_{\g_2}^y
B}_{\LL\times\LL},
\end{array}$$
where one of $\g_1,\g_2$ is $=\g$ and the other one is $=\g'$.
The same bound holds for $BA$.
\item[(ii)]
$$
\begin{array}{ll}
\a{\EE_{\g'}^{xyz}\MM_c(ABC)}_{D_{\L+6}^{xyz}(c)}\lsim&
\L^2(\frac1{\g-\g'})^{2d+1} \a{\EE_{\g_1}^{x}A}_{\LL\times\LL}
[\a{\EE_{\g_2}^{y}\MM_c(B)}_{D_{\L}^{y}(c)}+\\
&\a{\EE_{\g_2}^{y}B}_{ \LL\times\LL}]
\a{\EE_{\g_3}^{z}C}_{\LL\times\LL},
\end{array}$$
where two of $\g_1,\g_2,\g_3$ are $=\g$ and the third one is $=\g'$.
The same bound holds when the factors $A,B$ and $C$ are permuted.
\end{itemize}
\end{Cor}

{\it Multiplication.} Using  relations (\ref{e21}) and (\ref{e26})
we obtain  from Corollary~\ref{c23}.(i) that a product of two
T\"oplitz-Lipschitz matrices is again T\"oplitz-Lipschitz and for
all $\g'<\g$
\begin{equation}\label{e28}
\begin{array}{c}
\l{AB}_{\L+3,\g'}\lsim\\
\L^2(\frac1{\g-\g'})^{d+1}
\big[\l{A}_{\L,\g_1}\a{B}_{\g_2}+\a{A}_{\g_1}\l{B}_{\L,\g_2}\big],
\end{array}\end{equation}
where one of $\g_1,\g_2$ is $=\g$ and the other one is $=\g'$.

This formula cannot be iterated without consecutive loss of the
Lipschitz domain. However Corollary~\ref{c23}(ii) together with
(\ref{e24}) gives for all $\g'<\g$
\begin{equation}\label{e29}
\begin{array}{c}
\l{A_1\cdots A_n }_{\L+6,\g'}\le \\
(\cte)^n\L^2(\frac1{\g-\g'})^{(n-1)d+1} [\sum_{1\le k\le
n}\prod_{\begin{subarray}{c} 1\le j\le n\\ j\not =k\end{subarray}}
\a{A_j}_{\g_j}\l{A_k}_{\L,\g_k}],
\end{array}
\end{equation}
where all $\g_1,\hdots,\g_n$ are $=\g$ except one which is $=\g'$.

{\it Linear differential equation.} Consider the linear system
$$\left\{\begin{array}{l}
\frac{d}{dt}X=A(t)X\\
X(0)=I.\end{array}\right.$$
where $A(t)$ is T\"oplitz-Lipschitz with exponential decay.
The solution verifies
$$X(t_0)=I+
\sum_{n=1}^{\infty} \int_0^{t_0}\int_0^{t_1}\hdots \int_0^{t_{n-1}}
A(t_1)A(t_2)\hdots A(t_n)dt_n\hdots dt_2dt_1.$$
Using (\ref{e29}) we get
for $\g'<\g$
\begin{equation}\label{e210}
\begin{array}{c}
\l{X(t)-I}_{\L+6,\g'}\lsim\\
\L^2(\frac1{\g-\g'})|t|
\exp(\cte(\frac1{\g-\g'})^{d}|t|\al(t))
\sup_{|s|\le|t|}\l{A(s)}_{\L,\g},
\end{array}
\end{equation}
where
$$\al(t)=\sup_{0\le |s|\le |t|}\a{A(s)}_\g.$$

\subsection{T\"oplitz-Lipschitz matrices ($d\ge2$)}
\label{ss24}
\

\noindent
Let
$$A:\LL\times\LL\to gl(2,\C)$$
be a matrix. We say that $A$ is 1-T\"oplitz if all T\"oplitz-limits
$A(\pm,c)$ exist,  and we define, inductively, that $A$ is
$n$-T\"oplitz if all T\"oplitz-limits $A(\pm,c)$ are
$(n-1)$-T\"oplitz. We say that $A$ is {\it T\"oplitz} if it is
$(d-1)$-T\"oplitz.

In Section \ref{ss23} we have defined
$\l{A}_{\L,\g}$ which we shall now denote by
$${}^1\!\!\l{A}_{\L,\g}.$$
We define, inductively,
$${}^n\!\!\l{A}_{\L,\g}\sup_{c\in\Z^d}( {}^{n-1}\!\!\l{A(+,c)}_{\L,\g},{}^{n-1}\!\!
\l{A(-,c)}_{\L,\g})$$ ($c=0$ is allowed and $A(\pm,0)=A$) and we
denote
$$\sc{A}_{\L,\g}={}^{d-1}\!\!\l{A}_{\L,\g}.$$
The matrix $A$ is {\it T\"oplitz-Lipschitz} if it is T\"oplitz at
$\i$ and $\l{A}_{\L,\g}<\i$ for some $\L,\g$.

Proposition \ref{p22}, Corollary \ref{c23} and
(\ref{e27}-\ref{e210}) remain valid with this norm in any dimension
$d$.

\section{Functions with T\"oplitz-Lipschitz property}
\label{s3}

\subsection{T\"oplitz-Lipschitz property}
\label{ss31}
\

\noindent Let $\OO^{\g}(\s)$ be the set of vectors in the complex
space $l^2_\ga(\LL,\C^2)$ of norm less than $\s$, i.e.
$$\OO^\g(\s)
=\b{\z\in\C^{\LL}\times\C^{\LL}:\aa{\z}_\g<\s}.$$ Our functions
$f:\OO^0(\s)\to \C$ will be defined and real analytic on the domain
$\OO^0(\s)$. \footnote{The space $l^2_\ga(\LL,\C^2)$ is  the
complexification of  the space $l^2_\ga(\LL,\R)$ of real sequences.
``real analytic'' means that it is a holomorphic function which is
real on $\OO^0(\s)\cap l^2_\ga(\LL,\R)$.}

Its first differential
$$l^2_0(\LL,\C^2)\ni\hat\z\mapsto \sc{\hat\z,\partial_\z f(\z)}$$
defines a unique vector $\partial_\z f(\z)$ (the gradient with
respect to the paring  $\sc{\ ,\ }$), and its second differential
$$l^2_0(\LL,\C^2)\ni\hat\z\mapsto \sc{\hat\z,
\partial_\z^2 f(\z)\hat\z}$$
defines a unique symmetric matrix $\partial_\z^2 f(\z):
\LL\times\LL\to gl(2,\C)$ (the Hessian with respect to the paring
$\sc{\ ,\ }$). A matrix $A:\LL\times\LL\to gl(2,\C)$ is symmetric if
$${}^t\! A_a^b=A_b^a.$$

We say that $f$ is {\it T\"oplitz at} $\infty$ if
the vector $\partial_\z f(\z)$ lies in $l^2_0(\LL,\C^2)$ and
the matrix
$\partial_\z^2 f(\z)$ is T\"oplitz at $\infty$ for all
$\z\in\OO^0(\s)$. We define the norm
$$[f]_{\L,\g,\s}$$
to be the smallest $C$ such that
$$\left\{\begin{array}{ll}
\a{f(\z)}\le C & \forall \z\in \OO^0(\s)\\
\aa{\partial_\z f(\z)}_{\g'}\le \frac1{\s}C &
\forall \z\in \OO^{\g'}(\s),\ \forall \g'\le\g,\\
\l{\partial_\z^2 f(\z)}_{\L,\g'}\le \frac1{\s^2}C &
\forall \z\in \OO^{\g'}(\s),\ \forall \g'\le\g.
\end{array}\right.$$

\begin{Prop}\label{p31}
\begin{itemize}
\item[(i)]
$$[fg]_{\L,\g,\s}
\lsim
[f]_{\L,\g,\s}[g]_{\L,\g,\s} .$$
\item[(ii)] If $g(\z)=\sc{c,\partial_\z f(\z)}$, then
$$[g]_{\L,\g,\s'}\lsim \frac1{\s-\s'}\aa{c}_{\g}
[f]_{\L,\g,\s}$$
for $\s'<\s$.
\item[(iii)] If $g(\z)=\sc{C\z,\partial_\z f(\z)}$, then
$$
\begin{array}{ll}
[g]_{\L+3,\g',\s'}\lsim &
\big((1+\frac{\s'}{\s-\s'})(\frac1{\g-\g'})^{d+m_*}\a{C}_{\g}\\
&+\L^2(\frac1{\g-\g'})^{d+1}\l{C}_{\L,\g}\big) [f]_{\L,\g,\s}
\end{array}$$
for $\s'<\s$ and $\g'<\g$.
\end{itemize}
\end{Prop}

\begin{proof}
We have
$$\begin{array}{l}
fg(\z)=f(\z)g(\z)\\
\partial_\z fg(\z)=f(\z)\partial_\z g(\z)+
\partial_\z f(\z)g(\z)\\
\partial_\z^2 fg(\z)=f(\z)\partial_\z^2 g(\z)+
\partial_\z^2 f(\z)g(\z)+
2(\partial_\z f(\z)\otimes\partial_\z g(\z)).
\end{array}$$
(i) now follows from (\ref{e27}).

For $\z\in\OO^0(\s')$ we have
$$|g(\z)|\le \aa{c}_0\aa{\partial_\z f(\z)}_0\le
 \aa{c}_0\frac1{\s}\al,$$
where
$\al=[f]_{\L,\g,\s}$.

Let $\z\in\OO^{\g'}(\s')$ and $h(z)= \partial_\z f(\z+zc)$. $h$ is a
holomorphic function (with values in the Hilbert-space
$l^2_{\g'}(\LL,\C^2)$) in the disk
$\a{z}<\frac{\s-\s'}{\aa{c}_{\g'}}$ and
$$\aa{h(z)}_{\g'}\le\frac1{\s}\al.$$
Since $\partial_\z g(\z)= \partial_z h(0)$,
we get by a Cauchy estimate that
$$\aa{\partial_\z g(\z)}_{\g'}\le\frac1{\s'}(\frac{\s'}{\s}
\frac1{\s-\s'}\aa{c}_{\g'}\al).$$

Let $\z\in\OO^{\g'}(\s')$ and $k(z)= \partial_\z^2 f(\z+zc)$. $k$ is
a holomorphic  function (with values in the  Banach-space of
matrices with the norm  $\l{\cdot}_{\g',\L}$) in the disk
$\a{z}<\frac{\s-\s'}{\aa{c}_{\g'}}$ and
$$\l{k(z)}_{\L,\g'}\le\frac1{\s^2}\al.$$
Since $\partial_\z^2 g(\z)= \partial_\z k(0)$,
we get by a Cauchy estimate that
$$\l{\partial_\z g(\z)}_{\L,\g'}\le
(\frac1{\s'})^2((\frac{\s'}{\s})^2\frac1{\s-\s'}\aa{c}_{\g'}\al).$$
This proves (ii).

To see (iii) we replace $c$ by $C\z$ and notice that
$$\partial_\z g(\z)= \partial_z h(0)+{}^tC\partial_\z f(\z)$$
and
$$\partial_\z^2 g(\z)= \partial_z k(0)+{}^tC\partial_\z^2 f(\z)
+{}^t\partial_\z^2 f(\z)C.$$ $\partial_z h(0)$ and $\partial_z k(0)$
are estimated as above and $\aa{C\z}_{\g'}$ with Young's inequality
(\ref{e22}). The matrix products are estimated by (\ref{e28}).
\end{proof}

\subsection{Truncations}
\label{ss32}
\

\noindent
Let $Tf$ be the Taylor polynomial of order $2$ of $f$ at $\z=0$.

\begin{Prop}\label{p32}
\begin{itemize}
\item[(i)]
$$[Tf]_{\L,\g,\s}
\lsim
[f]_{\L,\g,\s} .$$
\item[(ii)]
$$[f-Tf]_{\L,\g,\s'}
\lsim (\frac{\s'}{\s})^3\frac{\s}{\s-\s'}
[f]_{\L,\g,\s}.$$
\end{itemize}
\end{Prop}

\begin{proof}
Let $\z\in\OO^0(\s')$ and let $g(z)=f(z\z)$. Then $g$ is a real
holomorphic function in the disk of radius $\frac{\s}{\s'}$ and
bounded by $\al=[f]_{\L,\g,\s}$. Since $Tf(z\z)=g(0)+g'(0)z+\frac12
g''(0)z^2$ we get by a Cauchy estimate that
$$|(f-Tf)'\z)|=|g(1)-g(0)-g'(0)-\frac12\,g''(0)|
\le(\frac{\s'}{\s})^3 \frac{\s}{\s-\s'}\al.$$

Let $\z\in\OO^{\g'}(\s')$ and let $h(z)=\partial_\z f(z\z)$. Then
$h$ is a holomorphic function in the disk of radius $\frac{\s}{\s'}$
and bounded by $\frac{\al}{\s}$. Since $\partial_\z
Tf(\z)=h(0)+h'(0)z$ we get by a Cauchy estimate that
$$\aa{\partial_\z (f-Tf)(\z)}_{\g'}\le(\frac{\s'}{\s})^2\frac{\s}{\s-\s'}
\frac{\al}{\s} .$$

Let $\z\in\OO^{\g'}(\s')$ and let $k(z)=\partial_\z^2 f(z\z)$. Then
$k$ is a holomorphic function  in the disk of radius
$\frac{\s}{\s'}$ and bounded by $\frac{\al}{\s^2}$. Since $
\partial_\z^2 Tf(\z)=k(0)$ we get by a Cauchy estimate that
$$\l{\p_\z^2(f-Tf)(\z)}_{\L,\g'}\le(\frac{\s'}{\s})
\frac{\s}{\s-\s'}\frac{\al}{\s^2}.$$
This gives (ii).

The first statement is obtained by taking $\s'=\frac12\s$. Since $f$
is a quadratic polynomial it satisfies the same (modulo a constant)
estimate on $\s$ as on $\frac12\s$.
\end{proof}

\subsection{Poisson brackets}
\label{ss33}
\

\noindent The Poisson bracket of two functions $f$ and $g$  is
defined by
$$ \b{f,g}(\z) \sc{\partial_\z f(\z),J\partial_\z g(\z)}.$$

\begin{Prop}\label{p33}
\begin{itemize}
\item[(i)]
If $g$ is a quadratic polynomial, then
$$
[\b{f,g}]_{\L+3,\g',\s'} \lsim [\frac1{\s_1\s_2}+
\L^2(\frac1{\g-\g'})^{d+1}(\frac{\s'}{\s_1\s_2})^2] [f]_{\L,\g,\s_1}
[g]_{\L,\g,\s_2},$$ for $0<\s_1-\s'\approx \s_1,\ 0<\s_2 -\s'\approx
\s_2$ and $\g'<\g$.
\item[(ii)]
If $g$ is a quadratic polynomial and $f(\z)=\sc{\z,A\z}$, then
$$
[\b{f,g}]_{\L+3,\g',\s'} \lsim
\big[(\frac1{\g-\g'})^{d+m_*}\frac1{\s_1^2}+
\L^2(\frac1{\g-\g'})^{d+1}\frac1{\s_1^2})\big] [f]_{\L,\g,\s_1}
[g]_{\L,\g,\s_2},$$ for $0< \s_1-\s'\approx \s_1,\ 0<\s_2-\s'\approx
\s_2$  and $\g'<\g$.
\end{itemize}
\end{Prop}

\begin{proof}
We have
$$\partial_\z \b{f,g}(\z)= \partial_\z^2 f(\z)J\partial_\z g(\z)
- \partial_\z^2 g(\z)J\partial_\z f(\z)$$ and $\partial_\z^2
\b{f,g}(\z)$ is the symmetrization of the infinite matrix
%$$\begin{array}{l}
%%\partial_\z^2 \b{f,g}(\z) %\partial_\z^3 f(\z)J\partial_\z g(\z)
%+\partial_\z^3 g(\z)J\partial_\z f(\z)+\\
%\partial_\z^2 f(\z)J\partial_\z^2 g(\z)+
%{}^t\partial_\z^2 f(\z)J\partial_\z^2 g(\z).
%\end{array}$$
$$
\partial_\z^3 f(\z)J\partial_\z g(\z)
-\partial_\z^3 g(\z)J\partial_\z f(\z)+
\partial_\z^2 f(\z)J\partial_\z^2 g(\z)+
\partial_\z^2 f(\z)J\partial_\z^2 g(\z).
$$

For $\z\in\OO^0(\s')$ we get, by Cauchy-Schwartz, that
$$\a{\b{f,g}(\z)}\le \aa{\partial_\z f(\z)}_0\aa{\partial_\z g(\z)}_0
\le (\frac{\al\be}{\s_1\s_2}),$$
where
$\al=[f]_{\L,\g,\s_1}$ and $\be=[g]_{\L,\g,\s_2}$.

For $\z\in\OO^{\g'}(\s')$, let $h(z)=\partial_\z f(\z+ zJ\partial_\z
g(\z))$. For $\a{z}<\frac{\s_1-\s'}{\aa{\p_\z g(\z)}_{\g'}}$ we have
$$\aa{h(z)}_{\g'}\le\frac{\al}{\s_1}.$$
Since $\partial_z h(0)\partial_\z^2 f(\z)J\partial_\z g(\z)$ and $\s_1-\s'\approx\s_1$,
we get by a Cauchy estimate that
$$\aa{\partial_\z^2 f(\z)J\partial_\z g(\z)}_{\g'}
\lsim\frac1{\s_1^2\s_2}\al\be.$$ The same estimate holds with $f$
and $g$ interchanged.

For $\z\in\OO^{\g'}(\s')$, let $k(z)=\partial_\z^2
f(\z+zJ\partial_\z g(\z))$. By a Cauchy-estimate we get  as above
that
$$\l{\partial_\z^3 f(\z)J\partial_\z g(\z)}_{\L,\g'}
\lsim \frac1{\s_1^3\s_2}\al\be.$$ The same estimate holds with $f$
and $g$ interchanged.

Finally, for $\z\in\OO^{\g'}(\s')$ we get by (\ref{e28}) that
$$\l{\partial_\z^2 f(\z)J\partial_\z^2 g(\z)}_{\L+3,\g'}
\lsim
\L^2(\g-\g')^{-d-1}
\l{\partial_\z^2 f(\z)}_{\L,\g'} \l{\partial_\z^2 g(\z)}_{\L,\g}.$$
By hypothesis we have
$$\l{\partial_\z^2 g(\z)}_{\L,\g}\le \frac{\be}{\s_2^2}$$
for $\z$ only in $\OO^{\g}(\s')$. But since $g$ is quadratic,
$\partial_\z^2 g(\z)$ is independent of $\z$ and, hence, this also
holds in the larger domain $\z\in\OO^{\g'}(\s')$. The symmetrized
matrices satisfy the same estimates, and (i) is established.

The second part follows directly from Proposition \ref{p31}(iii).
\end{proof}

\subsection{The flow map}
\label{ss34}
\

\noindent
Consider the linear system
$$\dot \z=J\p_\z f_t(\z)$$
where $f_t(\z)=\sc{\z,a_t}+\frac12\sc{\z,A_t\z}$, and let
$$\al(t)=\sup_{|s|\le|t|}\a{A_s}_{\g}\quad
\text{and}\quad
\be(t)=\sup_{|s|\le|t|}\aa{a_s}_{\g'}.$$
Consider the non-linear system
$$\dot z= g(\z,z)$$
where $g(\z,z)$ is real analytic in
$\OO^0(\s) \times {\mathbb{D}}(\m)$.
${\mathbb{D}}(\m)$ is the disk of radius $\m$ in $\C$. Let $0<\mu'<\mu$.

\begin{Prop}\label{p34}
\begin{itemize}
\item[(i)]
The flow map of the linear system
has the form
$$\z_t:\z\mapsto \z+b_t+B_t\z,$$
and for $\g'<\g$
$$
\begin{array}{c}
\aa{\z_t(\z)-\z}_{\g'}
\lsim\\
(\frac1{\g-\g'})^{m_*}\big[e^{\cte(\frac1{\g-\g'})^d|t|\al(t)}|t|\be(t)
+ [e^{\cte(\frac1{\g-\g'})^d|t|\al(t)}-1\big]\aa{\z}_{\g'}]
\end{array}$$
and
$$
\begin{array}{c}
\l{B_t}_{\L+6,\g'}
\lsim\\
\L^2(\frac1{\g-\g'})|t|
e^{\cte(\frac1{\g-\g'})^d|t|\al(t)}
\sup_{|s|\le|t|}\l{A_s}_{\L,\g}.
\end{array}$$
\item[(ii)]
For $|z|<\m'$, the flow of the non-linear system is defined for
$|t|\le\frac{\m-\m'}{2\ep}$ and
$$
[z_t(\cdot,z)-z]_{\L,\g,\s}\lsim (1+\frac{\m-\m'}{\ep}
(e^{\cte|t|\frac1{\m-\m'}\ep}-1))^2\ep\,,
$$
where
$$
\ep=\sup_{z\in{\mathbb D}(\m)}[g(\cdot,z)]_{\L,\g,\s}\le1.
$$
\end{itemize}
\end{Prop}

\begin{proof}
(i) We have
$$b_{t}\sum_{n=1}^{\infty} \int_0^{t}\hdots \int_0^{t_{n-1}}
JA_{t_1}\hdots JA_{t_{n-1}}Ja_{t_n}dt_ndt_{n-1}\hdots dt_1$$
and
$$B_{t}\sum_{n=1}^{\infty} \int_0^{t}\hdots \int_0^{t_{n-1}}
JA_{t_1}\hdots JA_{t_n}dt_n\hdots dt_1.$$

By (\ref{e24}) we have
$$\a{B_t}_{\g'}
\lsim(\g-\g')^{d}(\d(t)-1),\quad
\d(t)\exp(\cte(\g-\g')^{-d}|t|\al(t))$$
and by (\ref{e22}) we have
$$
\aa{B_t\z}_{\g'}
\lsim
(\frac1{\g-\g'})^{m_*}(\d(t)-1)\aa{\z}_{\g'}.$$
By (\ref{e22}+\ref{e24}) we have
$$
\aa{b_t}_{\g'}
\lsim (\frac1{\g-\g'})^{m_*}\d(t)|t|\be(t).$$

By (\ref{e210}) we have
$$\l{B_t}_{\L+6,\g'}
\lsim
\L^2(\g-\g')^{-1}\d(t)
\sup_{|s|\le|t|}\l{A_s}_ {\L,\g}.$$

The proof of (ii) easier. We have
$$\p_\z\dot z_t=\p_\z g(\hdots)+\p_z g(\hdots)\p_\z z_t$$
which implies that
$$\p_\z z_t=\int_0^t
e^{\int_s^t\p_z g(\z,z_r)dr}\p_\z g(\z,z_s)ds$$
This is easy to estimate.

We also have
$$\p_\z^2\dot z_t=\p_\z^2 g(\hdots)+
\p_z\p_\z g(\hdots)\otimes\p_\z z_t+
\p_z g(\hdots)\p_\z^2 z_t$$
which is treated in the same way.
\end{proof}

\begin{Rem}  The same result holds for $z=(z_1,\hdots,z_n)
\in{\mathbb{D}}(\m)^n$ and
$g=(g_1,\hdots,g_n)$.
\end{Rem}

\begin{Rem}
If $|t|\le1$ and
$$ \sup_{|s|\le|t|}\a{A_s}_{\g}\lsim(\g-\g')^d,$$
then
$$\aa{\z_t(\z)-\z}_{\g'}
\lsim(\frac1{\g-\g'})^{m_*}\sup_{|s|\le|t|}\aa{a_s}_{\g'}+
(\frac1{\g-\g'})^{m_*+d}
\sup_{|s|\le|t|}\a{A_s}_{\g}\aa{\z}_{\g'}$$
and
$$\l{B_t}_{\L+6,\g'}
\lsim
\L^2(\frac1{\g-\g'})\sup_{|s|\le|t|}\l{A_s}_{\L,\g}.$$

If $|t|\le1$ and
$$\ep=\sup_{z\in{\mathbb D}(\m)}[g(\cdot,z)]_{\L,\g,\s}
\lsim \m-\m',$$
then
$$[z_t(\cdot,z)-z]_{\L,\g,\s}\lsim \ep.$$
\end{Rem}

\subsection{Compositions}
\label{ss35}
\

\noindent
Let $f(\z,z)$ be a real analytic function on
$ \OO^0(\s) \times{\mathbb D}(\m) $ and
$$\sup_{z\in{\mathbb D}(\m)}[f(\cdot,z)]_{\L,\g,\s}<\i.$$
Let $0<\s'<\s$, $0<\mu'<\mu$ and
$$\Phi(\z,z)=\z+ b(z)+B(z)\z$$
with
$$\aa{b(z)+B(z)\z}_{\g'}<\s-\s',\quad
\forall(\z,z)\in\OO^{\g'}(\s')
 \times {\mathbb D}(\m')$$
for all $\g'\le\g$.
This implies that
$$\Phi(\cdot,z):\OO^{\g'}(\s')\to \OO^{\g'}(\s),\quad\forall\ \g'\le\g,
\quad\forall z\in {\mathbb D}(\m').$$
Let $g(\z,z)$ be a real holomorphic function on $ \OO^0(\s')
 \times{\mathbb D}(\m')$ such
that
$$|g|\le\frac12(\m-\m').$$

\begin{Prop}\label{p35}
For all $z\in{\mathbb D}(\m')$ and $\g'<\g$
$$
\begin{array}{c}
[f(\Phi(\cdot,z),z+g(\cdot,z))]_{\L+6,\g',\s'}
\lsim\\
\max(1,\al, \L^2(\frac1{\g-\g'})\al^2) \sup_{z\in{\mathbb
D}(\m)}[f(\cdot,z)]_{\L,\g,\s},
\end{array}$$
where
$$
\al=\frac1{\m-\m'} \sup_{z\in{\mathbb D}(\m)}
[g(\cdot,z)]_{\L,\g,\s'}+(\frac1{\g- \g'})^{d+m_*}\sup_{z\in{\mathbb
D}(\m)}\l{B}_{\L,\g}.$$
\end{Prop}

\begin{proof}
Let
$\ep=\sup_{z\in{\mathbb D}(\m)}[f(\cdot,z)]_{\L,\g,\s}$
and
$\be=\sup_{z\in{\mathbb D}(\m')}[g(\cdot,z)]_{\L,\g,\s'}$.

Let $h(\z,z)=f(\Phi(\z,z),z+g(\z,z))$. Then
$$\p_\z h = \p_z f(\hdots)\p_\z g
+{}^tB\p_\z f(\hdots)$$
and
$$\begin{array}{ll}
\p_\z^2 h = &
\p^2_z f(\hdots)(\p_\z g\otimes\p_\z g)
+
\p_z f(\hdots) \p_\z^2 g
+\\
& 2\,{}^t\!B(\p_\z\p_z f(\hdots)\otimes\p_\z g) + {}^t\!B\p_\z^2
f(\hdots)B.
\end{array}$$

For $(\z,z)\in \OO^0(\s') \times{\mathbb D}(\m') $ we get:
$\a{h(\z)}\le \ep$.

For $(\z,z)\in \OO^{\g'}(\s') \times{\mathbb D}(\m')$ we get:
$$\aa{\p_z f(\hdots)\p_\z g}_{\g'}\a{\p_z f(\hdots)}\aa{\p_\z g}_{\g'}
\lsim
(\frac1{\m-\m'})\ep\frac{\be}{\s'};$$
$$\aa{{}^tB\p_\z f(\hdots)}_{\g'}\lsim
(\frac1{\g-\g'})^{d+m_*}\a{B}_{\g}\frac{\ep}{\s}$$
by Young's inequality (\ref{e22}).

For $(\z,z)\in \OO^{\g'}(\s') \times{\mathbb D}(\m')$  we get:
$$\l{\p^2_z f(\hdots)\p_\z g\otimes\p_\z g}_{\L,\g'}
\lsim
(\frac1{\m-\m'})^2\ep(\frac{\be}{\s'})^2$$
by (\ref{e27});
$$\l{\p_z f(\hdots)\p_\z^2 g}_{\L,\g'}
\lsim
(\frac1{\m-\m'})\ep(\frac{\be}{(\s')^2});$$
$$\l{{}^tB(\p_\z\p_z f(\hdots)\otimes\p_\z g)}_{\L+3,\g'}
\lsim
\L^2(\frac1{\g-\g'})^{d+1}\l{B}_{\L,\g}
(\frac1{\m-\m'})\ep\frac{\be}{\s\s'}$$
by (\ref{e27}-\ref{e28});
$$\l{{}^tB\p_\z^2 f(\hdots)B}_{\L+6,\g'}
\lsim
\L^2(\frac1{\g-\g'})^{2d+1}\l{B}_{\L,\g}^2\frac{\ep}{\s^2}$$
by (\ref{e29}).
\end{proof}

\begin{Rem} The same result holds for $z=(z_1,\hdots,z_n)\in
 {\mathbb{D}}(\m)^n$ and
$g=(g_1,\hdots,g_n)$.
\end{Rem}

%\begin{Rem}
%If, for $z\in {\mathbb{D}}(\m')$,
%$$[g(\cdot),z]_{\L,\g,\s'}\lsim(\m-\m')\min(1,
%\sqrt{\g-\g'})\L^2$$
%and
%$$\l{B(z)}_{\L,\g}\lsim (\g-\g')^{d+m_*}\L^2,$$
%then
%$$[f(\Phi(\cdot,z),z+g(\cdot,z)]_{\L+6,\g',\s'}
%\lsim\L^6\ep.$$
%\end{Rem}

\bigskip
\bigskip

\centerline{PART II. THE HOMOLOGICAL EQUATIONS}

\bigskip
In this part we consider scalar-valued matrices  $Q:\LL\times\LL\to
\C$ which we identify with $gl(2,\C)$-valued matrices through the
identification
$$Q_a^b=Q_a^bI.$$

We will only consider the Lipschitz domains $D_\L^+(c)$ which we
denote by $D_\L(c)$.

We define the block decomposition $\EE_\D$ together with the blocks
$[ \ \cdot\ ]_\D$ and the bound $d_\D$ of the block diameter. We
consider parameters $U\sbs \R^{\AA}$, $\AA=\Z^d\sm\LL$, and define
the norms $\a{ \ \cdot\ }_{\left\{\begin{subarray}{l} \g\\
U\end{subarray}\right\}}$ and $\l{ \ \cdot\
}_{\left\{\begin{subarray}{l} \L,\g\\ U\end{subarray}\right\}}$.

\section{Decomposition of $\LL$}
\label{s4} In this section $d\ge2$. For a non-negative integer $\D$
we define an {\it equivalence  relation} on $\LL$ generated by the
pre-equivalence relation
$$a\sim b \iff
\left\{\begin{array}{l}
\a{a}^2=\a{b}^2\\
\a{a-b}\le\D.
\end{array}\right.
$$
Let $[a]_{\D}$ denote the equivalence class ({\it block}) of $a$,
and let $\EE_{\D}$ be the set of equivalence classes. It is trivial
that each block $[a]$ is finite with cardinality
$$\lsim\a{a}^{d-1}$$
that depends on $a$. But there is also a uniform $\D$-dependent
bound. Indeed, let $d_\D$ be the supremum of all block diameters. We
will see (Proposition~\ref{p41})
$$d_\D\lsim\D^{\frac{(d+1)!}2}.$$

$\D$ will be fixed in this section
and we will write $[ \ \cdot\ ]$ for $[ \ \cdot\ ]_{\D}$.

\subsection{Blocks}
\label{ss41}
\

\noindent
For any $X\sbs\Z^d$ we define its {\em rank} to be the dimension
of the smallest affine subspace in $\R^d$ containing $X$.

\begin{Prop}
\label{p41}
 Let $c\in\Z^d$ and rank$[c]=k$, $k=1,\dots,d$. Then
the diameter of $[c]$ is
$$
\lsim\ \D^{\frac{(k+1)!}2}.$$
\end{Prop}

\begin{proof}Let $\D_j,\ j\ge 1$ be an increasing sequence  of
numbers.

Assume that for any $1\le l\le k$
$$(\ast)_l\qquad \text{rank}(B_{\D_l}(c)\cap[c])\ge
 l\quad \forall c\in[c],$$
where $B_r(c)$ is the ball of radius $r$ centered at $c$.
This means that for any $c\in[c]$, there exist linearly
independent vectors $a_1,\dots,a_l$ in $\Z^d$ such that
$$ c+a_j\in[c]\ \text{and}\ \a{a_j}\le \D_l,\quad 1\le j\le l.$$

$(\ast)_l$ implies that the $\perp$ projection $\tilde c$ of $c$
onto $\sum\R a_j$ verifies
$$(\ast\ast)\qquad \a{\tilde c}\lsim
\left\{\begin{array}{ll}
\D_l & l=1\\
\D_l^{l+1} & l\ge 2.
\end{array}\right.$$

\begin{proof}In order to see this we observe that,
 since $|c+a_j|^2=|c|^2$ for each $j$,
the (row) vector $c$
verifies
$$cM=-\frac12(\a{a_1}^2\dots\a{a_l}^2),$$
where $M$ is the $d\times l$-matrix whose columns are
${}^t\!a_1,\dots {}^t\!a_l$. Now there exists an orthogonal matrix
$Q$ such that
$$QM=\left(\begin{array}{l} B\\0\end{array}\right),$$
where $B$ is an invertible $l\times l$-matrix. We have
$$(\det B)^2=\det({}^tBB)=\det({}^tMM)\ge1,$$
and (the absolute values of) the entries of $B$ are bounded by
$\lsim\D_l$.

Define now $x$ by
$$\left\{\begin{array}{l}
(x_1\dots x_l)=-\frac12(\a{a_1}^2\dots\a{a_l}^2)B^{-1}\\
x_{l+1}=\dots=x_d=0,
\end{array}\right.$$
and $y=xQ$.  Then $c-y\perp\sum\R a_j$, so $|\tilde c|\le|y|$. An
easy computation gives
$$\a{y}=\a{x}\lsim \D_l^{l+1}\quad\textrm{and}\quad  \lsim\D_1\ (\text{if}\ l=1).$$
\end{proof}

We shall now determine $\D_l$ so that $(\ast)_l$ holds. This  will
be done by induction on $l$. For $l=1$ $\D_1=\D$ works, so let us
assume that $(\ast)_l$ holds for some $1\le l < k$. If
$(\ast)_{l+1}$ does not hold, it is violated for some $c$. Let us
fix this $c\in[c]$, and let $X$ be the real subspace generated by
$(B_{\D_{l+1}}(c))\cap[c])-c$. $X$ has rank $=l$.

For any $b\in[c]$ with $\a{b-c}\le \D_{l+1}-\D_l$ we have
$$B_{\D_l}(b)\cap[c]\sbs B_{\D_{l+1}}(c)\cap[c].$$
By the induction assumption the $\perp$ projection $\tilde b$ of $b$
onto $X$ verifies $(\ast\ast)$.

Take now $b\in[c]$ such that $\D_{l+1}-\D_l-\D\le
\a{b-c}\le(\D_{l+1}-\D_l)$
---  such
a $b$ exists since rank of $[c]$ is $\ge l+1$. Since $b-c$ is
parallel to $X$  we have
$$\D_{l+1}-\D_l-\D\le \a{b-c}=|\tilde b-\tilde c|\lsim
\left\{\begin{array}{ll}
\D_l & l=1\\
\D_l^{l+1} & l\ge 2.
\end{array}\right.$$
So if we take $\D_{l+1}\approx$ the RHS, then the assumption  that
$(\ast)_{l+1}$ does not hold leads to a contradiction. Hence with
this choice $(\ast)_{l}$ holds for all $l\le k$.

To conclude we observe now that $[c]\sbs c+X$ where $X$ is a
subspace of dimension $k$. Clearly the diameter of $[c]$ is the same
as the diameter of its $\perp$ projection onto $X$, and, by
$(\ast\ast)$, the diameter of the projection is $\le\D_k$.
\end{proof}

We say that $[a]$ and $[b]$ have the same {\it block-type} if
there are $a'\in[a]$ and $b'\in[b]$ such that
$$[a]-a'=[b]-b'.$$
It follows from the proposition that there are only finitely many
block-types. We say that the block-type of $[a]$ is  {\it orthogonal
to} $c$ if
$$[a]-a\perp c.$$

{\it Description of blocks when $d=2,3$.}
For $d=2$, we have outside $\b{\a{a}:\le d_\D\approx \D^3}$
\begin{itemize}
\item[$\star$]  rank[a]=1 if, and only if, $a\in\frac b2+b^{\perp}
$ for some $0<\a{b}\le\D$
 --  then $[a]=\b{a,a-b}$ ;
\item[$\star$]   rank[a]=0  --  then $[a]=\b{a}$.
\end{itemize}

For $d=3$, we have outside $\b{\a{a}:\le d_\D\approx\D^{12}}$
\begin{itemize}
\item[$\star$]  rank[a]=2 if, and only if, $a\in\frac b2+b^{\perp}
\cap\frac c2+c^{\perp}$ for some $0<\a{b},\a{c}\le2\D$ linearly
independent -- then $[a]\sps\b{a,a-b,a-c}$; \item[$\star$]
rank[a]=1 if, and only if, $a\in\frac b2+b^{\perp}$ for some
$0<\a{b}\le\D$  -- then $[a]=\b{a,a-b}$; \item[$\star$]
rank[a]=0 --  then $[a]=\b{a}$.
\end{itemize}

\subsection{Neighborhood at $\i$.}
\label{ss42}

\begin{Prop}
\label{p42}
For any $\a{a}\gsim \L^{2d-1}$, there exist
$c\in\Z^d$,
$$0<\a{c}\lsim \L^{d-1},$$
such that
$$\a{a}\ge\L(\a{a_c}+\a{c})\a{c},\ \sc{a,c}\ge0.$$
($a_c$ is the lattice element on $a+\R c$ closest to the origin.)
\end{Prop}

\begin{proof}
For all $K\gsim 1$ there is a $c\in\Z^d\cap\b{\a{x}\le K}$ such that
$$\d=dist(c,\R a)\le C_1(\frac1K)^{\frac1{d-1}}$$
where $C_1$ only depends on $d$.

To see this we consider the segment $\G=[0,\frac{K}{\a{a}}a]$ in
$\R^d$ and a tubular neighborhood $\G_{\ep}$ of radius $\ep$:
$$\text{vol}(\G_{\ep})\approx K\ep^{d-1}.$$
The projection of $\R^d$ onto $\T^d$ is locally injective and
locally volume-preserving. If $\ep\gsim(\frac1K)^{\frac1{d-1}}$,
then the projection of $\G_{\ep}$ cannot be injective (for volume
reasons), so there are two different points $x,x'\in\G_{\ep}$ such
that
$$x-x'=c\in Z^d\sm0 .$$

Then
$$\a{a_c}\lsim\ \frac{\a{a}}{\a{c}}\d.$$

Now
$$\L(\a{a_c}+\a{c})\a{c}\le 2\L K^2
+C_2\frac{\L}{K^{\frac1{d-1}}}\a{a}.$$ If we choose
$K=(2C_2\L)^{d-1}$, then this is $\le\a{a}$.
\end{proof}

\begin{Cor}\label{c43}
For any $\L,N>1$, the subset
$$
\b{\a{a}+\a{b}\gsim\L^{2d-1}}\cap\b{\a{a-b}\le N}\sbs
\Z^d\times\Z^d$$ is contained in
$$\bigcup_{0<\a{c}\lsim \L^{d-1}}D_{\O}(c)$$
for any
$$\O\le\frac{\L}{N+1}-1.$$
\end{Cor}

\begin{proof}
Let $\a{a}\gsim\L^{2d-1}$. Then there exists  $0<\a{c}\lsim\L^{d-1}$
such that $\a{a}\ge\L(\a{a_c}+\a{c})\a{c}$. Clearly (because
$d\ge2$)
$$\frac{|a|}{|c|}\ge2\L^2\ge 2\O^2.$$

If we write $a=a_c+tc$ then $b=a_c+b-a+tc.$ According to
Lemma~\ref{l21}(iv)
$$|b|\ge\O(|a_c+b-a|+|c|)|c|,$$
and moreover
$$ \frac{|b|}{|c|}\ge \frac{|a|}{|c|}-N\ge 2\L^2-N\ge 2\O^2.$$
\end{proof}

\begin{Rem} This corollary is essential. It says that
any neighborhood
$$\b{(a,b): |a-b|\le N}\sbs \Z^d\times\Z^d$$
of the diagonal, outside some {\it finite} set, is covered
by {\it finitely} many Lipschitz domains.
\end{Rem}

\subsection{Lines $(a+\R c)\cap\Z^d$}
\label{ss43}

\begin{Prop}
\label{p44}
\begin{itemize}
\item[(i)] If $[a+tc]=[b+tc]$ for all $t>\!\!>1$, then
$[a+tc]=[b+tc]$ for all $t$.
\item[(ii)]
$[a+tc]-(a+tc)$ is constant and $\perp$ to $c$ for all
$t$ such that
$$\a{a+tc}\ge d_\D^2(\a{a_c}+\a{c})\a{c}.$$
\end{itemize}
\end{Prop}

\begin{proof}
To prove (i) we observe that
$$\a{a+tc}=\a{b+tc}\quad \forall t>\!\!>1,$$
which clearly implies that
$$\a{a+tc}=\a{b+tc}\quad \forall t.$$
If $\a{a-b}\le\D$ then this implies that $[a+tc]=[b+tc]$ for all
$t$.  Otherwise, for all $t>\!\!>1$ there is a $d_t\notin\b{a,b}$
such that
$$[d_t+tc]=[a+tc].$$
Since the diameter of each block is $\le d_\D$, it follows that
$\a{d_t-a}\le d_\D$. Since there are infinitely many $t$:s and only
finitely many $d_t$:s, there is some $d$ such that $d=d_t$ for at
least three different $t$:s. Then
$$\a{d+tc}=\a{a+tc}\quad \forall t.$$

If now $\a{a-d}\le\D$ and $\a{d-b}\le\D$, then $[a+tc]=[b+tc]$ for
all $t$. Otherwise, for all $t>\!\!>1$ there is a
$e_t\notin\b{a,b,d}$ such that
$$[e_t+tc]=[a+tc],$$
and the statement follows by a finite induction.

To prove (ii) it is enough to consider $a=a_c$. Let
$b\in[a+tc]-(a+tc)$  for some $t=t_0$, such that $\a{a+tc}\ge
d_\D^2(\a{a_c}+\a{c})\a{c}$. Then $\a{a+tc+b}^2=\a{a+tc}^2$, i.e.
$$2t\sc{b,c}+2\sc{b,a}+\a{b}^2=0.$$

If $\sc{b,c}\not=0$, then
$$\a{a+tc}\le \a{a} +\a{t\sc{b,c}}\a{c}\le
\a{a}+(\a{\sc{b,a}}+\frac12\a{b}^2)\a{c}$$ which is less than
$$((d_\D+1)\a{a}+\frac12 d_\D^2)\a{c}.$$
But this is impossible under the assumption on $a+tc$. Therefore
$\sc{b,c}=0$, i.e. $[a+tc]-(a+tc)$ $\perp$ to $c$.

Moreover it follows that $\a{a+tc+b}=\a{a+tc}$ for all $t$.   If
$\a{b}\le\D$ it follows that $[a+b+tc]=[a+tc]$  for all $t$. If not,
there is a sequence of points $0=b_1,b_2,\dots,b_k=b$ in
$[a+tc]-(a+tc)$ such that $\a{b_{j+1}-b_j}\le\D$ for all $j$. By a
finite induction it follows that $[a+b+tc]=[a+tc]$ for all $t$.
Hence
$$[a+tc]=(t-t_0)c+[a+t_0c]$$
for all $t\ge t_0$.
\end{proof}

{\it More on T\"oplitz-Lipschitz matrices.}
For a matrix $Q:\LL\times\LL\to\C$
we denote by $Q(tc)$ the matrix whose components are
$$Q_a^b(tc)=:Q(tc)_a^b=Q_{a+tc}^{b+tc}.$$
\footnote{Notice the abuse of notation. In order to avoid confusion
we shall in this section denote the T\"oplitz-limit in the direction
$c$ by $Q(\i c)$.} Clearly for any subset $I,J$ of $\LL$
$$Q_I^J(tc)=:Q(tc)_I^J=Q_{I+tc}^{J+tc}$$
in an obvious sense.

\begin{Cor}
\label{c45}
Let $\L\ge d_\D^2$. If $(a,b)\in D_{\L}(c)$, then
$$Q_{[a]_{\D}}^{[b]_{\D}}(tc)Q_{[a+tc]_{\D}}^{[b+tc]_{\D}}$$
for all $t\ge 0$.
In particular, if $Q$ is T\"oplitz at $\i$, then
$$\lim_{t\to\i}\aa{Q_{[a]_{\D}}^{[b]_{\D}}(tc)-
Q_{[a]_{\D}}^{[b]_{\D}}(\i c)}=0.$$
\end{Cor}

\begin{proof}
This follows immediately from Proposition \ref{p44}(ii).
\end{proof}

\section{Small Divisor Estimates}
\label{s5}

Let $\o\in U\sbs \R^{\AA}$ be a set contained in
\begin{equation}
\label{e51}\b{\a{\o}\le C_1},\qquad C_1\ge 1.
\end{equation}

If $A:\LL\times\LL\to gl(2,\C)$ depends on the parameters
$\o\in U$ we define
$$\a{A}_{
\left\{\begin{subarray}{l} \g\\ U\end{subarray}\right\}}\sup_{\o\in U}(\a{A(\o)}_{\g},\a{\p_{\o}A(\o)}_{\g}),$$ where the
derivative should be understood in the sense of Whitney.
\footnote{This implies that $\l{A}_{\left\{\begin{subarray}{l} \g\\
U\end{subarray}\right\}}$ bounds a $\CC^1$-extension of $A(\o)$ to a
ball containing $U$.} If the matrices  $A(\o)$ and $\p_{\o}A(\o)$
are T\"oplitz at $\i$ for all $\o\in U$, then we can define
$$\l{A}_{
\left\{\begin{subarray}{l} \L,\g\\ U\end{subarray}\right\}}\sup_{\o\in U} ( \l{A(\o)}_{\L,\g},\l{\p_{\o}A(\o)}_{\L,\g}).$$
(This Lipschitz-norm is defined in section \ref{ss23}-\ref{ss24}.) When $\g=0$
we shall also denote these norms by $\a{A}_U$ and
$\l{A}_{\left\{\begin{subarray}{l} \L\\
U\end{subarray}\right\}}$.

It is clear that if $\l{A}_{\left\{\begin{subarray}{l} \L,\g\\
U\end{subarray}\right\}}$ is finite, then the convergence to the
T\"oplitz-limit is uniform in $\o$ both for $A$ and $\partial_\o A$.

\subsection{Normal form matrices}
\label{ss51}
\

\noindent
A matrix $A:\LL\times\LL\to gl(2,\C)$  is on {\it normal form}
-- denoted $\NN\FF_{\D}$ --  if
\begin{itemize}
\item[(i)] $A$ is real valued;
\item[(ii)] $A$ is symmetric, i.e. $A_b^a={}^t(A_a^b)$;
\item[(iii)] $\pi A=A$\quad ($\pi$ is defined in section \ref{ss21});
\item[(iv)]
$A$ is {\it block-diagonal} over $\EE_{\D}$,  i.e.
$A_a^b=0$ for all $[a]_{\D}\not=[b]_{\D}$.
\end{itemize}

For a normal form matrix $A$ the quadratic form $\frac12\! \sc{\z,A\z}$
takes the form
$$ \frac12\! \sc{\xi,A_1\xi}+\sc{\xi,A_2\eta}+\frac12\!
\sc{\eta,A_1\eta}$$
where $A_1+iA_2$ is a Hermitian (scalar-valued) matrix.

Let
$$
w=\left(\begin{array}{c}
u_a\\ v_a
\end{array}\right)
=C^{-1}
\left(\begin{array}{c}
\xi_a\\ \eta_a
\end{array}\right)
\quad
C\left(\begin{array}{cc}
\frac{1}{\sqrt{2}} & \frac{1}{\sqrt{2}} \\
\frac{-i}{\sqrt{2}} & \frac{i}{\sqrt{2}}
\end{array}\right)
$$
and define ${}^t\!C AC:\LL\times\LL\to gl(2,\C)$ through
$$( {}^t\!C AC )_a^b= {}^t\!C A_a^bC.$$

Then $A$ is on normal form if, and only if,
$$\frac12\! \sc{w,{}^t\!C AC w}=\frac12\! \sc{u,Qv},$$
where $Q:\LL\times\LL\to\C$ is
\begin{itemize}
\item[(i)] Hermitian, i.e. $Q_b^a=\overline{Q_a^b}$,
\item[(ii)] block-diagonal over $\EE_{\D}$.
\end{itemize}
We say that a scalar-valued matrix $Q$ with this property  is on
{\it normal form}, denoted $\NN\FF_{\D}$.

\begin{Rem}Notice that a scalar
valued normal form matrix $Q$ will in general not become a
$gl(2,\R)$-valued normal matrix through the identification
$Q_a^b=Q_a^bI$, because the identification with ${}^t\!CAC$ is
different. However, the T\"oplitz properties are the same and the
two Lipschitz-norms (obtained by these two different identifications) are equivalent.
\end{Rem}

We denote for any subset $I$ of $\LL$
$$Q_I=Q_I^I=Q|_{I\times I}.$$

\subsection{Small divisor estimates}
\label{ss52}

Let $\O=\O(\o):\LL\times\LL\to \R$  be a real scalar valued
diagonal matrix with diagonal elements
$$\O_a(\o),\qquad \o\in U.$$
Consider the conditions
\begin{equation}\label{e52}
\left\{\begin {array}{l}
\a{\partial_\o^{\nu}(\O_a(\o)-|a|^2)}
\le C_2e^{-C_3\a{a}}, \quad C_3>0\\
(a,\o)\in\LL\times U,\quad\nu=0,1,
\end{array}\right.
\end{equation}
and
\begin{equation}\label{e53}
\left\{\begin{array}{ll}
\sc{\p_\o(\sc{k,\o}+\O_a(\o)),\frac{k}{|k|}}\ge C_4>0 &   \\
\sc{\p_\o(\sc{k,\o}+\O_a(\o)+\O_b(\o)),\frac{k}{|k|}}\ge C_4&
 a,b\in \LL,\ k\in\Z^{\AA}\sm0,\ \o\in U \\
\sc{\p_\o(\sc{k,\o}+\O_a(\o)-\O_b(\o)),\frac{k}{|k|}}
\ge C_4\ &(|a|\not=|b|)
\end{array}\right.\end{equation}

Let $H=H(\o):\LL\times\LL\to \C$ and consider
\begin{equation}\label {e54}
\aa{\p_{\o} H(\o)}\le \frac{C_4}4,\quad \o\in U.
\end{equation}
(Here $\aa{\ }$ is the operator norm.)

Let us first formulate and prove the easy case.

\begin{Prop}
\label{p51} Let $\D'>1$ and $1>\k>0$.  Assume that $U$ verifies
(\ref{e51}), that $\O$ is  real diagonal and verifies
(\ref{e52})+(\ref{e53}) and that $H$ verifies (\ref{e54}). Assume
also that $H(\o)$ is $\NN\FF_{\D}$ for all $\o\in U$.

Then there exists a closed
set $U'\sbs U$,
$$
Leb(U\sm U')\le\cte \max(\D',d_\D^2)^{2d+\#\AA-1}(C_1+
\sup_U\aa{H(\o)})^d \k C_1^{\#\AA-1}
$$
such that for all $\o\in U'$, all $0<\a{k}\le\D'$ and
for all
\begin{equation}\label{e55}
[a]_{\D},[b]_{\D}
\end{equation}
we have
\begin{equation}\label{e56}
\a{\sc{k,\o}}\ge \k,
\end{equation}
\begin{equation}\label{e57}
\a{\sc{k,\o}+\al(\o)}\ge \k
\quad \forall\
\al(\o)\in\s((\O+H)(\o)_{[a]_{\D}})
\end{equation}
and
\begin{equation}\label{e58}
\a{\sc{k,\o}+\al(\o)+\be(\o)}\ge\k
\quad \forall
\left\{\begin{array}{l}
\al(\o)\in\s((\O+H)(\o)_{[a]_{\D}})\\
\be(\o)\in\s((\O+H)(\o)_{[b]_{\D}}).
\end{array}\right.
\end{equation}
Moreover the $\k$-neighborhood of $U'\sbs U$ satisfies  the same
estimate.

The constant $\cte$ depends on the dimensions $d$ and  $ \#\AA $ and
on $C_4$.
\end{Prop}

\begin{proof}
It is enough to prove the statement for $\D'\ge d_\D^2$. Let us
prove the estimate (\ref{e58}), the other two being the same, but
easier. Let $C_5=\sup_U\aa{H(\o)}$.

Since $\a{k}\le\D'$, $\a{\sc{k,\o}}\lsim C_1\D'$. \footnote{In this
proof $\lsim$ depends on $d$, $ \#\AA $ and on $C_4$.} If the block
$I$ intersects $\b{\a{c}\gsim\sqrt{C_1\D'+C_5}}$, then any
eigenvalue $\al$ of $(\O+H)(\o)_I$ verifies
$$\al\gsim C_1\D'.$$
Hence
$$\a{<k,\o>+\al+\be}\gsim 1.$$

So it suffices to consider pair of eigenvalues
$\al\in\s((\O+H)(\o)_I)$ and $\be\in\s((\O+H)(\o)_J)$
with blocks
$$I,J\sbs\b{\a{c}\lsim\sqrt{C_1 \D'+C_5}}.$$
(Here we used that $\D'\ge d_\D^2$.)
These are at most
$$\lsim (C_1\D'+C_5)^d$$
many possibilities.

Now,  $(\sc{k,\o}+\al+\be)$ is an eigenvalue of the Hermitian
operator $\sc{k,\o}+\HH(\o)$,
$$\HH(\o):X\mapsto (\O+H)(\o)_IX+(\O+H)(\o)_JX$$
which extends $\CC^1$ to a ball around $U$ in $\b{\a{\o}<C_1}.$
Assumptions (\ref{e53}) and (\ref{e54}), via Proposition~\ref{pA3}
(Appendix), now imply that the inverse of $\HH(\o)$ is bounded from
above by $\frac1{\k}$ -- this gives a lower bound for its
eigenvalues -- outside a set of Lebesgue measure
$$\lsim d_\D^d \frac{\k}{\a{k}}C_1^{\#\AA-1}.$$

Summing now over all these blocks $I,J$ and all $\a{k}\le\D'$  gives
the result.
\end{proof}

We now turn to the main problem.

\begin{Prop}
\label{p52} Let $\D'>1$ and $0<\k<1$. Assume that $U$ verifies
(\ref{e51}), that $\O$ is  real diagonal and verifies
(\ref{e52})+(\ref{e53}) and that $H$ verifies (\ref{e54}). Assume
also that $H(\o)$ and $\p_\o H(\o)$ are T\"oplitz at $\i$ and
$\NN\FF_{\D}$ for all $\o\in U$.

Then there exists
a subset $U'\sbs U$,
$$\begin{array}{c}
Leb(U\sm U')\le\\
\cte\max(\D',d_\D^2,\L)^{\text{exp}+\#\AA-1}(C_1+
\l{H}_{\left\{\begin{subarray}{c}\L\\ U\end{subarray}\right\}} )^d
\k^{(\frac1{d+1})^d}C_1^{\#\AA-1},
\end{array}
$$
such that, for all $\o\in U'$, $0<\a{k}\le\D'$  and  all
\begin{equation}\label{e59}
\rm{dist}([a]_{\D},[b]_{\D})\le\D'
\end{equation}
we have
\begin{equation}\label{e510}
\a{\sc{k,\o}+\al(\o)-\be(\o)}\ge\k \quad \forall
\left\{\begin{array}{l}
\al(\o)\in\s((\O+H)(\o)_{[a]_{\D}})\\
\be(\o)\in\s((\O+H)(\o)_{[b]_{\D}}).
\end{array}\right.
\end{equation}
Moreover the $\k$-neighborhood of $U\sm U'$ satisfies the same
estimate.

The exponent $\text{exp}$ depends only on $d$. The constant $\cte$
depends on the dimensions $d$ and  $\#\AA $ and on $C_2,C_3,C_4$.
\end{Prop}

\begin{proof}
The proof goes in the following way: first we prove an estimate in a
large finite part of $\LL$ (this requires parameter restriction);
then we assume an estimate ``at $\i$'' of $\LL$ and we prove, using
the Lipschitz-property, that this estimate propagate from ``$\i$''
down to the finite part (this requires no parameter restriction); in
a third step we have to prove the assumption at $\i$. This will be
done by a finite induction on the ``T\"oplitz-invariance'' of $H$.

Let us notice that it is enough to prove the
statement for $\D'\ge\max(\L,d_\D^2)$.
We let $[\ ]$ denote $[\ ]_{\D}$.
Let $\O\approx(\D')^2$.

\medskip

{\it 1. Finite part.}
For the finite part, let us suppose $a$ belongs to
\begin{equation}\label{e511}
\b{a\in\LL:\a{a}\lsim(C_1+\frac1{\k_1}
d_\D^d\l{H}_{\left\{\begin{subarray}{c}\L\\ U\end{subarray}
\right\}})\O^{2d-1}},
\end{equation}
\footnote{In this proof $\lsim$ depends on $d,\#\AA $ and on
$C_2,C_3,C_4$.} where $\k_1=\k^{\frac{1}{d+1}}$. These are finitely
many possibilities and (\ref{e510})${}_\k$  is fulfilled, for all
$[a]$ satisfying (\ref{e511}), all $[b]$ with $\a{a-b}\lsim\D'$ and
all $0<\a{k}\le\D'$, outside a set of Lebesgue measure
\begin{equation}\label{e512}
\lsim\ d_\D^d(C_1+ d_\D^d \l{H}_{\left\{\begin{subarray}{c}\L\\
U\end{subarray}\right\}})^d
\O^{d(2d-1)}(\D')^{d+\#\AA-1}\frac{\k}{\k_1^d} C_1^{\#\AA-1}.
\end{equation}
(This is the same argument as in Proposition \ref{p51}.)

Let us now get rid of the diagonal terms $\hat V(a,\o)=\O_a(\o)-|a|^2$
which, by (\ref{e52}), are
$$\le\ C_2e^{-\a{a}C_3}.$$
We include them into $H$. Since they are diagonal, $H$ will  remain
on normal form. Due to the exponential decay of $\hat V$, $H$ and
$\p_{\o}H$ will remain T\"oplitz at $\i$. The Lipschitz norm gets
worse but this is innocent in view of the estimates. Also the
estimate of $\p_\o H(\o)$ gets worse, but if $a$ is outside
(\ref{e511}) then condition (\ref{e54}) remains true with a slightly
worse bound, say
$$ \aa{\p_\o H(\o)}\le\frac{3C_4}8,\quad \o\in U.$$

So from now on, $a$ is outside (\ref{e511}) and
$$\O_a=|a|^2.$$

\medskip

{\it 2. Condition at $\i$.} For each vector $c\in\Z^d$ such that
$0<\a{c}\lsim\O^{d-1}$,   we suppose that the T\"oplitz limit
$H(c,\o)$ verifies (\ref{e510})${}_{\k_1}$ for (\ref{e59}) and for
\begin{equation}\label{e513}
([a]-[b])\perp c.
\end{equation}
It will become clear in the next part why we only need
(\ref{e510})${}_{\k_1}$ and $(\ref{e59})$ under the supplementary
restriction (\ref{e513}).
\medskip

{\it 3. Propagation of the condition at $\i$.} We must now prove
that for $\a{b-a}\lsim\D'$ and an $a\in\LL$ outside (\ref{e511}),
(\ref{e510})$_\k$ is fulfilled.

By the Corollary \ref{c43} we get
$$
(a,b)\in \bigcup_{0<\a{c}\lsim \O^{d-1}}D_{\O'}(c),\quad \O' \approx
\frac{\O}{\D'}.$$ Fix now $0<|c|\lsim\O^{d-1}$ and $(a,b)\in
D_{\O'}(c)$. By Proposition \ref{p44} (ii) --  notice that $\O'\ge
d_\D^2$  --
$$[a+tc]=[a]+tc \quad\textrm{and}\quad [b+tc]=[b]+tc$$
for $t\ge 0$ and
$$[a]-a,\ [b]-b\ \perp\ c.$$
It follows (Corollary~\ref{c45}) that
$$\lim_{t\to\i}H(\o)_{[a+tc]}=H(c,\o)_{[a]}\quad\textrm{and}\quad
\lim_{t\to\i}H(\o)_{[b+tc]}=H(c,\o)_{[b]}.$$

The matrices $\O_{[a+tc]}$ and $\O_{[b+tc]}$ do not
have limits as $t\to\i$. However, for any
$(\#[a]\times\#[b])$-matrix X,
$$\O_{[a+tc]}X-X\O_{[b+tc]}\ =\ \O_{[a]}X-X\O_{[b]}+2t\sc{a-b,c}X$$
for $t\ge0$, and we must discuss two different  cases according to
if $<c,b-a>=0$ or not.

Consider for $t\ge0$ a pair of continuous eigenvalues
$$\left\{\begin{array}{l}
\al_t\in\s((\O+H(\o))_{[a+tc]})\\
\be_t\in\s((\O+H(\o))_{[b+tc]}) \end{array}\right.$$

Case I: $\sc{c,b-a}=0$. Here
$$(\O+H(\o))_{[a+tc]}X-X(\O+H(\o))_{[b+tc]}$$
equals
$$(\a{a}^2+H(\o))_{[a+tc]}X-X(\a{b}^2+H(\o))_{[b+tc]}$$
--  the linear and quadratic terms in $t$ cancel!

By continuity of
eigenvalues,
$$\lim_{t\to\i}(\al_t-\be_t)= (\al_{\i}-\be_{\i}),$$
where
$$\left\{\begin{array}{l}
\al_{\i}\in\s((\a{a}^2+H(c,\o))_{[a]})\\
\be_{\i}\in\s((\a{b}^2+H(c,\o))_{[b]})
\end{array}\right.$$
Since $[a]$ and $[b]$ verify (\ref{e513}),
our assumption on $H(c,\o)$ implies that $(\al_{\i}-\be_{\i})$
verifies (\ref{e510})${}_{\k_1}$.

For any two $a,a'\in[a]$ we have $|a|=|a'|$.
Hence
$$\aa{H(\o)_{[a]}-H(c,\o)_{[a]}}\frac{\a{a}}{\a{c}}\
\lsim\ d_\D^d \l{H}_{\left\{\begin{subarray}{c}\L\\
U\end{subarray}\right\}},$$ because $\D'\ge\L$, and the same for
$[b]$. Recalling that $a$ and, hence, $b$ violate (\ref{e511}) this
implies
$$\aa{H(\o)_{[d]}-H(c,\o)_{[d]}}\ \le\
\frac{\k_1}4,\quad d=a,b.$$ By Lipschitz-dependence of eigenvalues
(of Hermitian  operators) on parameters, this implies that
$$\a{(\al_0-\be_0)-(\al_{\i}-\be_{\i})}\le \frac{\k_1}2 $$
and we are done.

Case II: $\sc{c,b-a}\not=0$. We write $a=a_c+\t c$. Since
$$\a{a}\ge\O'(\a{a_c}+\a{c})\a{c},$$
it follows that
$$\a{a_c}\le\frac1{\O'}\frac{\a{a}}{\a{c}}.$$

Now,
$\al_0-\be_0$ differs from $\a{a}^2-\a{b}^2$ by at most
$$2\aa{H(\o)}\lsim d_\D^d\l{H}_{\left\{\begin{subarray}{c}\L\\
 U\end{subarray}\right\}},$$
and
$$\a{a}^2-\a{b}^2=-2\t\sc{c,b-a}-2\sc{a_c,b-a}-\a{b-a}^2.$$
Since $|\sc{c,b-a}|\ge1$ it follows that
$$\t\lsim |\al_0-\be_0|+|a_c|\D'+(\D')^2+
d_\D^d\l{H}_{\left\{\begin{subarray}{c}\L\\
U\end{subarray}\right\}}.$$ If now $\a{\al_0-\be_0}\lsim\ C_1\D'$
then $\a{a}\le\a{a_c}+\a{\t\ c}$ is
$$\le\cte(\a{a_c}\D'\a{c}+C_1(\D')^2\a{c}+
d_\D^d\l{H}_{\left\{\begin{subarray}{c}\L\\
U\end{subarray}\right\}}|c|)$$
$$
\le\frac12|a|+\cte(C_1(\D')^2\a{c}+
d_\D^d\l{H}_{\left\{\begin{subarray}{c}\L\\
U\end{subarray}\right\}}|c|).$$ Since $a$  violates (\ref{e511})
this is impossible. Therefore $\a{\al_0-\be_0}\gsim\ C_1\D'$ and
(\ref{e510})${}_\k$ holds.

Hence, we have proved that (\ref{e510})$_\k$ holds for any
$$\left\{
\begin{array}{l}
a\in (\ref{e511})_{\k_1}\\
(a,b)\in(\ref{e59})
\end{array}\right.
\ \cup\
\left\{
\begin{array}{l}
(a,b)\in
\bigcup_{0<\a{c}\lsim \O^{d-1}}D_{\O'}(c)\\
(a,b)\in(\ref{e59})
\end{array}\right.$$
under the condition at $\i$. Therefore (\ref{e510})$_\k$  holds for
any $(a,b)\in (\ref{e59})$.

\medskip

{\it 4. Proof of condition at $\i$  --  induction.} Let $c_1$ be a
primitive vector in $0<\a{c_1}\lsim\O^{d-1}$, and let $G$ be the
T\"oplitz limit $H(c_1)$. Then $G$ verifies (\ref{e54}),  $G(\o)$
and $\p_\o G(\o)$ are T\"oplitz at $\i$ and
$$\l{G}_{\left\{\begin{subarray}{c}\L\\ U\end{subarray}\right\}}
\le
\l{H}_{\left\{\begin{subarray}{c}\L\\ U\end{subarray}\right\}}.$$
Clearly $G(\o)$ is Hermitian and, by Proposition \ref{p44} (i),
$G(\o)$ and $\p_\o G(\o)$ are block diagonal over $\EE_\D$,
i.e. $G(\o)$ and $\p_\o G(\o)$ are $\NN\FF_\D$.
Moreover $G$ is T\"oplitz in the direction $c_1$,
$$G_{a+tc_1}^{b+tc_1}=G_a^b,\quad \forall a,b,tc_1.$$

$\O_a=|a|^2$ for all $a$, so $\O$ verifies (\ref{e52}+\ref{e53}).

We want to prove that $G$ verifies (\ref{e510})$_{\k_1}$ for all
$(a,b)\in (\ref{e59})+(\ref{e513})_{c_1}$, i.e. for all
$$|a-b|\lsim \D' \quad\textrm{and}\quad ([a]-[b])\perp c_1.$$
Since $G$ is T\"oplitz in the direction $c_1$ it is enough to
show this for
\begin{equation}\label{e514}
\a{\mathrm{proj}_{\mathrm{Lin}(c_1)}a}\lsim \O^{d-1}.
\end{equation}
To prove this we repeat the previous arguments.

{\it Finite part.} In the set $(\ref{e511})_{\k_2}$,
$\k_2=\k_1^{\frac1{d+1}}$, there are only finitely many
possibilities and $(\ref{e510})_{\k_1}$ will be  fulfilled outside a
set of $\o$ of Lebesgue measure
$(\ref{e512})_{\frac{\k_1}{\k_2^d}}$.

{\it A second condition at $\i$.} For each vector $c\in\Z^d$ such
that $0<\a{c}\lsim\O^{d-1}$  and $c$ and $c_1$ being linearly
independent, we suppose that the T\"oplitz limit $G(c,\o)$ verifies
(\ref{e510})${}_{\k_2}$ for all $(a,b)\in (\ref{e59})
+(\ref{e513})_{c_1}+(\ref{e513})_{c}$, i.e. for all
$$|a-b|\lsim \D' \quad\textrm{and}\quad ([a]-[b])\perp c_1,c.$$

{\it Propagation of condition at $\i$.} The same argument  as before
shows that (\ref{e510})$_{\k_1}$ holds for any
$$\left\{
\begin{array}{l}
a\in (\ref{e511})_{\k_2}\\
(a,b)\in(\ref{e59})
\end{array}\right.
\ \cup\
\left\{
\begin{array}{l}
(a,b)\in
\bigcup_{
\begin{subarray}{l}
0<\a{c}\lsim \O^{d-1}\\
c\not\parallel c_1\end{subarray}}
D_{\O'}(c)\\
(a,b)\in(\ref{e59})+(\ref{e513})_{c_1}
\end{array}\right.$$
under the condition at $\i$.

Since $a$ verifies (\ref{e514}), it follows that $a\in (\ref{e511})_{\k_2}$ or
$$(a,b)\notin D_{\O'}(c_1).$$
Indeed, if $(a,b)\in D_{\O'}(c_1)$, then (Corollary \ref{c22} (i))
$$\a{a}\approx \frac{\a{\sc{a,c_1}}}{\a{c_1}}\lsim \O^{d-1}$$
which implies that $a\in (\ref{e511})_{\k_2}$.
Therefore (\ref{e510})$_{\k_1}$ holds for any
$(a,b)\in (\ref{e59})+(\ref{e513})_{c_1}$.

\medskip

{\it 5. The first inductive step.}
Suppose we have a matrix
$G$ verifying (\ref{e54}) and such that $G(\o)$ and $\p_\o G(\o)$
are T\"oplitz at $\i$ and $\NN\FF_\D$ and
$$\l{G}_{\left\{\begin{subarray}{c}\L\\ U\end{subarray}\right\}}
\le \l{H}_{\left\{\begin{subarray}{c}\L\\
U\end{subarray}\right\}}.$$ Suppose also that there are primitive
and linearly independent vectors $c_1,\dots,c_{d-1}$ of norm
$\lsim\O^{d-1}$, such that  $G$ is T\"oplitz in these directions,
i.e.
$$G_{a+tc_j}^{b+tc_j}=G_a^b,\quad \forall a,b,tc_j,\quad
j=1,\dots,d-1.$$

We want to prove that $G$ verifies $(\ref{e510})_{\k_{d-1}}$, $\k_{d-1}=\k_{d-2}^{\frac1{d+1}}$,
for all $(a,b)\in
(\ref{e59})+ (\ref{e513})_{c_1}+\dots+(\ref{e513})_{c_ {d-1}}$.
Since $G$ is T\"oplitz in the directions $c_1,\dots,c_{d-1}$ it
suffices to prove this for $a\in (\ref{e514})_{c_1,\dots,c_{d-1}}$, i.e.
$$
\a{\mathrm{proj}_{\mathrm{Lin}(c_1,\dots,c_{d-1})}a}\lsim \O^{d-1}.
$$

If $(a,b)\in (\ref{e511})_{\k_d}$, $\k_d=\k_{d-1}^{\frac1{d+1}}$,
then $(\ref{e510})_{\k_{d-1}}$ will be
fulfilled outside a set of $\o$ of Lebesgue measure
$(\ref{e512})_{\frac{\k_{d-1}}{\k_d^d}}$.

By assumptions $(\ref{e513})_{c_1}+\dots+(\ref{e513})_{c_ {d-1}}$,
 $[a]$ and $[b]$ are contained in one and the same affine line, so
$\#[a],\#[b]\le 2$.
If now $(a,b)\not\in (\ref{e511})_{\k_d}$, then
$$
\a{a}\gsim \a{\mathrm{proj}_{\mathrm{Lin}(c_1,\dots,c_{d-1})}a},
$$
and the same for $b$. Therefore
$\#[a]=\#[b]=1$ and
$$|a+b|\gsim (C_1+\sup_U\aa{G(\o)})\O^{2d-1} .$$
Since $a$ and $b$ are parallel it follows that
$$||a|^2-|b|^2|\gsim (C_1+\sup_U\aa{G(\o)})\O^{2d-1} ,$$
unless $[a]=[b]=\b{a}$. In the first case we are done because
$|\sc{k,\o}|\lsim C_1\D'$ and in the second case
condition $(\ref{e510})_{\k_{d-1}}$ reduces to
$$\a{\sc{k,\o}}>\k.$$

This completes the proof of the first inductive step  and, hence, of
the proposition.
\end{proof}

\section{The homological equations}
\label{s6}

\subsection{A first equation}
\label{ss61}
\

\noindent
For $k\in\Z^n$ consider the equation
\begin{equation}\label{e61}
i\sc{k,\o}S+i(\O(\o)+H(\o))S=F(\o),
\end{equation}
where $F(\o)$ and $\p_\o F(\o)$ are elements in
$l^2_\g(\LL,\C)=\b{\xi=(\xi_a)_{a\in\LL}:\aa{\xi}_\g<\i}$,
$$\aa{\xi}_\g
=\sqrt{\sum_{a\in\LL}\a{\xi_a}^2e^{2\g|a|}\langle a\rangle^{2m_*}}$$
($\langle a\rangle=\max(1,|a|)$). Denote
$$\aa{F}_{\left\{\begin{subarray}{l} \g \\ U \end{subarray}\right\}}
=\sup_{\o\in U}(\aa{F(\o)}_\g,\aa{\p_\o F(\o)}_\g).$$

Let $U'\sbs U$ be a set such that for all $\o\in U'_\k$ the  small
divisor condition $(\ref{e57})$ holds for all $a$, i.e.
$$
\a{\sc{k,\o}+\al(\o)}\ge \k,
\quad \forall\
\al(\o)\in\s((\O+H)(\o)).$$

\begin{Prop}\label{p61}
Let $0<\k<1$. Assume that $\O$ is  real diagonal and verifies
(\ref{e52}) and  that $H$ verifies (\ref{e54}). Assume also   that
$H(\o)$ and $\p_\o H(\o)$ are $\NN\FF_{\D}$ for all $\o\in U$.

Then the equation
(\ref{e61}) has for all $\o\in U'$ a unique solution $S(\o)$
such that
$${\aa{S}}_{\left\{\begin{subarray}{l} \g\\ U'\end{subarray}\right\}}
\le\cte \frac1{\k^2}d_\D^{2m_*}e^{2\g d_\D}(1+|k|)
\aa{F}_{\left\{\begin{subarray}{l} \g\\ U'\end{subarray}\right\}}.$$

The constant $\cte$ only depends on $d, \#\AA,m_* $ and
$C_2,C_3,C_4$.
\end{Prop}

\begin{proof} This is a standard result. The equation (\ref{e61})
has a unique solution verifying
$$\aa{S(\o)}_\g\lsim  \frac1\k d_\D^{m_*}e^{\g d_\D}\aa{F(\o)}_\g.$$
The factor $d_\D^{m_*}e^{\g d_\D}$ comes in because the
block-diagonal character of $\O(\o)+H(\o)$ interferes with the
polynomial and exponential decay.

If we differentiate equation (\ref{e61}) with respect to $\o$
we get
$$
i\sc{k,\o} \p_\o S+i(\O(\o)+H(\o))\p_\o S$$
$$=\p_\o F(\o)-i(\p_\o \sc{k,\o})S-i\p_\o (\O(\o)+H(\o))S.$$
If we apply the same estimate to this equation we get the  result on
$U'$.

In order to extend $S$ from $U'$ to a ball we take a $\CC^1$ cut off
function $\chi$ which is $1$ on $U'$ and $0$ outside $U'_{\k}$. We
now first solve the equation on $U'_{\k}$  as above to get a
solution $\tilde S$ and then we define $S=\chi\tilde S$.
\end{proof}

\subsection{Truncations}
\label{ss62}
\

\noindent For a matrix $Q:\LL\times\LL\to \C$ consider three
truncations
$$\begin{array}{l}
\TT_{\D'} Q= Q\ \text{restricted to}\ \b{(a,b):
\a{a-b}\le\D'}\\
\PP_{c}Q= Q\ \text{restricted to}\ \b{(a,b):
(a-b)\perp c}\\
\DD_{\D'}Q= Q\ \text{restricted to}\
\b{(a,b):\a{a-b}\le \D'\
\text{and} \a{a}=\a{b}}.
\end{array}$$
These truncations  all commute. Moreover,

\begin{Lem}
\label{l62}
\begin{itemize}
\item[(i)]
$$\left\{\begin{array}{ccc}
\a{ \TT_{\D'} Q }_{\left\{\begin{subarray}{l} \g\\
U\end{subarray}\right\}} \ &\le\ & \a{Q}_{\left\{\begin{subarray}{l}
\g\\ U\end{subarray}\right\}}
\\
\l{\TT_{\D'} Q }_{\left\{\begin{subarray}{l} \L,\g\\
U\end{subarray}\right\}} \ &\le\ & \
\l{Q}_{\left\{\begin{subarray}{l} \L,\g\\ U\end{subarray}\right\}}
\end{array}\right.$$
and
$$(\TT_{\D'}  Q)(c)=\TT_{\D'} (Q(c))$$
for all $c$.
\item[(ii)] The same result holds for $\PP_{c}$.
\item[(iii)]
$$\left\{\begin{array}{ccc}
\a{ \DD_{\D'} Q }_{\left\{\begin{subarray}{l} \g\\
U\end{subarray}\right\}} \ &\le\ & \a{Q}_{\left\{\begin{subarray}{l}
\g\\ U\end{subarray}\right\}}
\\
\l{\DD_{\D'} Q }_{\left\{\begin{subarray}{l} \L,\g\\
U\end{subarray}\right\}} \ &\le\ & \
\l{Q}_{\left\{\begin{subarray}{l} \L,\g\\ U\end{subarray}\right\}},
\end{array}\right.$$
for any $\L\ge (d_{\D'})^2$. Moreover
$$(\PP_c\DD_{\D'}Q)(c)=(\PP_c\DD_{\D'})(Q(c))$$
for all $c$.
\end{itemize}
\end{Lem}

\begin{proof} (i) and (ii) are obvious. Let us consider (iii).

 We have $(\DD_{\D'}Q)_a^b(c)$ is $=Q_a^b(c)$ if
$$\a{a-b}\le\D',\quad \a{a}=\a{b},\quad (a-b)\perp c,$$
and is $=0$ otherwise. This gives immediately the last statement.

If $\a{a-b}\le\D'$, then
$$\a{a}=\a{b} \Longrightarrow [a]_{\D'}=[b]_{\D'}.$$
Hence, if $(a,b)\in D_{\L}(c)$ and $\a{a-b}\le\D'$, then
$$\a{a}=\a{b} \Longrightarrow (a-b)\perp c.$$
>From this we derive that
$ (\DD_{\D'}Q)_a^b-(\DD_{\D'}Q)_a^b(c)$ is $=Q_a^b-Q_a^b(c)$ or $=0$.
\end{proof}

\subsection{A second equation, $k\not=0$}
\label{ss63}
\

\noindent
For $k\in\Z^n\sm\b{0}$ consider the equation
\begin{equation}\label{e62}
i\sc{k,\o}S+i[\O(\o)+H(\o),S]=\TT_{\D'}F(\o)
\end{equation}
where $F(\o):\LL\times\LL\to\C$ and $\p_\o F(\o)$ are
T\"oplitz at $\i$.

Let $U'\sbs U$ be a set such that for all $\o\in U'_\k$ the  small
divisor condition $(\ref{e59})_{\D'+2d_\D}+(\ref{e510})$ holds, i.e.
$$\a{\sc{k,\o}+\al(\o)-\be(\o)}\ge\k \quad \forall
\left\{\begin{array}{l}
\al(\o)\in\s((\O+H)(\o)_{[a]_{\D}})\\
\be(\o)\in\s((\O+H)(\o)_{[b]_{\D}})
\end{array}\right.$$
for
$$\rm{dist}([a]_{\D},[b]_{\D})\le\D'+2d_\D.$$

\begin{Prop}\label{p63}
Let $\D'>1$ and $0<\k<1$.
Assume that $U$ verifies
(\ref{e51}), that $\O$ is  real diagonal and verifies
(\ref{e52}), and that $H$ verifies (\ref{e54}). Assume also
that $H(\o)$ and
$\p_\o H(\o)$ are T\"oplitz at $\i$ and
$\NN\FF_{\D}$ for all $\o\in U$.

Then the equation
$$(\ref{e62})\quad\textrm{and}\quad S=\TT_{\D'+2d_\D}S$$
has for all $\o\in U'$  a unique solution $S(\o)$ verifying
\begin{itemize}
\item[(i)]
$$\a{S}_{\left\{\begin{subarray}{l} \g\\ U'\end{subarray}\right\}}
\le\cte\ \frac{1}{\k^2}d_\D^{2d}e^{2\g d_\D}(1+\a{k})
\a{F}_{\left\{\begin{subarray}{l} \g\\ U'
\end{subarray}\right\}};$$
\item[(ii)] $S(\o)$ and $\p_\o S(\o)$ are T\"oplitz at $\i$ and
the T\"oplitz-limits verify
$$\left\{\begin{array}{l}
i\sc{k,\o}S+i[\O(\o)+H(c,\o),S]=\TT_{\D'}\PP_{c}F(c,\o)\\
S=\TT_{\D'+2d_\D}S;
\end{array}\right.
$$
\item[(iii)]
$$
\l{S}_{\left\{\begin{subarray}{l} \L'+d_\D+2,\g \\
U'\end{subarray}\right\}}\ \le$$
$$\cte\frac{1}{\k^3}d_\D^{2d}e^{2\g d_\D}
 (1+\a{k}+\l{H}_{\left\{\begin{subarray}{l} \L\\
 U'\end{subarray}\right\}})
\l{F}_{\left\{\begin{subarray}{l} \L',\g\\ U'
\end{subarray}\right\}}
$$
for any
$$\L'\gsim \max(\L,d_\D^2,\D',\sup_U\aa{H(\o)}).$$
\end{itemize}
The constant $\cte$ only depends on the dimensions  $d$ and $\#\AA $
and on $C_1,C_2,C_3,C_4$.
\end{Prop}

\begin{proof}
Let us first get rid of the diagonal terms $\hat V(a,\o)=\O_a(\o)-|a|^2$
which by (\ref{e52}) are
$$\lsim\ C_2e^{-\a{a}C_3}.$$
\footnote{In this proof $\lsim$ depends on $d,\#\AA $  and on
$C_1,C_2,C_3,C_4$.} We include them into $H$  --  in view of the
estimates of the proposition this is innocent. Let us also notice
that it is enough to prove the statement for $\L\ge\ d_\D^2$. We
first assume that $F=\TT_{\D'}F$.

So from now on we assume $\O_a=|a|^2$ and $\L\ge\ d_\D^2.$
We shall denote the blocks $[\ \ ]_{\D}$ by $[\ \ ]$.

\medskip

We first block decompose the equation  (\ref{e62}) over $\EE_{\D}$ taking into account the truncation
of $S$ and the small divisor condition. It
becomes
\begin{equation}\label{e621}
\left\{\begin{array}{ll}
i\sc{k,\o}S_{[a]}^{[b]}+i(\O+H(\o))_{[a]}S_{[a]}^{[b]}-&
\text{if}\ \text{dist}([a],[b])\le\D'   \\
iS_{[a]}^{[b]}(\O+H(\o))_{[b]}=F_{[a]}^{[b]}(\o) & \\
S_{[a]}^{[b]}=0 &\text{if not}.
\end{array}\right.
\end{equation}

Since $\O+H$ is Hermitian, under the small divisor condition the
equation (\ref{e621}) has a unique solution which is $\CC^1$ in
$\o$ and verifies
$$
\a{S_a^b}\le \aa{S_{[a]}^{[b]}}
\le\frac{1}{\k}\aa{F_{[a]}^{[b]}}$$
($\aa{\ }$ is the operator norm), hence
\begin{equation}\label{e63}
\a{S}_{\g}\le\frac{1}{\k}d_\D^{d}e^{2\g d_\D}\a{F}_{\g}.
\end{equation}
The factor $d_\D^d$ comes from the two different matrix norms used
here, and the exponential factor occurs because the block character
of $\O+H$ interferes with the exponential decay.

In order to estimate the
derivatives in $\o$ we just differentiate (\ref{e621}) with respect
to $\o$:
\begin{equation}\label{e631}
\begin{array}{c}
(i\sc{k,\o}+i(\O+H(\o))_{[a]})\p_{\o}S_{[a]}^{[b]}-
i\p_{\o}S_{[a]}^{[b]}(\O+H(\o))_{[b]}=\\
=\p_{\o}F_{[a]}^{[b]}(\o) -i(\p_\o\sc{k,\o}+
\p_{\o}H(\o)_{[a]}S_{[a]}^{[b]}-S_{[a]}^{[b]} \p_{\o}H(\o)_{[b]}).
\end{array}
\end{equation}
If  $G_{[a]}^{[b]}$ is the matrix on RHS, then
$$
\begin{array}{c}
\aa{G_{[a]}^{[b]}}\ \le\ \aa{\p_\o F_{[a]}^{[b]}}+\\
(|k|+\aa{\p_\o H_{[a]}}+\aa{\p_\o H_{[b]}})\aa{S_{[a]}^{[b]}}
\end{array}$$
and $\p_{\o}S_{[a]}^{[b]}$ is now estimated like  $S_{[a]}^{[b]}$.

We do now the same thing on $U'_\k$ and then we extend $S$ from $U'$
to be $0$ outside $U'_\k$ by a $\CC^1$ cut-off. This gives (i).

\medskip

{\it T\"oplitz at $\i$.}
Let $Q$ be a matrix on $\LL$ and denote by $Q(tc)$
the matrix whose elements are
$$Q_a^b(tc)=Q_{a+tc}^{b+tc}.$$
\footnote{In order to avoid confusion
we shall denote the T\"oplitz-limit in the direction $c$
by $Q(\i c)$.}

By Proposition \ref{p44} (ii), for $(a,b)\in D_{\L'}(c)$ --   notice that
$\L'\ge d_\D^2$  --
$$[a+tc]=[a]+tc \quad\textrm{and}\quad [b+tc]=[b]+tc$$
for $t\ge 0$ and
$$[a]-a,\ [b]-b\ \perp\ c.$$
It follows that
\begin{equation}\label{e64}
\begin{array}{c}
i\sc{k,\o}S_{[a]}^{[b]}(tc)
+i(\O+H)_{[a]}(tc)S_{[a]}^{[b]}(tc)-\\
iS_{[a]}^{[b]}(tc)(\O+H)_{[b]}(tc)=F_{[a]}^{[b]}(tc)
\end{array}
\end{equation}
for all $t\ge0$.

Moreover $H_{[a]}(tc),\ H_{[b]}(tc)$ and  $F_{[a]}^{[b]}(tc)$ have
limits as $t\to\i$ (Corollary \ref{c45}). $\O_{[a]}(tc)$ and
$\O_{[b]}(tc)$ do not have limits,  and we must analyze two
different cases according to if $\sc{c,a-b}=0$ or not.

{\it Case I: $\sc{c,a-b}=0$.} We have that
$\O_{[a]}(tc)X-X\O_{[b]}(tc)$ (for any $(\#[a]\times\#[b])$-matrix
$X$) equals
$$\a{a}^2X-X\a{b}^2$$
--  the linear and quadratic terms in $t$ cancel! Therefore
equation (\ref{e64}) has a limit as $t\to\i$:
$$i\sc{k,\o}X+i(\O_{[a]}+H_{[a]}(\i c))X-iX(\O_{[b]}+
H_{[b]}(\i c))=F_{[a]}^{[b]}(\i c).$$
Since eigenvalues are continuous in parameters we have
$$
\a{\sc{k,\o}+\al-\be}\ge\k\quad \forall \left\{\begin{array}{l}
\al\in\s(\a{a}^2+H_{[a]}(\i c) )\\
\be\in\s(\a{b}^2+  H_{[b]}(\i c) ).
\end{array}\right.
$$
Therefore the limit equation has a unique solution $X$ which
is $\CC^1$ in $\o$ and verifies
$$\aa{X}\le\frac1{\k}\aa{F_{[a]}^{[b]}(\i c) }.$$
Since $ S_{[a]}^{[b]}(tc) $  is bounded, it follows  from uniqueness
that
$$S_{[a]}^{[b]}(tc) \to S_{[a]}^{[b]}(\i c)=X$$
as $t\to\i$.

{\it Case II: $\sc{c,a-b}\not=0$}. We have that
$\O_{[a]}(tc)X-X\O_{[b]}(tc)$ equals
$$
(2t\sc{a,c}+\a{a}^2)X-X(2t\sc{b,c}\a{b}^2)$$
--  only the quadratic terms in $t$ cancel! Dividing (\ref{e64})
by $t$ and letting $t\to\i$, the limit equation becomes
$$2\sc{c,a-b}X=0.$$
It has the unique solution $X=0$. For the same reason as
in the previous case we have that
$$S_{[a]}^{[b]}(tc) \to S_{[a]}^{[b]}(\i c)=0$$
as $t\to\i$.

We have thus shown that, for any $c$, the solution $S$ has a
T\"oplitz-limit $S(\i c)$ which verifies, for $(a,b)\in D_{\L'}(c)$,
\begin{equation}\label{e641}
\left\{\begin{array}{ll}
i\sc{k,\o}S_{[a]}^{[b]}+i(\O+H(\i c,\o))_{[a]}S_{[a]}^{[b]}-&
\text{if}\ \text{dist}([a],[b])\le\D'\\
iS_{[a]}^{[b]}(\O+H(\i c,\o))_{[b]}=F_{[a]}^{[b]}(\i c,\o) &
\text{and}\ (a-b)\perp c \\
S_{[a]}^{[b]}=0 &\text{if not}.
\end{array}\right.
\end{equation}
Since $S(\i c)$ is invariant under $c$-translations, this implies that
$S(\i c)$ verifies the equation in (ii).

Moreover
$$
\a{S(\i c)}_{\g}\le \frac{1}{\k}d_\D^{d}e^{2\g d_\D}\a{F(\i c)}_{\g}.$$

\medskip

{\it Estimate of Lipschitz norm.}
Consider the ``derivative'' $\p_c$:
$$\p_c Q_{[a]}^{[b]}(tc)=(Q_{[a]}^{[b]}(tc)-Q_{[a]}^{[b]}(\i c))
\max(\frac{\a{a}}{\a{c}},\frac{\a{b}}{\a{c}}).$$ (Notice that the
definition does not depend on the choice of representatives $a$ and
$b$ in $[a]$ and $[b]$ respectively.) We shall ``differentiate''
equation (\ref{e64}) and estimate the solution of the
``differentiated'' equation over $[a]\times[b]\sbs D_{\L'}(c)$ which
is $\sbs D_\L(c)$ because  $\L'\ge\L$.  By Corollary \ref{c22}(iii)
this will provide us with an estimate of the Lipschitz constant
$\text{Lip}^+_{\L'+d_\D+2,\g}$.

So we take $[a]\times[b]\sbs D_{\L'}(c)$. Since $S$ is $0$ at
distances $\gsim\D'+ d_\D$ from the diagonal we only need to treat
$\a{a-b}\lsim\D'+ d_\D$. Again we must consider two cases.

{\it Case I: $\sc{c,a-b}=0$.} Subtracting the equation (\ref{e641})
for $S_{[a]}^{[b]}(\i c)$ from the equation (\ref{e621}) for
$S_{[a]}^{[b]}$ and multiplying by $\max(
\frac{|a|}{|c|},\frac{|b|}{|c|})$ gives
$$
\begin{array}{c}
i\sc{k,\o}\p_c S_{[a]}^{[b]}+i(\O+H)_{[a]}\p_c S_{[a]}^{[b]}-
\p_cS_{[a]}^{[b]}(\O+H)_{[b]}=\\
 \p_cF_{[a]}^{[b]}-\p_cH_{[a]}S_{[a]}^{[b]}(\i c)+
 S_{[a]}^{[b]}(\i c)\p_c H_{[b]}.
\end{array}
$$
Now we get as for equation (\ref{e621}) that
$$
\aa{ \p_cS_{[a]}^{[b]}} \ \le\ \frac{1}{\k}
(\aa{\p_cF_{[a]}^{[b]}}+(\aa{\p_cH_{[a]}}+\aa{\p_c H_{[b]}})
\aa{S_{[a]}^{[b]}(\i c)}).
$$

{\it Case II: $\sc{c,a-b}\not=0$.}
Then
$$|\a{a}^2-\a{b}^2| \approx \frac{|a|}{|c|}|\sc{c,a-b}|
\approx \frac{|b|}{|c|}|\sc{c,a-b}|\gsim\L'.$$
Indeed $\a{a}^2-\a{b}^2|$ can be written
$$\a{a'+\t c}^2-\a{b'+\t c}^2\a{a'}^2-\a{b'}^2+2\t\sc{c,a-b},$$
and (recalling Lemma \ref{l21}(ii))
$$\a{\a{a'}^2-\a{b'}^2}\le\a{a-b}(\a{a'}+\a{b'})\le \cte(\D'+
d_\D)\frac{\t}{\L'}$$ and this is $\le\frac12\t$, since
$\L'\ge2\cte(\D'+ d_\D)$. Moreover (Lemma \ref{l21}(i)+(iii))
$$\frac{|a|}{|c|}\approx\frac{|b|}{|c|}\approx \t \ge\L'.$$

Since $\L'\gsim\aa{H}$, assuring that $\aa{H}$ is small
compared with $\a{a}^2-\a{b}^2|$, we have
$$\a{\al-\be}\approx 2\a{\sc{a-b,c}}\ge 2\quad\forall
\left\{\begin{array}{l}
\al\in\s(\frac1\t(\O+H)_{[a]})\\
\be\in\s(\frac1\t(\O+H)_{[b]}).\end{array}\right.$$

Since $S_{[a]}^{[b]}(\i c)=0$, multiplying (\ref{e62}) by
$\frac1\t \max( \frac{|a|}{|c|},\frac{|b|}{|c|})$ gives,
$$\begin{array}{c}
\frac i\t\sc{k,\o}\p_cS_{[a]}^{[b]}+
\frac i\t( \O+H)_{[a]}\p_cS_{[a]}^{[b]}-
\p_c S_{[a]}^{[b]}\frac i\t (\O+H)_{[b]}=\\
F_{[a]}^{[b]}\frac1\t
\max(\frac{\a{a}}{\a{c}},\frac{\a{b}}{\a{c}})\approx F_{[a]}^{[b]}.
\end{array}
$$
Since $\L'\ge C_1\D'$, the
absolute value of the eigenvalues of the LHS-operator is $\ge 1$
and it follows  that
$$
\aa{\p_c S_{[a]}^{[b]}  }
\ \lsim\ \aa{F_{[a]}^{[b]}}.$$

If $(a,b)\in D_{\L'+d_\D+2}(c)$, then
both $(a,a)$ and $(b,b)$ belongs to $D_{\L'+d_\D+2}(c)$ and,
by Corollary \ref{c22} (iii),
$$
[a]\times [b],[a]\times [a],[b]\times [b]
\sbs D_{\L'}(c)\sbs D_{\L}(c).$$
Therefore
$$\aa{\p_cH_{[a]}}+\aa{\p_c H_{[b]}}\le d_\D^{d} \l{H}_{\L}.$$
Using this, the estimates (in Case I and II) for $\aa{\p_c S_{[a]}^{[b]}}$
and the estimate (\ref{e63}) we obtain
$$
{}^1\!\!\l{S}_{\L'+d_\D+2,\g}\lsim d_\D^{2d}e^{2\g d_\D}
(\frac{1}{\k}\l{F}_{\L',\g}+
\frac{1}{\k^2}\l{H}_{\L}\a{F}_{\g}).
$$
(This norm is defined in section \ref{ss24}.)
The estimate of $\l{S}_{\L'+d_\D+2,\g}$ is obtained by a finite
induction using this estimate and the equation (ii) for the
T\"oplitz-limits.

{\it Estimate of $\o$-derivatives.} In order to estimate the
derivatives in $\o$ we consider the differentiated equation
(\ref{e631}). The RHS  $G_{[a]}^{[b]}$ verifies
\begin{equation}\label{e67}
\begin{array}{cc}
\aa{\p_c G_{[a]}^{[b]}}\ \le\ \aa{\p_c\p_\o F_{[a]}^{[b]}}+
(|k|+\aa{\p_\o H_{[a]}}+\aa{\p_\o H_{[b]}}) \aa{S_{[a]}^{[b]}}\\
+(\aa{\p_c\p_\o H_{[a]}}+\aa{\p_c\p_\o H_{[b]}})
\aa{S_{[a]}^{[b]}}.
\end{array}
\end{equation}
and $\p_c\p_{\o}S_{[a]}^{[b]}$  is now estimated like  $\p_c
S_{[a]}^{[b]}$ but with $G$ instead of $F$. Combining these
estimates now gives the result when $F=\TT_{\D'}F$. By
Lemma~\ref{l62}(i) we get the result for a general $F$.
\end{proof}

\subsection{A second equation, $k=0$}
\label{ss64}
\

\noindent
Consider the equation
\begin{equation}\label{e68}
i[\O(\o)+H(\o),S]= (\TT_{\D'}-\DD_{\D'})F(\o)
\end{equation}
where $F(\o):\LL\times\LL\to\C$ and $\p_\o F(\o)$ are
T\"oplitz at $\i$.

Let $U'\sbs U$ be a set such that for all $\o\in U'_\k$ the  small
divisor condition
\begin{equation}\label{e69}
\left\{\begin{array}{l}
\a{\al(\o)-\be(\o)}\ge\k \quad \forall
\left\{\begin{array}{l}
\al(\o)\in\s((\O+H)(\o)_{[a]_{\D}})\\
\be(\o)\in\s((\O+H)(\o)_{[b]_{\D}})
\end{array}\right.\\
\rm{dist}([a]_{\D},[b]_{\D})\le\D'+2d_\D\quad\text{and}\quad |a|\not=|b|.
\end{array}\right.
\end{equation}
holds.

\begin{Prop}\label{p64}
Let $\D'>1$ and $0<\k<1$. Assume that $U$ verifies (\ref{e51}), that
$\O$ is  real diagonal and verifies (\ref{e52}), and that $H$
verifies (\ref{e54}). Assume also that $H(\o)$ and $\p_\o H(\o)$ are
T\"oplitz at $\i$ and $\NN\FF_{\D}$ for all $\o\in U$.

Then the equation
$$(\ref{e68})\quad\textrm{and}\quad S-\TT_{\D'+2d_\D}S\DD_{\D'}S=0$$
has for all $\o\in U'$  a unique solution $S(\o)$ verifying
\begin{itemize}
\item[(i)]
$$\a{S}_{\left\{\begin{subarray}{l} \g\\ U'\end{subarray}\right\}}
\le\cte\ \frac{1}{\k^2} d_\D^{2d}e^{2\g\D}
\a{F}_{\left\{\begin{subarray}{l} \g\\ U'\end{subarray}\right\}};$$
\item[(ii)] $S(\o)$ and $\p_\o S(\o)$ are T\"oplitz at $\i$ and
the T\"oplitz-limits verify
$$\left\{\begin{array}{l}
i\sc{k,\o}S+i[\O(\o)+H(c,\o),S]=(\TT_{\D'}-
\DD_{\D'})\PP_{c}F(c,\o)\\
S-\TT_{\D'+2d_\D}S=\DD_{\D'}S=0;
\end{array}\right.
$$
\item[(iii)]
$$
\l{S}_{\left\{\begin{subarray}{l} \L'+d_\D+2,\g \\
U'\end{subarray}\right\}}\ \le\cte \
\frac{1}{\k^3}d_\D^{2d}e^{2\g\D}
 (1+\l{H}_{\left\{\begin{subarray}{l} \L\\
 U'\end{subarray}\right\}})
\l{F}_{\left\{\begin{subarray}{l} \L',\g\\ U'
\end{subarray}\right\}}
$$
for any
$$\L'\gsim \max(\L,d_\D^2,(d_{\D'})^2,\sup_U\aa{H(\o)}).$$
\end{itemize}
The constant $\cte$ only depends on the dimensions  $d$ and $ \#\AA
$ and on $C_1,C_2,C_3,C_4$.
\end{Prop}

\begin{proof}
We first assume that $F=(\TT_{\D'}-\DD_{\D'})F$.  The
proof is the same as in Proposition~\ref{p63}, with $k=0$.
Notice that the limit equation in (ii) is invariant under $c$-translations,
due to Lemma \ref{l62} (iii).

The proof gives a
$$\L'\gsim \max(\L ,d_\D^2,\D',\sup_U\aa{H(\o)}).$$
In order to get the result we need to estimate
$(\TT_{\D'}-\DD_{\D'})F$ in terms of $F$. This is
done by Lemma \ref{l62}(i)+(iii) and requires a larger $\L'$.
\end{proof}

\subsection{A third equation.}
\label{ss65}
\

\noindent
Consider the equation
\begin{equation}\label{e610}
i\sc{k,\o}S+i(\O(\o)+H(\o))S+iS\II (\O(\o)+{}^t\!H(\o))=F(\o)
\end{equation}
where $F(\o):\LL\times\LL\to\C$ and $\p_\o F(\o)$ are
T\"oplitz at $\i$ and $\II Q$ is defined by
$$(\II Q)_a^b=Q_{-a}^{-b}.$$
(This equation will be motivated in the proof of
Proposition~\ref{p67}.)

Let $U'\sbs U$ be a set such that for all $\o\in U'_\k$ the  small
divisor condition (\ref{e58}) holds for all $a,b$, i.e.
$$\a{\sc{k,\o}+\al(\o)+\be(\o)}\ge\k \quad \forall
\left\{\begin{array}{l}
\al(\o)\in\s((\O+H)(\o))\\
\be(\o)\in\s((\O+H)(\o))).
\end{array}\right.$$

\begin{Prop}\label{p65} Let $0<\k<1$.
Assume that $U$ verifies
(\ref{e51}), that $\O$ is  real diagonal and verifies
(\ref{e52}), and that $H$ verifies (\ref{e54}). Assume also
that $H(\o)$ and
$\p_\o H(\o)$ are T\"oplitz at $\i$ and
$\NN\FF_{\D}$ for all $\o\in U$.

Then the equation $(\ref{e610})$ has for all $\o\in U'$  a unique
solution  $S(\o)$ verifying
\begin{itemize}
\item[(i)]
$$\a{S}_{\left\{\begin{subarray}{l} \g\\ U'\end{subarray}\right\}}
\le\cte\ \frac{1}{\k^2} d_\D^{2d}e^{2\g\D} (1+\a{k})
\a{F}_{\left\{\begin{subarray}{l} \g\\ U'\end{subarray}\right\}};$$
\item[(ii)] $S(\o)$ and $\p_\o S(\o)$ are T\"oplitz at $\i$ and
all T\"oplitz-limits $S(c,\o),\ c\not=0$, are $=0$;
\item[(iii)]
$$
\l{S}_{\left\{\begin{subarray}{l} \L'+d_\D+2,\g \\
U'\end{subarray}\right\}}\ \le\cte \\
\frac{1}{\k^3}d_\D^{2d}e^{2\g\D}
 (1+\a{k}+\l{H}_{\left\{\begin{subarray}{l} \L\\
 U'\end{subarray}\right\}})
\l{F}_{\left\{\begin{subarray}{l} \L',\g\\ U'
\end{subarray}\right\}}
$$
for any
$$\L'\gsim \max(\L,d_\D^2,\D',\sup_U\aa{H(\o)}).$$
\end{itemize}
The constant $\cte$ only depends on the dimensions  $d$ and $\#\AA$
and on $C_1,C_2,C_3,C_4$.
\end{Prop}

\begin{proof} As before we reduce to $\O_a=|a|^2$ and we block
decompose
the equation over
$\EE_{\D}$:
$$i\sc{k,\o}S_{[a]}^{[b]}+
i(\O+H)_{[a]}S_{[a]}^{[b]}+iS_{[a]}^{[b]}(\O+{}^tH)_{-[b]}F_{[a]}^{[b]}.$$ We then repeat the proof as for Proposition
\ref{p63}. There is a difference in the computation of the T\"oplitz
limits. The equation (\ref{e64}) becomes
$$
\begin{array}{rl}
i\sc{k,\o}S_{[a]}^{[b]}(tc) &+i(\O+H)_{[a]}(tc) S_{[a]}^{[b]}(tc) + \\
& +iS_{[a]}^{[b]}(tc)(\O+{}^t\!H)_{[-b]}(-tc)=F_{[a]}^{[b]}(tc)
\end{array}$$
and now
$$\O_{[a]}(tc)X+X\O_{[-b]}(-tc)$$
equals
$$
(t^2\a{c}^2+2t\sc{a,c}+\a{a}^2)X+X(t^2\a{c}^2+2t\sc{b,c}+\a{b}^2)
$$
--  the quadratic terms in $t$ do not cancel! Dividing the equation
by $t^2$ and letting $t\to\i$, the limit equation becomes
$$2\a{c}^2X=0,$$
which has the unique solution $X=0$. Therefore
$$S_{[a]}^{[b]}(tc)\to S_{[a]}^{[b]}(\i c)=0$$
as $t\to\i$,  i.e. the T\"oplitz limits
are always $0$.

In order to estimate the Lipschitz-norm we only need to consider the
analogue of Case~II (even when $\sc{c,a-b}=0$). We have for
$[a]\times[b]\sbs D_{\L'}(c)$
$$\a{a}^2+\a{b}^2 \gsim (\frac{|a|}{|c|})^2
\approx (\frac{|b|}{|c|})^2\gsim(\L')^2.$$
To avoid any problems with $\sc{k,\o}$ and $H$
it is sufficient that  $(\L')^2$ is $\gsim C_1\D'$ and $\gsim\aa{H}$.
\end{proof}

\subsection{The homological equations.}
\label{ss66}
\

\noindent
Let $\O(\o):\LL\times\LL\to gl(2,\C)$  be a real diagonal
matrix, i.e.
$$\O_a^b(\o)\left\{\begin{array}{ll}
\O_a(\o)I & a=b\\ 0& a\not=b\end{array}\right.$$
Consider
\begin{equation}
\label{e611}
\left\{\begin{array}{ll}
\a{\O_a(\o)} \ge C_5>0 &\\
\a{\O_a(\o)+\O_b(\o)} \ge C_5 & a,b\in\LL,\ \o\in U\\
\a{\O_a(\o)-\O_b(\o)}\ge C_5,\ |a|\not=|b| &
\end{array}\right.
\end{equation}

Let $H(\o):\LL\times\LL\to gl(2,\C)$ and $\p_\o H(\o)$ be  T\"oplitz
at $\i$ for all $\o\in U$ and consider
\begin{equation}
\label{e612}
\left\{\begin{array}{l}
\aa{H(\o)}\le \frac{C_5}4\quad \o\in U\\
\l{H}_{ \left\{ \begin{subarray}{l} \L\\ U\end{subarray}\right\} }
\le C_6
\end{array}\right.
\end{equation}
(Here $\aa{\ }$ is the operator norm.)

\begin{Prop}\label{p66}
Let $\D'>0$ and $0<\k<\frac{C_5}2$. Assume that $U$  verifies
(\ref{e51}), that $\O$ is real diagonal and verifies
$(\ref{e52})+(\ref{e53})+(\ref{e611})$, and that $H$ verifies
$(\ref{e54})+(\ref{e612})$. Assume also that $H(\o)$ and $\p_\o
H(\o)$ are $\NN\FF_{\D}$ for all $\o\in U$.

Then there is a subset $U'\sbs U$,
$$\begin{array}{ll}
\Leb(U\sm U')&\le\
\cte
\max(\D',d_\D^2)^{2d+\#\AA-1}\k,
\end{array}
$$
such that for all $\o\in U'$ the following hold:
\begin{itemize}
\item[(i)]
for any $0<\a{k}\le\D'$
$$\a{\sc{k,\o}}\ge \k.$$
\item[(ii)]
for any $\a{k}\le\D'$ and for any vector $F(\o)\in l^2_\g(\LL,\C^2)$
there exists a unique vector $S(\o)\in l^2_\g(\LL,\C^2)$  such that
$$i\sc{k,\o}S+J(\O+H)S= F$$
and satisfying
$$\aa{S}_{\left\{\begin{subarray}{l} \g\\ U'\end{subarray}\right\}}
\le\cte \frac1{\k^2}\D'
d_\D^{2m_*}e^{2\g d_\D}
\aa{F}_{\left\{\begin{subarray}{l} \g\\ U'\end{subarray}\right\}}.$$
\end{itemize}

The constants $\cte$ only depend on $d,\#\AA ,m_*$ and on
$C_1,\dots,C_6$.
\end{Prop}

\begin{proof}
(i) holds outside a set of $\o$ of Lebesgue measure
$\lsim(\D')^{\#\AA}\k$, so it suffices to consider (ii). Let
$$C\left(\begin{array}{cc}
\frac{1}{\sqrt{2}} & \frac{1}{\sqrt{2}} \\
\frac{-i}{\sqrt{2}} & \frac{i}{\sqrt{2}}
\end{array}\right)
$$
and define ${}^t\!C AC:\LL\times\LL\to gl(2,\C)$ through
$$( {}^t\!C AC )_a^b= {}^t\!C A_a^bC.$$

We change to complex coordinates
$ \tilde S =C^{-1}S\quad\textrm{and}\quad \tilde F =C^{-1}F.$
Then the equation becomes
$$i\sc{k,\o}\tilde S -iJ
\left(\begin{array}{cc} 0 & \O+H\\ \O+{}^t\!H & 0
\end{array}\right)\tilde S =\tilde F$$
where $\O,H:\LL\to \C$  are the scalar-valued normal form   matrices
associated to $\O,H$ (see section \ref{ss51})  --  $\O$ is real
symmetric and $H$ is Hermitian.

This equation decouples into two equations for (scalar-valued)
matrices of type
$$i\sc{k,\o}R \pm i(\O+Q)R =G,$$
where $Q=H$ or ${}^t\!H$. By Proposition (\ref{p61})
we can solve these equations uniquely for all
$\o\in U'$ such that
$$|\sc{k,\o}+\al(\o)|\ge\k\quad\forall
\al(\o)\in\s((\O+H)(\o)),\  |k|\le\D'.$$
If $k=0$ this follows from $(\ref{e611})+(\ref{e612})$
since $\k\le\frac{C_5}2$. If $k\not=0$ this follows from
Proposition \ref{p51}.
\end{proof}

\begin{Prop}\label{p67}
Let $\D'>0$ and $0<\k<\frac{C_5}2$. Assume that $U$  verifies
(\ref{e51}), that $\O$ is real diagonal and verifies
$(\ref{e52})+(\ref{e53})+(\ref{e611})$, and that $H$ verifies
$(\ref{e54})+(\ref{e612})$. Assume also that $H(\o)$ and $\p_\o
H(\o)$ are $\NN\FF_{\D}$ for all $\o\in U$.

Then there is a subset $U'\sbs U$,
$$\begin{array}{ll}
\Leb(U-U')&\le\
\cte
\max(\L,\D,\D')^{\exp}
\k^{(\frac1{d+1})^d},
\end{array}
$$
such that for all $\o\in U'$ the following hold:

for any $\a{k}\le\D'$ and for any matrix
$$\left\{\begin{array}{l}
F(\o):\LL\times\LL\to gl(2,\C)\\
F(\o)\ \textrm{symmetric, i.e.}\ F_a^b={}^t\!F_b^a\\
(\pi F)_a^b=0\quad \textrm{when }\
\a{a-b}>\D',
\end{array}\right.$$
there exist symmetric matrices
$S(\o)$ and $H'(\o)$ such that
$$i\sc{k,\o}S+(\O+H)JS-SJ(\O+H)= F-H'$$
and satisfying  -- for any
$$\L'\ge \cte \max(\L,d_\D^2,(d_{\D'})^2)\quad -$$
\begin{itemize}
\item[(i)]
$$
\l{S}_{\left\{\begin{subarray}{l} \L'+d_\D+2,\g\\  U'\end{subarray}
\right\}}\ \le\ \cte \frac1{\k^3}\D' d_\D^{2d}e^{2\g d_\D}
\l{F}_{\left\{\begin{subarray}{l} \L',\g\\
U'\end{subarray}\right\}},$$

\item[(ii)] for $k\not=0$ $H'(\o)=0$ and for $k=0$ $H'(\o)$
 and $\p_\o H'(\o)$ are block diagonal
over $\EE_{\D'}$ and
$$
\l{H'}_{\left\{\begin{subarray}{l} \L'+d_\D+2\\ U'\end{subarray}\right\}}\ \le\
\l{F}_{\left\{\begin{subarray}{l} \L'\\ U'\end{subarray}\right\}}.$$
Moreover, if $F$ is real then $H'(\o)$ and $\p_\o H'(\o)$ are $\NN\FF_{\D'}$
\end{itemize}

The exponent $\exp$ only depends on $d, \#\AA $ and the constants
$\cte$ also depend on $C_1,\dots,C_6$.
\end{Prop}

\begin{proof}
We change to complex coordinates $ \tilde S
={}^t\!CSC\quad\textrm{and}\quad \tilde F ={}^t\!CFC.$ Then the
equation becomes $\tilde F-\tilde H'=$
$$i\sc{k,\o}\tilde S -
i\left(\begin{array}{cc} 0 & \O+H\\ \O+{}^t\!H & 0
\end{array}\right)J\tilde S
-i\tilde S J\left(\begin{array}{cc} 0 & \O+H\\ \O+{}^t\!H & 0
\end{array}\right)$$
where $\O,H:\LL\to \C$  are the scalar-valued normal form   matrices
associated to $\O,H$ (see section \ref{ss51})  --  $\O$ is real
symmetric and $H$ is Hermitian.

If we write
$$F=\left(\begin{array}{cc} F_1 & F_2\\{}^t\!F_2 & F_3
\end{array}\right)$$
then
$$\tilde F=\frac12\left(\begin{array}{cc}
(F_1-F_3)-i(F_2+{}^t\!F_2)  & (F_1+F_3)+i(F_2-{}^t\!F_2)\\
(F_1+F_3)-i(F_2-{}^t\!F_2)  & (F_1-F_3)+i(F_2+{}^t\!F_2)
\end{array}\right),$$
the diagonal parts coming from $(I-\pi)F$ and the off-diagonal parts
from $\pi F$.

The equation decouples into four (scalar-valued)
matrices of the types
$$i\sc{k,\o}R \pm i((\O+Q)R-R(\O+Q)) =G-P,$$
for the off-diagonal terms, and
$$i\sc{k,\o}R \pm i((\O+Q)R+R(\O+{}^t\! Q)) =G-P,$$
for the diagonal terms. Here $Q=H$ or ${}^t\!H$.

Let us first consider the off-diagonal equations. By  the assumption
on $F$, $\TT_{\D'}G=G$, $G$ is T\"oplitz at $\i$ and
$$
\l{G}_{\left\{\begin{subarray}{l} \L',\g\\ U' \end{subarray}\right\}}
\le
\l{F}_{\left\{\begin{subarray}{l} \L',\g\\ U' \end{subarray}\right\}}.$$
Moreover, $G$ is Hermitian if $F$ is real.

If $k\not=0$ we take $P=0$ and we can solve the equation
by Proposition \ref{p63} for all $\o$ such that
$$\a{\sc{k,\o}+\al(\o)-\be(\o)}\ge\k \quad \forall
\left\{\begin{array}{l}
\al(\o)\in\s((\O+H)(\o)_{[a]_{\D}})\\
\be(\o)\in\s((\O+H)(\o)_{[b]_{\D}})
\end{array}\right.$$
for
$$\rm{dist}([a]_{\D},[b]_{\D})\le\D' +2d_\D .$$
The set of such $\o$ is estimated in Proposition \ref{p52}. The
solution is unique if we impose $\TT_{\D'+2d_\D } R-R=0$.

If $k=0$ we take $P= \DD_{\D'} G$ and we can solve the equation
by Proposition \ref{p64} for all $\o$ such that
$$\a{\al(\o)-\be(\o)}\ge\k \quad \forall
\left\{\begin{array}{l}
\al(\o)\in\s((\O+H)(\o)_{[a]_{\D}})\\
\be(\o)\in\s((\O+H)(\o)_{[b]_{\D}})
\end{array}\right.$$
for
$$\rm{dist}([a]_{\D},[b]_{\D})\le\D'+2d_\D \quad\textrm{and}\quad |a|\not=|b|.$$
This condition on $\o$ holds by assumptions
$(\ref{e611})+(\ref{e612})$ since $\k\le\frac{C_5}2$.  The solution
is unique if we impose $\TT_{\D'+2d_\D } R-R=\DD_{\D'} R=0$. $P$ is
estimated by Lemma \ref{l62}(iii).

To treat the diagonal equations let us consider the operators
$$(\RR G)_a^b=G_a^{-b}\ \text{and}\ (\II G)_a^b=G_{-a}^{-b}.$$
Now $\RR G$, $G$ coming from $(I-\pi)F$, is T\"oplitz at $\i$ and
$$
\l{\RR G}_{\left\{\begin{subarray}{l} \L',\g\\ U' \end{subarray}\right\}}
\le
\l{F}_{\left\{\begin{subarray}{l} \L',\g\\ U' \end{subarray}\right\}}.$$
With $T=\RR R$ the equation takes the form
$$i\sc{k,\o}T \pm i((\O+Q)T+T\II(\O+{}^t\!Q)) =\RR G-\RR P.$$
We take $\RR P=0$ and then the result follows from
Proposition~\ref{p65} under the assumption  (\ref{e58}) on $\o$.
This assumption holds for $k=0$ by $(\ref{e611})+(\ref{e612})$ and
for $k\not=0$ on a set $U'$ which is estimated in
Proposition~\ref{p51}.

By construction $H'$ is symmetric. Moreover, for $k=0$
$$(\pi S)_a^b=0\quad\textrm{when }\
\a{a-b}>\D'+2d_\D \
\textrm{or}\ [a]_\D=[b]_\D;$$
and for $k\not=0$
$$(\pi S)_a^b=0\quad\textrm{when }\
\a{a-b}>\D'+2d_\D .$$
These conditions determine $S$ uniquely and symmetry follows from
this.
\end{proof}

\bigskip
\bigskip

\centerline{PART III. KAM}

\bigskip

\section{A KAM theorem}
\label{s7}

\subsection{Statement of the theorem}
\label{ss71}
\

\noindent
Let
$$\OO^\g(\s,\r,\m)=\OO^\g(\s)\times\T_\r^\AA\times{\mathbb{D}}(\m)^\AA$$
be the set of all $\z,\f,r$ such that
$$\z=(\xi,\eta)\in\OO^\g(\s),\ |\Im\f_a|<\r,\ |r_a|<\m\quad\forall
a\in\AA.$$

Let
$$h_\o(\z,r)=h(\z,r,\o)=\sc{\o,r}+\frac12\!\!\sc{\z,(\O(\o)+H(\o))\z}$$
where $\O(\o)$ is a real diagonal matrix with diagonal elements
$\O_a(\o)I$ and $H(\o)$ and $\p_\o H(\o)$ are T\"oplitz at $\i$ and
$\NN\FF_\D$ for  all $\o\in U$. We recall (section \ref{ss51}) that
a matrix $H:\LL\times\LL\to gl(2,\C)$ is $\NN\FF_\D$ if it is real,
symmetric and can be written
$$H=\left(\begin{array}{cc}
Q_1 & Q_2\\ {}^t\!Q_2 & Q_1\end{array}\right)$$
with $Q=Q_1+iQ_2$ Hermitian and
block-diagonal over the decomposition $\EE_\D$ of $\LL$.

We assume (\ref{e51}-\ref{e53})+(\ref{e611}), i.e.

$$U\ \text{is an open subset of}\ \b{|\o|< C_1}\sbs\R^{\#\AA},$$

$$\left\{\begin {array}{l}
\a{\partial_\o^{\nu}(\O_a(\o)-|a|^2)}
\le C_2e^{-C_3\a{a}}, \quad C_3>0\\
(a,\o)\in\LL\times U,\quad\nu=0,1,
\end{array}\right.$$

$$
\left\{\begin{array}{ll}
\sc{\p_\o(\sc{k,\o}+\O_a(\o)),\frac{k}{|k|}}\ge C_4>0 &   \\
\sc{\p_\o(\sc{k,\o}+\O_a(\o)+\O_b(\o)),\frac{k}{|k|}}\ge C_4&
 a,b\in \LL,\ k\in\Z^{\AA}\sm0,\ \o\in U \\
\sc{\p_\o(\sc{k,\o}+\O_a(\o)-\O_b(\o)),\frac{k}{|k|}}\ge
C_4\ &(|a|\not=|b|)
\end{array}\right.$$

$$\left\{\begin{array}{ll}
\a{\O_a(\o)} \ge C_5>0 &\\
\a{\O_a(\o)+\O_b(\o)} \ge C_5 & a,b\in\LL,\ \o\in U\\
\a{\O_a(\o)-\O_b(\o)}\ge C_5,\ |a|\not=|b|. &
\end{array}\right.$$

\begin{Rem} The conditions on the directional derivative
hold trivially for $C_4=\frac12$ if
$$\a{\p_\o\O_a(\o)}\le\frac14\quad\forall (a,\o)\in\LL\times U.$$
\end{Rem}

We also assume (\ref{e54})+(\ref{e612}), i.e.
$$
\left\{\begin{array}{l}
\aa{\p_\o H(\o)}\le \frac{C_4}4\\
\aa{H(\o)}\le \frac{C_5}4\\
\l{H}_{ \left\{ \begin{subarray}{l} \L \\ U\end{subarray}\right\}
}\lsim 1
\end{array}\right.
$$
for some $\L$. (Here $\aa{\ }$ is the operator norm.)

\begin{Rem}
For simplicity we shall assume that  $\g,\s,\r,\s$ are $<1$ and
that $\D,\L$ are $\ge3$.
\end{Rem}

Let
$$f:\OO^\g(\s,\r,\m)\times U\to\C$$
be real analytic in $\z,\f,r$ and $\CC^1$ in $\o\in U$
and let
$$[f]_{\left\{\begin{subarray}{l} \L,\g,\s\\
U,\r,\m\end{subarray}\right\}}\sup_{
\begin{subarray}{c}
\f\in\T_\r^\AA \\ r\in{\mathbb{D}}(\m)^\AA
\end{subarray}}
[f(\cdot,\f,r,\cdot)]_{\left\{\begin{subarray}{l} \L,\g,\s\\
U\end{subarray}\right\}}.$$

\begin{Thm}\label{t71}
Assume that $U$ verifies (\ref{e51}),
that $\O$ is real diagonal and verifies
$(\ref{e52})+(\ref{e53})+(\ref{e611})$, that
$H(\o)$ and $\p_\o H(\o)$ are T\"oplitz at $\i$ and $\NN\FF_{\D}$
for all $\o\in U$, and that $H$ verifies (\ref{e54})+(\ref{e612}).

Then there is a constant $\Cte$ and an exponent $\exp$ such that,
if
$$[f]_{\left\{\begin{subarray}{l}\L,\g,\s\\ U,\r,\m
\end{subarray}\right\}}=\ep
\le\Cte\min(\g,\r,\frac1\L,\frac1\D)^{\exp}\min(\s^2,\m)^2$$
then there is a $U'\sbs U$ with
$$\Leb(U\sm U')\le\cte\ep^{\exp'}$$
such that for all $\o\in U'$ the following hold: there is an
analytic symplectic diffeomorphism
$$\Phi: \OO^{0}(\frac\s 2,\frac\r 2,\frac\m 2)\to\OO^{0}(\s,\r,\m)$$
and a vector $\o'$
such that $(h_{\o'}+f)\circ\Phi$ equals (modulo a constant)
$$\sc{\o,r} +
\frac12\!\!\sc{\z,(\O+H')(\o)\z}
+ f'(\z,\f,r,\o)$$
where
$$\p_\z f'=\p_r f'= \p_\z^2 f'=0\ \text{ for }\ \z=r=0$$
and
$$H'=\left(\begin{array}{cc}
Q'_1 & Q'_2\\ {}^t\!Q'_2 & Q'_1\end{array}\right)$$ with
$Q'=Q'_1+iQ'_2$  Hermitian and block diagonal
$$(Q')_a^b=0\quad \forall |a|\not=|b|.$$

Moreover
$\Phi =(\Phi_\z,\Phi_\f,\Phi_r)$
verifies, for all $(\z,\f,r)\in
\OO^{0}(\frac\s 2,\frac\r 2,\frac\m 2)$
$$
\aa{\Phi_\z-\z}_0+\a{\Phi_\f-\f}+\a{\Phi_r-r}\le
\cte\be(\g,\s,\r,\m,\L,\D,\o)\ep,$$ and the mapping $\o\mapsto\o'(\o)$
verifies
$$\a{\o'-\mathrm{id}}_{\CC^1(U')}\le\cte\frac{\ep}\m.$$

The exponents $\exp,\exp'$  only depend on $d,\#\AA,m_*$  while the
constants $\Cte,\cte$ also depends on $C_1,\hdots,C_5$.
\end{Thm}

\begin{Rem}
Each block-component of $\O'$ is of finite dimension but in general there is no
uniform bound --  they may be of
arbitrarily large dimension. Due to this lack of uniformity we
loose, in our estimates, all exponential decay in the space modes.
However, if there were a uniform bound  --  as happens in some cases
\cite{GY06}  -- we would retain some exponential decay.
\end{Rem}

\begin{Rem}
It follows from the proof that $\Phi$ is of the form
$$\left\{
\begin{array}{l}
\Phi_\z(\z,\f,r)=z(\f)+Z(\f)\z\\
\Phi_\f(\z,\f,r)=\f+a(\f) \\
\Phi_r(\z,\f,r)=r+b(\z,\f)+c(\f)r
\end{array}\right.$$
where $b(\z,\f)$ is quadratic in $\z$, because $\Phi$ is a composition of mappings of this form.

If $f$ does not depend on $r$, then
$$a=c=0
\quad\textrm{and}\quad \o'=\o,$$
because $\Phi$ is a composition of mappings of this form,
and it preserves Hamiltonians of this form.

If $f(\z,\f)=\frac12\sc{\z,F(\f)\z}$, then also
$$z=0
\quad\textrm{and}\quad b(\z,\f)=\frac12\sc{\z,B(\f)\z},$$
because $\Phi$ is a composition of mappings of this form,
and it preserves Hamiltonians of this form.
\end{Rem}

Since the consequences of the theorem are discussed in  the
introduction, let us instead here discuss a special case.
Consider a linear non-autonomous
Hamiltonian system with quasiperiodic coefficients
$$
\dot\zeta=J\big(\Omega+H(\o)+\ep F(\varphi,\omega)\big)\zeta,\quad \dot\varphi=\omega$$
where $\O$ and $H(\o)$ are as in Theorem \ref{t71} and $F$ is symmetric and T\"oplitz at $\i$
and
$$\l{F(\f,\cdot)}]_{\left\{\begin{subarray}{l} \L,\g\\
U\end{subarray}\right\}}<\i$$
for $|\Im\f|<\r$ and for some $\g>0$. Then, by Young's inequality (\ref{e22}),
$$\aa{F(\f,\o)\z}_{\g'}\le (\frac1{\g-\g'})^{d+m_*}\a{F(\f,\o)}_\g\aa{\z}_{\g'}\quad \forall \g'<\g$$
and
$$|\sc{\z,F(\f,\o)\z}|\le  (\frac1{\g})^{d+m_*}\a{F(\f,\o)}_\g\aa{\z}_0^2.$$
Therefore we can apply Theorem \ref{t71}+Remark to the Hamiltonian
$$
h+\ep f=\sc{\omega,r}+\frac12\,\sc{\zeta,(\Omega+H(\o)+F(\f,\o))\zeta}
$$
If $\ep$ is sufficiently small, it gives a mapping $\Phi$ such that
$$(h+\ep f)\circ\Phi(\z,\f,r)\sc{\omega,r}+\frac12\,\sc{\zeta,(\Omega+H'(\o))\zeta}$$
with
$$\Phi(\z,\f,r)\left(\begin{array}{c}
Z(\f)\z\\
r+\frac12\sc{\z,B(\f)\z}\\
\f\end{array}\right).$$
>From this form and from the symplectic character of $\Phi$ we derive that
$$\sc{\p_\f Z(\f),\o}=J(\O+H+F(\f))Z(\f)-Z(\f)J(\O+H').$$
This implies that the mapping
$$(\z,\f)\mapsto (w=Z(\f)\z,\f)$$
reduces the linear non-autonous system to autonomous system
$$
\dot w=J\big(\Omega+H'(\o)\big)\zeta,\quad \dot\varphi=\omega.$$
Notice also that $J(\O+H)$ is block-diagonal with purely imaginary eigenvalues.

\bigskip

\subsection{Application to the Schr\"odinger equation}
\label{ss72}
\

\noindent
Consider a non-linear Schr\"odinger equation
$$
  -i\dot u=-\D u+V(x)*u+\ep\frac{\p F}{\p\bar{u}}(x,u,\bar u),
  \quad u=u(t,x),\;x\in\T^d,\quad(*)
$$
where $V(x)=\sum \hat V(a)e^{i\sc{a,x}}$ is an analytic function
with $\hat V$ real and where $F$ is real analytic in $\Re u,\Im u$
and in $x\in\T^d$.

Let $\AA\sbs \Z^d$ be a finite set and consider a function
$$u_1(\f,x)=\sum_{a\in\AA} \sqrt{p_a} e^{i\f_a}e^{i\sc{a,x}},\quad
p_a>0,$$
such that $(x,u_1(\f,x),\bar u_1(\f,x))$ belongs to the domain of $F$
for all $(x,\f)\in\T^d\times\T^{\AA}$. Then
$$u_1(t,x)= u_1(\f+t\o,x)$$
is a solution of $(*)$ for $\ep=0$.

Let $\LL$ be the complement of $\AA$
and let
$$\begin{array}{l}
\o=\b{\o_a=|a|^2+\hat V(a): a\in\AA}\\
\O=\b{\o_a=|a|^2+\hat V(a): a\in\LL}
\end{array}$$

Let $V$ depend $\CC^1$ on a parameter $w\in W\sbs\R^{\#\AA}$ and
assume that it satisfies conditions analogous to
(\ref{e51}-\ref{e53} )+(\ref{e611}), i.e.

$$W\ \text{is an open subset of}\ \b{|w|< C_1}\sbs\R^{\#\AA},$$

$$\left\{\begin {array}{l}
\a{\partial_w^{\nu}(\O_a(w)-|a|^2)}
\le C_2e^{-C_3\a{a}}, \quad C_3>0\\
(a,w)\in\LL\times W,\quad\nu=0,1,
\end{array}\right.$$

$$
\left\{\begin{array}{ll}
\sc{\p_w(\sc{k,\o(w)}+\O_a(w)),\frac{k}{|k|}}\ge C_4>0 &   \\
\sc{\p_w(\sc{k,\o(w)}+\O_a(w)+\O_b(w)),\frac{k}{|k|}}\ge C_4&
 a,b\in \LL,\ k\in\Z^{\AA}\sm0,\ w\in W \\
\sc{\p_w(\sc{k,\o(w)}+\O_a(w)-\O_b(w)),\frac{k}{|k|}}\ge C_4\
&(|a|\not=|b|)
\end{array}\right.$$

$$\left\{\begin{array}{ll}
\a{\O_a(w)} \ge C_5>0 &\\
\a{\O_a(w)+\O_b(w)} \ge C_5 & a,b\in\LL,\ \o\in U\\
\a{\O_a(w)-\O_b(w)}\ge C_5,\ |a|\not=|b|. &
\end{array}\right.$$

We also assume that the mapping
$$W\ni w\mapsto \o(w)=\b{\o_a=|a|^2+\hat V(a,w); a\in \AA}\sbs U$$
is a diffeomorphism whose inverse is bounded in the $\CC^1$-norm,
i.e.
\begin{equation}
\a{\o^{-1}}_{\CC^1}\le C_6.
\end{equation}

\begin{Thm}\label{t72}
For $\ep$ sufficiently small, there is a subset
$W'\sbs W$,
$$\Leb(W\sm W')\le\cte\ep^{\exp},$$
such that on  $W'$ there is an
$ u(\f,x)$, analytic in $\f\in\T^d_{\frac\r 2}$ and of class
$\CC^{m_*-d}$ in $x\in\T^d$, with
$$\sup_{|\Im\f|<\frac\r2}\aa{u(\f,\cdot)-u_1(\f,\cdot)}_{H^{m_*}(\T^d)}
\le \be\ep,$$
and there is a $\o':W'\to U$,
$$\a{\o'-\o}_{\C^1(W')}\le \be\ep,$$
such that
$$u(t,x)=u(\f+t\o'(w),x)$$
is a solution of $(*)$ for any $w\in W'$. $\be$ is a constant that
depends on the dimensions $d,\#\AA,m_*$, the constants
$C_1,\dots,C_6$ and on $w$ and $F$.

Moreover, the linearized equation
$$
\begin{array}{c}
-i\dot v=\D v+V(x)*v+\ep\frac{\p^2 F}{\p \bar{u}^2}(x,u(t,x),
\bar{u}(t,x))\bar{v}+\\
\ep\frac{\p^2 F}{\p u\p\bar u}(x,u(t,x),\bar{u}(t,x))v
\end{array}
$$
is reducible to constant coefficients and has only
time-quasi-periodic solutions  -- except for a $(\#\AA)$-dimensional
subspace where solutions may increase at most linearly  in $t$.
\end{Thm}

\begin{proof}
We write
$$\left\{\begin{array}{l}
u(x)=\sum_{a\in\Z^d}u_ae^{i<a,x>}\\
\oli{u(x)}=\sum_{a\in\Z^d}v_{a}e^{i<-a,x>} \quad(v_a=\bar u_a),
\end{array}\right.$$
and let
$$\zeta_a\left(\begin{array}{c}
\xi_a\\
\eta_a
\end{array}\right)
\left(\begin{array}{c}
\frac{1}{\sqrt{2}}(u_a+v_a)\\
\frac{-i}{\sqrt{2}}(u_a-v_a)
\end{array}\right).
$$
In the symplectic space
$$
\b{(\xi_a,\eta_a):a\in\Z^d}=\R^{\Z^d}\times \R^{\Z^d},\quad
\sum_{a\in\Z^d}d\xi_a\wedge d\eta_a,
$$
the equation becomes a Hamiltonian equation in infinite  degrees of
freedom. The Hamiltonian function has an integrable part
$$\frac12\sum_{a\in\Z^d}(\a{a}^2+\hat V(a))(\xi_a^2+\eta_a^2)$$
plus a perturbation.

In a neighborhood of the unperturbed solution
$$\frac12(\xi_a^2+\eta_a^2)=p_a,\quad a\in\AA,$$
we introduce the action angle variables $(\f_a,r_a)$ (notice that
each $p_a>0$ by assumption), defined through the relations
$$\begin{array}{l}
\xi_a=\sqrt{2(r_a+p_a)}\cos(\f_a)\\
\eta_a=\sqrt{2(r_a+p_a)}\sin(\f_a).
\end{array}
$$
The integrable part of the Hamiltonian becomes
$$h(\zeta,r,\o)= \sc{\o,r} +
\frac12\sum_{a\in\LL}\O_a(\o)(\xi_a^2+\eta_a^2),
$$
while the perturbation
$$\ep f(u,\bar u)=\ep \int_{\T^d} F(x,u(x)\bar u(x))dx$$
will be a function of $\z,\f,r$. If we write
$$G(x,u_1,\bar u_1,u,\bar u)=F(x,u_1+u,\bar u_1+\bar u)$$
then $G$ is an analytic function in $x,u,\bar u$ which depends
analytically on $\f,r$. Then one verifies (see Lemma 1 in \cite{EK1}) that, since
$m_*>\frac d2$, there exist $\g,\s,\r,\m$ such that  $f$ is real
analytic on $\OO^\g(\s,\r,\m)$ and that $f$ has the
T\"oplitz-Lipschitz-property:
\begin{equation}
[f]_{\left\{\begin{subarray}{l} \L,\g,\s\\ U,\r,\m
\end{subarray}\right\}} \le C_7
\end{equation}
for some constant $C_7$.

The assumptions of Theorem \ref{t71} are now fulfilled and  gives
the result.
\end{proof}

\section{Proof of theorem}
\label{s8}

\subsection{Preliminaries}
\label{ss81}
\

\noindent
Let
$$f:\OO^\g(\s,\r,\m)\times U\to\C$$
be real analytic in $\z,\f,r$ and $\CC^1$ in $\o\in U$
and consider
$$[f]_{\left\{\begin{subarray}{l} \L,\g,\s\\ U,\r,\m
\end{subarray}\right\}}.$$

\begin{Not}
We let
$$
\al=\left(\begin{array}{cc}\g&\s\\ \r&\m\end{array}\right),
$$
and we write this norm as
$$
[f]_{\left\{\begin{subarray}{l} \L\\ U\end{subarray}\,
\al\right\}}.
$$
\end{Not}

\begin{Rem}
We shall assume that all $\g,\s,\r,\m$ are $<1$, that
$0<\s-\s'\approx\s,\ 0<\m-\m'\approx\m$ and that $\L,\D\ge3$.
\end{Rem}

{\it Cauchy estimates.} It follows by Cauchy estimates that
\begin{equation}\label{e81}
\begin{array}{l}
[\p_\f f]_{\left\{\begin{subarray}{l} \L\\ U\end{subarray}\,
\al'\right\}}\lsim
\frac1{\r-\r'}[f]_{\left\{\begin{subarray}{l} \L\\
 U\end{subarray}\, \al\right\}}\\

[\p_r f]_{\left\{\begin{subarray}{l} \L\\ U\end{subarray}\,
\al'\right\}} \lsim \frac1{\m-\m'}[f]_{\left\{\begin{subarray}{l}
\L\\  U\end{subarray}\,\al\right\}}.
\end{array}
\end{equation}

{\it Truncation.} We obtain $\TT_\D f$ from $f$ by: 1) truncating
the  Taylor expansion in $\z$ at order 2; 2) truncating the Taylor
expansion in $r$ at order 0 for the first and the second order term
in $\z$ and at order 1 for the zero'th order term in $\z$; 3)
truncating the Fourier modes at order $\D$; 4) truncating the space
modes of the second order term in $\z$ at order $\D$. Formally
$\TT_\D f$ is
$$\begin{array}{c}
\sum_{|k|\le\D}
[\hat f(0,k,0,\o)+\p_r \hat f(0,k,0,\o)r+
\sc{\p_\z\hat f(0,k,0,\o),\z}\\
+\frac12\!\!\sc{\z,\TT_\D\p_\z^2\hat f(0,k,0,\o)\z}]e^{i\sc{k,\f}}.
\end{array}$$
We have
\begin{equation}\label{e82}
[\TT_\D f]_{\left\{\begin{subarray}{l} \L\\ U\end{subarray}\,
\al\right\}}\lsim \D^{\#\AA} [f]_{\left\{\begin{subarray}{l} \L\\
U\end{subarray}\, \al\right\}}
\end{equation}
and
\begin{equation}\label{e83}
[f-\TT_\D f]_{\left\{\begin{subarray}{l} \L\\
U\end{subarray}\,\al'\right\}}\lsim A(\al,\al',\D)
[f]_{\left\{\begin{subarray}{l} \L\\ U\end{subarray}\,\al\right\}},
\end{equation}
where $A(\al,\al',\D)$  is
$$
(\frac{\s'}{\s})^3+(\frac{\s'}{\s}+\frac{\m'}{\m})\frac{\m'}{\m}
+
(\frac1{\r-\r'})^{\#\AA}e^{-\D(\r-\r')}+e^{-\D(\g-\g')}.$$

This follows from Proposition~\ref{p32}, from Cauchy estimates in
$r$ and $\f$, and from formula (\ref{e265}).

{\it Poisson brackets.} The Poisson bracket is defined by
$$\b{f,g}=\sc{\p_\z f,J\p_\z g}+\p_\f f\p_r g-\p_r f\p_\f g.$$
If $g$ is a quadratic polynomial in $\z$, then
\begin{equation}\label{e84}
[\b{f,g}]_{\left\{\begin{subarray}{l} \L+3\\
U\end{subarray}\,\al'\right\}}\lsim B(\g-\g',\s,\r-\r',\m,\L)
[f]_{\left\{\begin{subarray}{l} \L\\ U\end{subarray}\,\al\right\}}
[g]_{\left\{\begin{subarray}{l} \L\\ U\end{subarray}\,\al\right\}},
\end{equation}
where
$$B=\L^2\frac1{\s^2}(\frac1{\g-\g'})^{d+m_*}+
\frac1{\r-\r'}\frac1{\m}.$$

If also $f$ is a quadratic polynomial in $\z$ and, moreover,
independent of $\f$ and of the form
$$\sc{a,r}+\frac12\!\sc{\z,A\z},$$
then
\begin{equation}\label{e844}
[\b{f,g}]_{\left\{\begin{subarray}{l} \L+3\\
U\end{subarray}\,\al'\right\}}\lsim
B(\bar\g-\g',\s_1,\bar\r-\r',\m_1,\L)
[f]_{\left\{\begin{subarray}{l} \L\\ U\end{subarray}\,\al_1\right\}}
[g]_{\left\{\begin{subarray}{l} \L\\
U\end{subarray}\,\al_2\right\}},
\end{equation}
$$\al_i=\left(\begin{array}{cc}\g&\s_i\\ \r&\m_i\end{array}\right),\quad
i=1,2.$$
and
$\bar\g=\min(\g_1,\g_2),\ \bar\r=\min(\r_1,\r_2)$.
\footnote{In the expression for $B$ we have assumed that
$0<\s_j-\s'\approx\s_,\ 0<\m_j-\m'\approx\m_j$, $j=1,2$.}

In both cases, the first term to the right (in the expression for
$\b{f,g}$ above) is estimated by Proposition~\ref{p33} and the other
two terms by Cauchy estimates.

We shall use both these estimates. Notice that (\ref{e844}) is much
better  than (\ref{e84}) when $\s_2,\m_2$ are much smaller than
$\s_1,\m_1$.

{\it Flow maps.} Let
$$s=\TT_\D
s=S_0(\f,r,\o)+\sc{\z,S_1(\f,\o)}+\frac12\!\sc{\z,S_2(\f,\o)\z}.$$
Notice that, since $s=\TT_\D$, $S_0$ is of first order in $r$.
Consider the Hamiltonian vector field
$$
\frac{d}{dt}\left(\begin{array}{c}\z\\\f\\r\end{array}\right)\left(\begin{array}{c}J\p_\z s\\\p_r s\\-\p_\f s\end{array}\right)\left(\begin{array}{c} JS_1(\f,\o)+JS_2(\f,\o)\z\\\p_r S_0(\f,0,\o)\\
-\p_\f s(\z,\f,r,\o)\end{array}\right)
$$
and let
$$\Phi_t\left(\begin{array}{c}\z_t\\\f_t\\r_t\end{array}\right)
\left(\begin{array}{c} \z+b_t(z,\o)+B_t(z,\o)\z\\
z+g_t(\z,z,\o)\end{array}\right)$$
be the flow. Here we have denoted $\f$ and $r$ by $z$.

Assume that
\begin{equation}\label{e85}
[s]_{\left\{\begin{subarray}{l} \L\\
U\end{subarray}\,\al\right\}}=\ep
\lsim\min((\r-\r')\m,(\g-\g')^{d+m_*}\s^2).
\end{equation}
Then for $|t|\le1$ we have:
$$\Phi_t:\OO^{\g''}(\s',\r',\m')\to\OO^{\g''}(\s,\r,\m),\quad
\forall\g''\le\g';$$
\begin{equation}\label{e86}
[g_t]_{\left\{\begin{subarray}{l}\L, \g',\s'\\
U,\r',\m'\end{subarray} \right\}} \lsim
\frac{\ep}{\m}\quad\textrm{or}\quad\frac{\ep}{\r-\r'}
\end{equation}
depending on if $g$ is an $\f$-component or a $r$-component;
\begin{equation}\label{e87}
\aa{b_t+B_t\z}_{\left\{\begin{subarray}{l} \g''\\
U,\r'\end{subarray}\right\}} \lsim
((\frac1{\g-\g'})^{m_*}+(\frac1{\g-\g'})^{d+m_*}\frac1{\s}\aa{\z}_{\g''})
\frac{\ep}{\s}
\end{equation}
for all $\g''\le\g'$;
\begin{equation}\label{e88}
\l{B_t}_{\left\{\begin{subarray}{l}\L+6, \g'\\
U,\r'\end{subarray}\right\}} \lsim
\L^2(\frac1{\g-\g'})\frac{\ep}{\s^2}.
\end{equation}
Moreover, for $1\ge\bar\s\ge\s'$ and $1\ge\bar\m\ge\m'$,
$\Phi_t$ has an analytic (because polynomial in $\z$ and $\r$)
extension to $\OO^{\g''}(\bar{\s},\r',\bar{\m})$ for all $\g''\le\g'$
and verifies on this set
\begin{equation}\label{e89}
\left\{\begin{array}{l} \aa{\z_t-\z}\lsim
(\frac1{\g-\g'})^{d+m_*}(\frac{\bar{\s}}{\s}+1) \frac{\ep}{\s} \\
\a{\f_t-\f}\lsim \frac{\ep}\m\\
\a{r_t-r}\lsim (\frac1{\r-\r'})(\frac{\bar{\m}}{\m}+(\frac{\bar\s}\s)^2+1)\ep.
\end{array}\right.
\end{equation}

\begin{proof}
We have $\f_t=\f+a_t(\f,\o)$ and since
$$|\p_r S_0(\f,0,\o)|\lsim\frac{\ep}{\m},\qquad \forall
\f\in\T_\r^{\AA},$$ $\f_t$ remains in $\T_\r^{\AA}$ for $|t|\le1$ if
$\frac{\ep}{\m}\lsim(\r-\r')$. The $\o$-derivative verifies
$$\frac{d}{dt}(\p_\o\f_t)=\p_\o\p_r S_0(\f,0,\o)+\p_\f\p_r
S_0(\f,0,\o)(\p_\o\f_t)$$
and can be solved explicitly by an integral formula.
This gives (\ref{e86}) for $z=\f$ and the $\f$-part of (\ref{e89}).

For a fixed $\o$ (\ref{e87}) follows from the first part of
Proposition~\ref{p34}(i) if $|JS_2|_\g\lsim(\g-\g')^d$, i.e. if
$\ep\lsim(\g-\g')^d\s^2$. This also gives the $\z$-part of
(\ref{e89}). In order to get $\aa{\z_t-\z}_{\g'}\le\s-\s'\approx\s$
for $\aa{\z}_{\g'}\le\s$ we need $\ep\lsim(\g-\g')^{d+m_*}\s^2$.
(\ref{e88}) follows from the second part of Proposition
\ref{p34}(i). The $\o$-derivative of $\z_t$ satisfies
$$\frac{d}{dt}(\p_\o\z_t)=\p_\o JS_1(\f,0,\o)+\p_\o
JS_2(\f,0,\o)\z_t+
JS_2(\f,0,\o)(\p_\o\z_t)$$
which is solved in the same way.

$r_t=r+c_t(\z,\f,\o)+d_t(\f,\o)r$ and for a fixed $\o$ (\ref{e86})
follows from Proposition \ref{p34}(ii) if
$\ep\lsim(\r-\r')(\m-\m')\approx(\r-\r')\m$. The $\o$-derivative
satisfies a similar equation which  is solved in the same way. The
$r$-part of (\ref{e89}) follows from these estimates since $r_t$ is
linear in $r$.
\end{proof}

{\it Composition.} Consider now the composition $f(\Phi_t,\o)$. If
\begin{equation}\label{e810}
\ep\lsim\min((\r-\r')\m,(\g-\g')^{d+m_*+1}\s^2)\sqrt{\g-\g'}
\end{equation}
then
\begin{equation}\label{e811}
[f(\Phi_t,\cdot)]_{\left\{\begin{subarray}{l}\L+18\\
U\end{subarray}\,\al'\right\}} \lsim\L^{14}
[f]_{\left\{\begin{subarray}{l}\L\\ U\end{subarray}\,\al \right\}}.
\end{equation}

\begin{proof}
Consider first a fixed $\o$. We have
$$\aa{\z_t(\z,z)-\z}_{\g'}<\s-\s'\quad\forall
(\z,z)\in\OO^{\g'}(\s') \times \T_{\r'}^{\AA} \times{\mathbb{D}}(\m')^{\AA}$$
by (\ref{e87})+(\ref{e810}), and we have
$$\a{g_t(\z,z)}< \frac12(\m-\m')\ \textrm{or}\ \frac12(\r-\r')\quad\forall
(\z,z)\in\OO^{0}(\s') \times \T_{\r'}^{\AA} \times
{\mathbb{D}}(\m')^{\AA} ,$$ depending on if $g$ is an $r$-component
or a $\f$-component, by  (\ref{e86})+(\ref{e810}). By Proposition
\ref{p35} we get
$$ [f(\Phi_t(\cdot,\o),\o)]_{\left\{\begin{subarray}{l}\L+12,\g'',\s'\\
\quad \r',\m'\end{subarray}\right\}}
\lsim A\ [f(\cdot,\o)]_{\left\{\begin{subarray}{l}\L+6,\g',\s\\\quad \r,\m
\end{subarray}\right\}},$$
where
$$A=\max(1, \al,\L^2\frac1{\g'-\g''}\al^2)$$
and
$$\begin{array}{l}
\al= \frac1{\m-\m'}[r_t-r]_{\left\{\begin{subarray}{l}\L+6, \g',\s'\\
\quad  \r',\m'\end{subarray}\right\}}
+
\frac1{\r-\r'}[\f_t-\f]_{\left\{\begin{subarray}{l}\L+6, \g',\s'\\
\quad  \r',\m'\end{subarray}\right\}}\\
+ (\frac1{\g'-\g''})^{d+m_*}
\l{B_t}_{\left\{\begin{subarray}{l}\L+6, \g'\\ \quad
\r'\end{subarray} \right\}}.\end{array}$$ If we choose
$\g'-\g''=\g-\g'$, then (\ref{e86})+(\ref{e88}) and the bound
(\ref{e810}) gives $A\lsim\L^6$.

Consider  now the dependence on $\o$. We have
$$\p_\o(f(\Phi_t))=\p_\o f(\Phi_t) + \sc{\p_z f(\Phi_t),\p_\o g_t}
+ \sc{\p_\z f(\Phi_t),\p_\o \z_t}.$$

The first term is a composition and we get the same estimate as
above but with $f$ replaced by $\p_\o f$.

The second term is a finite sum of products, each of which is
estimated  by Proposition \ref{p31}(i), i.e.
$$ [\sc{\p_z f( \Phi_t,\o),\p_\o
g_t}]_{\left\{\begin{subarray}{l}\L+12,\g'',\s'\\
 \quad \r'',\m''\end{subarray}\right\}}
\lsim [\p_z
f(\Phi_t,\o)]_{\left\{\begin{subarray}{l}\L+12,\g'',\s'\\ \quad
\r'',\m''
\end{subarray}\right\}}
[\p_\o g_t]_{\left\{\begin{subarray}{l}\L+12,\g'',\s'\\ \quad
\r'',\m'' \end{subarray}\right\}}.$$ The first factor is a
composition which is estimated as above: if we take
$\r'-\r''=\r-\r'$ and $\m'-\m''=\m-\m'$, then we get
$$\lsim\L^6
[\p_z f(\cdot,\o)]_{\left\{\begin{subarray}{l}\L+6,\g',\s\\ \quad
\r',\m' \end{subarray}\right\}} [\p_\o
g_t]_{\left\{\begin{subarray}{l}\L+12,\g'',\s'\\ \quad
\r',\m'\end{subarray}\right\}}.$$ Using Cauchy estimates for the
first factor and (\ref{e86})+(\ref{e88}) for the second factor gives
$$\lsim \L^6
[f(\cdot,\o)]_{\left\{\begin{subarray}{l}\L+6,\g',\s\\ \quad \r,\m
\end{subarray}\right\}}.$$

The third term is a composition of the function
$$\tilde f=\sc{\p_\z f,(\p_\o\z_t)\circ \Phi_{-t}}$$
with $\Phi_t$. Evaluating $\tilde f$ we find that it
has the form $\sc{\p_\z f, \tilde b_t+\tilde B_t\z}$ where
$$\begin{array}{l}
\tilde b_t=\p_\o b_t(\f_{-t})+ \p_\o B_t(\f_{-t}) b_{-t}\\
\tilde B_t= \p_\o B_t(\f_{-t}) +\p_\o B_t(\f_{-t})B_{-t}.
\end{array}$$

For $\f\in \T_{\r''}^{\AA}$ we get by (\ref{e86})+(\ref{e810}) that
$$\a{\f_{-t}-\f}\le \r'-\r''=\r-\r',$$
so $\tilde b_t$ and $\tilde B_t$ are defined on  $\T_{\r''}^{\AA}$.
By (\ref{e87})+(\ref{e810})
$$\aa{\tilde b_t}_{\g'}\le\s-\s',$$
and by (\ref{e88})+(\ref{e810}) and the product formula (\ref{e28})
$$\l{\tilde B_t}_{\left\{\begin{subarray}{l}\L+9,\g'\\\ \r''
\end{subarray}\right\}}
\lsim
\L^6(\frac1{\g-\g'})\frac{\ep}{\s^2},$$
so by Proposition \ref{p31}(ii-iii) and (\ref{e810}) we obtain
$$[\tilde f]_{\left\{\begin{subarray}{l}\L+9,\g',\s'\\ \quad \r'',\m'
\end{subarray}\right\}}
\lsim\L^8 [f]_{\left\{\begin{subarray}{l}\L+6,\g,\s\\ \quad
\r'',\m'\end{subarray}\right\}}.$$

Finally by the same argument as above we get
$$ [\tilde
f(\Phi_t(\cdot,\o),\o)]_{\left\{\begin{subarray}{l}\L+15,\g'',\s''\\
 \ \ \r''',\m''\end{subarray}\right\}}
\lsim \L^6 [\tilde
f(\cdot,\o)]_{\left\{\begin{subarray}{l}\L+9,\g',\s'\\ \quad
\r'',\m'\end{subarray}\right\}},$$ if we choose
$\r''-\r'''=\r'-\r''$, $\s'-\s''=\s-\s'$ and $\m'-\m''=\m-\m'$. This
completes the proof.
\end{proof}

\subsection{A finite induction}
\label{ss82}
\

\noindent
Let
$$h(\z,r,\o)=\sc{\o,r}+\frac12\!\!\sc{\z,(\O(\o)+H(\o))\z}$$ satisfy

(\ref{e51}-\ref{e54})+(\ref{e611}-\ref{e612})
and let $H(\o)$ and $\p_\o H(\o)$  be $\NN\FF_\D$. Let
$$f:\OO^\g(\s,\r,\m)\times U\to\C$$
be real analytic in $\z,\f,r$ and $\CC^1$ in $\o\in U$
and consider
$$[f]_{\left\{\begin{subarray}{l}\L\\
U\end{subarray}\,\al\right\}}=\ep,\quad
\al=\left(\begin{array}{cc}\g&\s\\ \r&\m\end{array}\right).$$

Besides the assumption that all constants $\g,\s,\r,\m$ are $<1$ and
that $\D,\L$ are $\ge3$, we shall also assume that
$$\m=\s^2\quad\textrm{and}\quad d_\D\g\le 1.$$
The first  assumption is just for convenience, but the second is
forced upon us by the occurrence of a factor $e^{d_\D\g}$ in the
estimates of Propositions~\ref{p66} and \ref{p67}  which we must
control.

Fix $\r'<\r$, $\g'<\g$ and $0<\k<1$  and let
$$\D'=(\log(\frac1\ep))^2\frac1{\min(\g-\g',\r-\r')},\quad
n=[\log(\frac1\ep)].$$

Define for $1\le j\le n$
$$
\begin{array}{ll}
\ep_{j+1}=(\frac{\ep}{\s^2\k^3})\ep_j &
\ep_1=\ep,\ \\
\L_{j+1}=\L_j +d_{\D}+ 23, & \L_1=\cte \max(\L,d_\D^2,(d_{\D'})^2)\\
\g_j=\g-(j-1)\frac{\g-\g'}n,& \r_j=\r-(j-1)\frac{\r-\r'}n\\
\s_{j+1}=(\frac{\ep}{\s^2\k^3})^{\frac13}\s_j& \s_1=\s\\
\m_{j+1}=(\frac{\ep}{\s^2\k^3})^{\frac23}\m_j\ &\m_1=\m.
\end{array}$$
\footnote{The constant in the definition of $\L_1$ is the one in
Proposition~\ref{p67}.}

We have the following proposition.

\begin{Prop}\label{p81}
Under the above assumptions there exist a constant $\Cte$ and  an
exponent $\exp_1$ such that if
$$\ep\ \le\ \k^3\Cte\min(\g-\g',\r-\r',\frac1\D,\frac1\L,
\frac1{\log(\frac1\ep)})^{\expo_1}\min(\s^2,\m),$$
then there is a subset $U'\sbs U$,
$$\Leb(U\sm U')\le\cte \ep^{\exp_2},$$
such that for all $\o\in U'$ the following holds for $1\le j\le n$:
there is an analytic symplectic diffeomorphism
$$\Phi_j:\OO^{\g''}(\s_{j+1},\r_{j+1},\m_{j+1})\to\OO^{\g''}(\s_j,\r_j,\m_j),
\quad\forall \g''\le\g_{j+1},$$
such that
$$(h+h_1+\hdots+h_{j-1}+f_j)\circ\Phi_j=h+h_1+\hdots+h_j+f_{j+1}$$
$(f_1=f)$ with
\begin{itemize}
\item[(i)]
$$h_j=c_j(\o)+\sc{\chi_j(\o),r}+\frac12\!\sc{\z,H_j(\o)\z},$$
$H_j(\o)$ and $\p_\o H_j(\o)$ in $\NN\FF_{\D'}$, and
$$
[h_j]_{\left\{\begin{subarray}{l}\L_j\\ U'\end{subarray}\,\al_j \right\}}
\le\be^{j-1}\ep_j$$

\item[(ii)]
$$[f_{j+1}]_{\left\{\begin{subarray}{l}\L_{j+1}\\
 U'\end{subarray}\,\,\al_{j+1}\right\}}
\le \be^j\ep_{j+1},$$
\end{itemize}
for some
$$\be\lsim\cte \max(\frac1{\g-\g'},\frac1{\r-\r'},\L,\D,
\log(\frac1{\ep}))^{\exp_3}.$$

Moreover, for $1\ge\bar\s\ge\s_{j+1}$ and $1\ge\bar\m\ge\m_{j+1}$,
$\Phi_j=(\z_j,\f_j,r_j)$ has an analytic extension to
$\OO^{\g''}(\bar{\s},\r_{r+j},\bar{\m})$ for all $\g''\le\g_{j+1}$
and verifies on this set
$$\left\{\begin{array}{l}
\aa{\z_j-\z}\lsim
(\frac1{\g_j-\g_{j+1}})^{d+m_*}(\frac{\bar{\s}}{\s_j}+1)
\be^{j-1}\frac{\ep_j}{\s_j} \\
\a{\f_j-\f}\lsim \be^{j-1}\frac{\ep_j}\m_j\\
\a{r_j-r}\lsim (\frac1{\r_j-\r_{j+1}})(\frac{\bar{\m}}{\m_j}+
(\frac{\bar\s}\s_j)^2+1)\be^{j-1}\ep.
\end{array}\right.$$

The exponents $\exp_1,\exp_2,\exp_3$ only depend on $d,\#\AA,m_*$
while the constants $\Cte$ and $\cte$
also depend on $C_1,\hdots,C_5$.
\end{Prop}

\begin{proof}
We start by solving inductively
$$\b{h,s_j}=-\TT_{\D'}f_j+h_j,$$
where $\TT_{\D'}f_j$ is the truncation (section \ref{ss81}) and
$s_j$ and $h_j$ are to be found using Propositions \ref{p66} and
\ref{p67}. To see how this works, write
$$
\begin{array}{l}
s_j= S_0+\sc{\z,S_1}+\frac12\!\!\sc{\z,S_2\z}\\
\TT_{\D'}f_j=F_0+\sc{\z,F_1}+\frac12\!\sc{\z,F_2\z}\\
h_j=c_j(\o)+\sc{\chi_j(\o),r}+\frac12\sc{\z,H_j\z}.
\end{array}$$
The equation written in Fourier modes becomes
$$
\begin{array}{l}
-i\sc{k,\o}\hat S_0(k)=-\hat
F_0(k)+\d_0^k(c_j(\o)+\sc{\chi_j(\o),r})\\ -i\sc{k,\o}\hat
S_1(k)+J(\O(\o)+H(\o))\hat S_1(k)
=-\hat F_1(k)\\
-i\sc{k,\o}\hat S_2(k)+(\O(\o)+H(\o))J\hat S_2(k)-\hat
S_2(k)J(\O(\o)+H(\o))\\ \qquad =-\hat F_2(k)+\d_0^k H_j.
\end{array}
$$
Using Propositions \ref{p66} and \ref{p67}
these equations can now be solved for $\o$
in a set $U_j$ with
$$\Leb(U_{j-1}\sm U_j)\le\cte \ep^{\exp}\quad (U_0=U).$$
Indeed with
$$c_j(\o)=\hat F_0(0)\quad\textrm{and}\quad \chi_j(\o)=\hat
F_1(0)$$
the first equation follows from Proposition \ref{p66}(i).
The second equation follows from Proposition \ref{p66}(ii) and the third
from Proposition \ref{p67}.
($H_j$ is not the full mean value $\hat F_2(0)$ but only the part
$\pi\hat F_2(0)$.)

This gives, after summing up the (finite) Fourier series,
$$
\begin{array}{l}
[s_j]_{\left\{\begin{subarray}{l}\L_j+d_{\D}+2\\
U_j\end{subarray}\,\al_j\right\}} \le\cte
(\D'\D)^{\exp}\frac1{\kappa^3}\be^{j-1}\ep_j=\tilde\ep_j\\

[h_j]_{\left\{\begin{subarray}{l}\L_j+d_{\D}+ 2\\
U_j\end{subarray}\,\,\al_j\right\}} \le\cte
(\D'\D)^{\exp}\be^{j-1}\ep_j
\end{array}
$$
If the solutions $s_j$ and $h_j$ were non-real (they are not because
the construction gives real functions) then their real parts would
give real solutions.

In a second step, for $0\le t\le 1$  we estimate
$$
f_j-h_j+\b{h+h_1+\hdots+h_{j-1}+(1-t)h_j+tf_j,s_j}$$
which is equal
$$
(f_j-\TT_{\D'}f_j)+t\b{f_j,s_j}+
\b{h_1+\hdots+h_{j-1}+(1-t)h_j,s_j}
=:g_1+g_2+g_3.$$

According to (\ref{e83}) we have
$$
[g_1]_{\left\{\begin{subarray}{l}\L_{j}+d_{\D}+ 2\\
U_j\end{subarray}\  \tilde\al_{j+1}\right\}} \lsim
A(\al_j,\tilde\al_{j+1},\D')\be^{j-1}\ep_j,$$ where
$$\tilde\al_{j+1}=\left(\begin{array}{cc}\g_j-
\frac{\g_j-\g_{j+1}}2&2\s_{j+1}\\
\r_j-\frac{\r_j-\r_{j+1}}2&2\m_{j+1}.\end{array}\right)$$
By our choice of constants and the assumption on $\ep$ we have
$$A\lsim(\frac{1}{\s^2\k^3}
+(\frac1{\r-\r'})^{\#\AA})\ep
\lsim \frac1{\L_j^{14}}\be\frac{\ep}{\s^2\k^3}.$$

According to (\ref{e84}) we have
$$
[g_2]_{\left\{\begin{subarray}{l}\L_{j}+d_{\D}+ 5\\
U_j\end{subarray}\  \tilde\al_{j+1}\right\}} \lsim
B_j(\D'\D)^{\exp}\frac1{\kappa^3}\be^{2j-2}\ep_j^2,$$ where
$$B_j=B(\g_j-\g_{j+1},\s_j,\r_j-\r_{j+1},\m_j,\L_j).$$
$\be$ takes care of this when $j=1$ and when $j\ge2$ we have the
factor $\frac{\ep_j}{\ep_1}$ that controls everything, and we get
the bound
$$\lsim \frac1{\L_j^{14}}\be^{j}\frac{\ep}{\s^2\k^3}\ep_j.$$

According to (\ref{e844}) we have
$$
[g_3]_{\left\{\begin{subarray}{l}\L_{j}+d_{\D}+ 5\\
U_j\end{subarray}\  \tilde\al_{j+1}\right\}} \lsim \sum_{1\le i\le
n} B_i
(\D'\D)^{\exp}\be^{i-1}\ep_i\cte(\D'\D)^{\exp}
\frac1{\kappa^3}\be^{j-1}\ep_j,$$
where
$$B_i=B(\g_j-\g_{j+1},\s_i,\r_j-\r_{j+1},\m_i,\L_j).$$
The same argument applies again: $\be$ takes care of this when $i=1$
and when $i\ge2$ we have the factor$\frac{\ep_i}{\ep_1}$ that
controls everything. We get as before  the bound
$$\lsim \frac1{\L_j^{14}}\be^{j}\frac{\ep}{\s^2\k^3}\ep_j.$$

In a third step we construct the time-$t$-map, $|t|\le1$,  $\Phi_t$
of the Hamiltonian vector field $J\p s_j$. Condition (\ref{e85}),
$$\tilde\ep_j\lsim
\min((\tilde\r_{j+1}-\r_{j+1})\tilde\m_{j+1},(\tilde\g_{j+1}-
\g_{j+1})^{d+m_*} \tilde\s_{j+1}^2),$$ is fulfilled for all $j$ by
assumption on $\ep$, so
$$\Phi_t:\OO^{\g''}(\s_{j+1},\r_{j+1},\m_{j+1})\to
\OO^{\g''}(\tilde\s_{j+1},\tilde\r_{j+1},\tilde\m_{j+1})$$ for all
$\g''<\g_{j+1}$, and it will verify conditions (\ref{e86}-\ref{e89})
with $\al,\al',\L$ replaced by
$\tilde\al_{j+1},\al_{j+1},\L_j+d_{\D}+2$. Then the time-1-map
$\Phi_t,\ t=1,$  will be our $\Phi_j$ and do what we want --  this
is a well-known relation.

Finally we define
$$f_{j+1}=\int_0^1 (g_1+g_2+g_3)\circ\Phi_t dt.$$
It only remains to verify the estimate for $f_{j+1}$.
Condition (\ref{e810}),
$$\tilde\ep_j\lsim
\min((\tilde\r_{j+1}-\r_{j+1})\tilde\m_{j+1},
(\tilde\g_{j+1}-\g_{j+1})^{d+m_*+1}
\tilde\s_{j+1}^2)\sqrt{\tilde\g_{j+1}-\g_{j+1}},$$ is fulfilled for
all $j$ by assumption on $\ep$, so we get by (\ref{e811})
$$
[f_{j+1}]_{\left\{\begin{subarray}{l}\L_{j+1}\\ U_j\end{subarray}\,
\al_{j+1}\right\}} \lsim\L_{j}^{14}
[g]_{\left\{\begin{subarray}{l}\L_{j}+d_{\D}+ 5\\ U_j\end{subarray}\
\tilde\al_{j+1}\right\}},$$ and we are done.
\end{proof}

\begin{Cor}\label{c62}
There exist a constant $\Cte$ and an exponent $\exp_1$ such that, if
$$\ep\ \le\ \Cte\min(\g-\g',\r-\r',\frac1\D,\frac1\L)^{\exp_1}
\min(\s^2,\m)^{\frac1{1-3\t}}\quad (\t=\frac1{6}),$$ \footnote{The
bound on $\ep$ in Proposition~\ref{p81} is implicit due
$\log(\frac1\ep))$ and depends on $\k$. Here we have an explicit
bound,  but the price for taking $\k$ to be fractional power of
$\ep$ is that the bound must depend on $\max(\s^2,\m)$ to a power
larger than $1$. The choice of $\t$ is only for convenience -- any
$\t<\frac13$ will do.} then there is a subset $U'\sbs U$,
$$\Leb(U\sm U')\le\cte \ep^{\exp_2},$$
such that for all $\o\in U'$ the following hold:
there is an analytic symplectic diffeomorphism
$$\Phi:\OO^{\g''}(\s',\r',\m')\to\OO^{\g''}(\s,\r,\m),
\quad\forall \g''\le\g',$$
and a vector $\o'$ such that
$$(h_{\o'}+f)\circ\Phi=h'+f'$$
with
\begin{itemize}
\item[(i)]
$$h'=\sc{\o,r}+\frac12\!\sc{\z,(\O(\o)+H'(\o))\z}\quad  (\textrm{modulo a
constant}),$$
$H'(\o)$ and $\p_\o H'(\o)$ in $\NN\FF_{\D'}$, and
$$
[h'-h_{\o'}]_{\left\{\begin{subarray}{l}\L'\\ U'\end{subarray}\,\al'
\right\}} \le\cte\ep$$

\item[(ii)]
$$[f']_{\left\{\begin{subarray}{l}\L'\\
U'\end{subarray}\,\,\al'\right\}}
\le \ep'\le e^{-\t(\log(\frac1{\ep}))^2}$$
\end{itemize}
where
$$\begin{array}{l}
\D'=(\log(\frac1\ep))^2\frac1{\min(\g-\g',\r-\r')}, \\
\L'=\cte \max(\L,d_\D^2,(d_{\D'})^2)+\log(\frac1\ep)(d_\D+23)\\
\s'=(\ep')^{\frac13+\t}\s\\
\m'=(\ep')^{\frac23+2\t}\m.
\end{array}$$

Moreover, for $1\ge\bar\s\ge\s'$ and $1\ge\bar\m\ge\m'$,
$\Phi=(\Phi_\z,\Phi_\f,\Phi_r)$ has an analytic
extension to $\OO^{\g''}(\bar{\s},\r',\bar{\m})$ for all $\g''\le\g'$
and verifies on this set
$$\left\{\begin{array}{l}
\aa{\Phi_\z-\z}\le (\frac{\bar{\s}}{\s}+1)\be\frac{\ep}{\s} \\
\a{\Phi_\f-\f}\le \be\frac{\ep}\m\\
\a{\Phi_r-r}\le (\frac{\bar{\m}}{\m}+(\frac{\bar\s}\s)^2+1)\be\ep
\end{array}\right.$$
for some
$$\be\le \cte \max(\frac1{\g-\g'},\frac1{\r-\r'},\L,\D,
\log(\frac1{\ep}))^{\exp_3},$$
and the mapping $\o\mapsto\o'$ verifies
$$\a{\o'-\mathrm{id}}_{\CC^1(U')}\le\cte \frac\ep\m.$$

The exponents $\exp_1,\exp_2,\exp_3$ only depend on
$d,\#\AA,m_*$ while the constants $\Cte$ and $\cte$
also depend on $C_1,\hdots,C_5$.
\end{Cor}

\begin{proof}
Take $\k^3=\ep^{\t}$. Then
$$\be^n\ep_{n+1}=\ep', \quad \s_{n+1}\ge (\ep')^{\frac13+\t},\quad
 \m_{n+1}\ge (\ep')^{\frac23+2\t},$$
and
$$\ep'\le e^{-\t(\log(\frac1{\ep}))^2}$$
if
$$\ep^{1-2\t}\lsim(\frac1\be)^{\frac{1+3\t}{3\t} }\s^2.$$

The result is an immediate consequence of
Proposition \ref{p81} with
$$h_\o'=\sc{\o+\chi(\o),r}+\frac12\!\!\sc{\z,(\O(\o)+H'(\o))\z}.$$

By Proposition \ref{p81}(ii) we get $\a{\chi}_{\CC^1(U')}\le\cte
\frac\ep\m.$ Therefore the image of $U'$ under the mapping
$\o\to\o+\chi(\o)$ covers a subset $U''$ of $U$ of the same
complementary Lebesgue measure, and we can replace $\o+\chi(\o)$  by
$\o$ if we take $\o'=(Id+\chi)^ {-1}(\o)$.
\end{proof}

\subsection{The infinite induction}
\label{ss83}
\

\noindent
Let $h$ and $f$ be as in the previous section with the same
restrictions on the constants $\g,\s,\r,\m$ are $<1$ and
$\D,\L$.

{\it Choice of constants.} We define
$$\begin{array}{ll}
\ep_{j+1}=e^{-\t(\log(\frac1{\ep_j}))^2}\ (\t=\frac1{33}),&
\ep_1=\ep\\
\g_j=(d_{\D_j})^{-1}, & \g_1=\min(d_\D,\g)\\
\s_{j}=\ep_j^{\frac13+\t}\s_{j-1}\ j\ge2 &  \s_1=\s\\
\m_{j}=\ep_j^{\frac23+2\t}\m_{j-1}\ j\ge2 &  \m_1=\m\\
\r_{j}=(\frac12+\frac1{2^j})\r  &\\
\D_{j+1}=(\log(\frac1{\ep_j}))^2\frac1{\min(\g_j,\r_j-\r_{j+1})},
&\D_1=\D\\ \L_{j}=\cte(d_{\D_{j}})^2. &
\end{array}$$
\footnote{The constant in the definition of $\L_j$ is the one in
Proposition~\ref{p67}.}

With this choice of constants we prove

\begin{Lem}\label{l83}
There exist a constant $\Cte'$ and an exponent $\exp'$ such that if
$$\ep\le\Cte'\min(\g,\r,\frac1\D,\frac1\L)^{\exp'}
\min(\s^2,\m)^{\frac1{1-3\t}}
,$$
then for all $j\ge1$
$$\ep_j\le\Cte\min(\g_j-\g_{j+1},\r_j-\r_{j+1},
\frac1{\D_j},\frac1{\L_j})^{\exp}\min(\s_{j}^2,\m_{j})^{\frac1{1-3\t}}
$$ and
$$
\sum_{1\le i\le j}(d_{\D_i})^2\ep_i\le\frac14\min(C_4,C_5,1),$$
where
$\Cte,\exp$ are those of Corollary \ref{c62}.

The exponents $\exp'$ only depend on
$d,\#\AA,m_*$ while the constant $\Cte'$
also depend on $C_1,\hdots,C_5$.
\end{Lem}

\begin{Rem}
Notice that $\D_j$ increases much faster than quadratically at each
step
---
$\D_{j+1}\ge \D_j^{\frac{(d+1)!}2}$ due to its coupling with $\g_j$.
This is the reason why we cannot grant the convergence by a
quadratic iteration but need a much  faster iteration scheme, as the
one provided by Proposition~\ref{p81} and Corollary~\ref{c62}.
\end{Rem}

The proof is an exercise on the theme ``superexponential growth
beats  (almost) everything''.

\begin{Prop}\label{p84} Under the above assumptions,
there exist a constant $\Cte$ and an exponent $\exp$ such that if
$$\ep\ \le\
\Cte\min(\g-\g',\r-\r',\frac1\D,\frac1\L)^{\exp}
\min(\s^2,\m)^{\frac1{1-3\t}},$$ then there is a subset $U'\sbs U$,
$$\Leb(U\sm U')\le\cte \ep^{\exp'},$$
such that for all $\o\in U'$ the following hold:
for all $j\ge1$ there is an analytic symplectic diffeomorphism
$$\Phi_j:\OO^{\g''}(\s_{j+1},\r_{j+1},\m_{j+1})\to\OO^{\g''}(\s_j,\r_j,\m_j),
\quad\forall \g''\le\g_{j+1},$$
and a vector $\o_j$ such that
$$(h_{j-1}+f_j)\circ\Phi_j=h_j+f_{j+1}\quad(h_0=h_{\o_j},\ f_1=f)$$
and satisfying:
\begin{itemize}
\item[(i)]
$$h_j=\sc{\o,r}+\frac12\!\sc{\z,(\O(\o)+H_j(\o))\z}\quad(\textrm{modulo  a
constant}),$$
$H_j(\o)$ and $\p_\o H_j(\o)$ in $\NN\FF_{\D_{j+1}}$, and
$$
[h_{j}-h_{j-1}]_{\left\{\begin{subarray}{l}\L_j\\
U'\end{subarray}\,\al_{j+1} \right\}} \le\cte\ep_j$$

\item[(ii)]
$$[f_{j+1}]_{\left\{\begin{subarray}{l}\L_{j+1}\\ U'
\end{subarray}\,\,\al_{j+1}\right\}}
\le \ep_{j+1}.$$
\end{itemize}

Moreover, $\Phi_j=(\z_j,\f_j,r_j)$ has an analytic
extension to $\OO^{0}(\frac\s2,\frac\r2,\frac{\m}2)$
and verifies on this set
$$\left\{\begin{array}{l}
\aa{\z_j-\z}\le (\frac{\s}{\s_j}+1)\be\frac{\ep_j}{\s_j} \\
\a{\f_j-\f}\le \be_j\frac{\ep_j}{\m_j}\\
\a{r_j-r}\le (\frac{\m}{\m_j}+(\frac{\s}\s_j)^2+1)\be_j\ep_j
\end{array}\right.$$
for some
$$\be_j\le \cte \max(\frac1{\g_j-\g_{j+1}},\frac1{\r_j-\r{j+1}},\L_j,
\D_j,\log(\frac1{\ep_j}))^{\exp_3},$$
and the mapping $\o\mapsto\o_j$ verifies
$$\a{\o_j-\o_{j-1}}_{\CC^1(U')}\le\cte \frac{\ep_j}{\m_j}.$$

The exponents $\exp,\exp'$ only depend on
$d,\#\AA,m_*$ while the constants $\Cte$ and $\cte$
also depend on $C_1,\hdots,C_5$.
\end{Prop}

\begin{proof} The proof is an immediate consequence of
Corollary~\ref{c62} and Lemma \ref{l83}. The first part of the lemma
implies that the smallness assumption in the corollary is fulfilled
for every $j\ge1$, and the second part implies that assumption
$(\ref{e54})+(\ref{e612})$  holds for every $j\ge1$. The remaining
assumptions are only on $\O$.
\end{proof}

Theorem \ref{t71} now follows from  this proposition. Indeed,
$$\o_j\to \o'$$
and we have
$$(h_{\o'}+f)\circ\Phi= \lim_{t\to\infty}(h_{\o_j}+f)
\circ\Phi_1\circ\dots\circ\Phi_j
=\lim_{t\to\infty}(h_j+f_{j+1}),$$
and since the sequence $h_j$ clearly converges on
$\OO^{0}(\frac{\s}2,\frac{\r}2,\frac{\m}2)$,
also $f_j$ converges on this set  --  to a function $f'$.

Moreover, for $\z=r=0$ and
$|\Im\f|<\frac\r2$ we have, as $j\to\infty$,
$$\a{f_j},\ \a{\p_r f_j},\aa{\p_\z f_j}_0\to 0$$
and, by Young's inequality,
$$\aa{\p_\z^2 f_j\hat\z}_0\lsim(\frac1{\g_j})^d\a{\p_\z^2 f_j}_0
\aa{\hat\z}_0\to 0.$$
Therefore
$$\p_\z f'=\p_r f'= \p_\z^2 f'=0\ \text{ for }\ \z=r=0.$$

\bigskip
\bigskip

\section{Appendix A - Some estimates}
\label{a.A}

\begin{Lem}
\label{lA1} Let $f:I=]-1,1[\to \R$ be of class $\CC^n$ and
$$\a{f^{(n)}(t)}\ge 1\quad \forall t\in I.$$
Then, $\forall \ep>0$, the Lebesgue measure of $\b{t\in I:
\a{f(t)}<\ep}$ is
$$\le {\mathrm{cte.}} \ep^{\frac 1n},$$
where the constant only depends on $n$.
\end{Lem}

\begin{proof}
We have $\a{f^{(n)}(t)}\ge \ep^{\frac 0n}$ for all $t\in I$. Since
$$f^{(n-1)}(t)-f^{(n-1)}(t_0)=\int_{t_0}^t f^{(n)}(s)ds,$$
we get that $\a{f^{(n-1)}(t)}\ge \ep^{\frac 1n}$ for all $t$
outside an interval of length $\le2\ep^{\frac 1n}$. By induction we
get that $\a{f^{(n-j)}(t)}\ge \ep^{\frac jn}$ for all $t$ outside
$2^{j-1}$ intervals of length $\le2\ep^{\frac 1n}$. $j=n$ gives
the result.
\end{proof}

\begin{Rem} The same is true if
$$\max_{0\le j\le n}\a{f^{(j)}(t)}\ge1\quad\forall t\in I$$
and $f\in\CC^{n+1}$. In this case the constant will depend on
$\a{f}_{\CC^{n+1}}$.
\end{Rem}

Let $A(t)$ be a real diagonal $N\times N$-matrix with diagonal
components $a_j$ which are $\CC^1$ on $I=]-1,1[$ and
$$a_j'(t)\ge 1\qquad j=1,\dots,N,\ \forall t\in I.$$
Let $B(t)$ be a Hermitian $N\times N$-matrix of class $\CC^1$ on
$I=]-1,1[$ with
$$\aa{B'(t)}\le\frac12\quad\forall t\in I.$$

\begin{Lem}
\label{lA2} The Lebesgue measure of the set
$$\b{t\in I: \min_{\la(t)\in\s(A(t)+B(t))}\a{\la(t)}<\ep}$$
is
$$\le {\mathrm{cte.}} N\ep,$$
where the constant is independent of $N$.
\end{Lem}

\begin{proof}
Assume first that $A(t)+B(t)$ is analytic in $t$.  Then each
eigenvalue $\la(t)$ and its (normalized) eigenvector $v(t)$  are
analytic in $t$, and
$$\la'(t)=\sc{v(t),(A'(t)+B'(t))v(t)}$$
(scalar product in $\C^N$).
Under the assumptions on $A$ and $B$, this is $\ge1-\frac12$.
Lemma~\ref{lA1} applied to each eigenvalue $\la(t)$ gives the
result.

If $B$ is non-analytic we get the same result by analytic
approximation.
\end{proof}

\begin{Prop}
\label{pA3}
$$\aa{(A(t)+B(t))^{-1}}\le \frac 1{\ep}$$
outside a set of $t\in I$ of Lebesgue measure
$$\le {\mathrm{cte.}} N\ep.$$
\end{Prop}

\begin{proof}
The exists an unitary matrix $U(t)$ such that
$$U(t)^*(A(t)+B(t))U(t)\left(\begin{array}{cccc}
\la_1(t)& \ldots& 0 \\
\vdots& \ddots& \vdots \\
0&\ldots & \la_N(t)
\end{array}\right)
$$
Now
$$\aa{(A(t)+B(t))^{-1}}=\max_{0\le j\le N}\a{\frac1{\la_j(t)}}.$$
\end{proof}

\bibliography{BIB}
\bibliographystyle{amsalpha}

\end{document}